\documentclass{article}

\input{macro}

\usepackage[a4paper,top=3.2cm,bottom=2.7cm,left=4.5cm,right=4.5cm, bindingoffset=0mm]{geometry}
\usepackage{a4wide}
\usepackage{microtype}

\title{Pyramids and Extended Metric Measure Spaces}
\author{Nicola Gigli\footnote{SISSA, Trieste, Italy. E-mail: 
ngigli@sissa.it} \and  Kohei Suzuki\footnote{Department of Mathematical Sciences, Durham University, United Kingdom. E-mail: kohei.suzuki@durham.ac.uk} \and Simone Vincini\footnote{Faculty of Mathematics, University of Vienna, Vienna, Austria. E-mail: simone.vincini@univie.ac.at} }
\date{July 29, 2026}

\begin{document}

\maketitle
\begin{abstract} A pyramid is a generalisation of metric measure spaces (mm-spaces) introduced by M.~Gromov (Birkh\"auser 1999) to establish a geometric framework for measure-concentration problems. An extended metric measure space (emm-space) is another generalisation of mm-spaces introduced by Ambrosio--Gigli--Savar\'e (Invent.~Math.~2014) to extend Sobolev calculus and optimal transport theory to a broader extent. 

The goal of this paper is to forge a bridge between these two frameworks.  We prove that every pyramid has a representation by an emm-space through $1$-Lipschitz order. Furthermore, if pyramids are concentrated, this correspondence is unique up to isomorphism. 
This shows, for the first time, that all pyramids can be realised by concrete geometric spaces.

Based on this representation, we introduce a new notion, {\it a concentrated emm-space}, by the compactness of  the space of $1$-Lipschitz functions modulo constant shift.
 We then define the observable distance for concentrated emm-spaces and establish three equivalent characterisations of concentration in terms of pyramids, Lipschitz observables and the observable distance. Furthermore, by developing an fibration approach, we  show the $\Gamma$-$\limsup$ inequality of Cheeger energies as well as the stability of the log-Sobolev inequality and the Poincar\'e inequality under the weak convergence of pyramids associated with emm-spaces. 

Our results provide a new geometric approach 
for studying both the convergence of emm-spaces and concentration-of-measure phenomena across a broad class of infinite-dimensional models that have so far remained largely beyond the reach of existing geometric methods.
Many of the significant examples arise in probability theory, including the Wiener space, the configuration space, Gaussian fields such as massive Gaussian free fields, spatial white noise and massive bi-Laplacian fields. As an application,  the virtually infinite-dimensional Gaussian space introduced by Shioya (EMS Lecture Note 2016) can be represented by the abstract Wiener space.
\end{abstract}
\tableofcontents

\section{Introduction}

\subsection{Background}
\paragraph{Pyramids.} A {\it pyramid} is a generalisation of metric spaces endowed with measure (called metric measure space), which was introduced by M.~Gromov~\cite[Section~3${\frac{1}{2}}$]{Gro06} as an appropriate geometric framework for describing concentration-of-measure phenomena. We briefly recall relevant definitions.
A {\it metric measure space} (or {\it mm-space}, also called {\it Gromov's triple}) is a complete separable metric space $(\X,\sfd)$ equipped with a Borel probability measure $\mm$. Two mm-spaces $\X=(\X, \sfd_\X, \mm_\X)$ and $\Y=(\Y, \sfd_\Y, \mm_\Y)$ are said to be {\it mm-isomorphic} (or simply {\it isomorphic}) if there exists a measure-preserving map from $\X$ to~$\Y$ that is an isometry on the support~$\supp(\mm_\X)$. We say that {\it $\X$ is dominated by $\Y$ in the sense of $1$-Lipschitz order} and write 
$$\X \prec \Y$$ if there exists a measure-preserving map from $\Y$ to $\X$ that is $1$-Lipschitz on the support $\supp(\mm_\X)$. Let $\mathcal X$ denote the collection of isomorphism classes of mm-spaces.
The space $\mathcal X$ carries a natural complete metric $\square$, called the {\it box distance}. The idea of the box distance is to  compare two mm-spaces by choosing measure-preserving parametrisations of both spaces from the unit interval and then comparing the resulting pullback distance functions on the unit square (see~\eqref{e:Box} for the precise definition). 

In order to describe measure-concentration phenomena, Gromov introduced another weaker metric on $\mathcal X$, the {\it observable distance} $\sfd_{\sf conc}(\X, \Y)$ (see~\Cref{d: conc mm}), measuring the difference between observables (i.e., the spaces of \(1\)-Lipschitz functions) on $\X$ and $\Y$ by the Hausdorff distance induced by a distance metrising the convergence-in-probability (a.k.a.~Ky Fan distance). 
However, unlike the box distance, $\sfd_{\sf conc}$ is not a complete distance on $\mathcal X$. Thus, in order to study measure-concentration limits of mm-spaces, it would be more natural to investigate the whole completion $\overline{\mathcal X}^{\sfd_{\sf conc}}$ rather than $\mathcal X$ itself.  
The elements 
$$\overline{\mathcal X}^{\sfd_{\sf conc}} \setminus \mathcal X$$
added with the completion are, however, no longer associated with an actual metric or measure; they exist only as abstract points in the completion and there is, in principle, no intrinsic geometric way to describe these elements. These elements are called {\it ideal metric measure spaces} (\cite[Definition~7.3]{Sh16}).

Gromov then introduced the notion of pyramids, which are described as subsets in $\mathcal X$: a pyramid $\mathcal P$ is a subset of $\mathcal X$ satisfying the following properties:
\begin{enumerate}[label=\textnormal{(P\arabic*)}]
    \item $\mathcal P$ is nonempty and $\square$-closed;
    \item\label{itt: left closure} if $\X \in \mathcal P$ and $\Y \prec \X$, then $\Y \in \mathcal P$;
    \item\label{itt: filtered} if $\X,\Y \in \mathcal P$, then there exists $\Z \in \mathcal P$ such that $\X,\Y \prec \Z$.
\end{enumerate}
Condition~(P1) was added by Shioya~in~\cite{Sh16}, which was not imposed in Gromov's original definition.
We denote by $\Pi$ the set of all pyramids.
Endowed with the Painlev\'e--Kuratowski topology (a.k.a.~the weak topology), $\Pi$ is a compact topological space. Moreover, every mm-space $\X$ naturally determines a pyramid uniquely by
\[
\mathcal P_{\X}\coloneqq\{\Y\in\mathcal X:\Y\prec\X\},
\]
which we call {\it the pyramid associated with the mm-space~$\X$}.
Gromov further proved that this geometric representation by mm-spaces yields a topological embedding (i.e., a homeomorphism on the image)
\[
\mathcal P_\bullet: (\mathcal X, \tau_{\sfd_{\sf conc}}) \hookrightarrow (\Pi, \tau_\Pi)
\qquad
\X \mapsto \mathcal P_{\X}.
\]
In this sense, the space~$(\Pi, \tau_\Pi)$ of pyramids provides a compactification of the space~$(\mathcal X, \tau_{\sfd_{\sf conc}})$ of mm-spaces. Even more importantly, this construction extends to the completion of $\mathcal X$ with respect to~$\sfd_{\sf conc}$:
\[
\mathcal P_\bullet: \overline{\mathcal X}^{\sfd_{\sf conc}} \hookrightarrow \Pi
\qquad
\X \mapsto \mathcal P_{\X}.
\]
Therefore, the space of pyramids serves as a suitable ambient space in which one can discuss concentration and its limit objects. A pyramid  associated with $\X\in\overline{\mathcal X}^{\sfd_{\sf conc}}$ is called  a {\it concentrated pyramid}. Equivalently, a pyramid $\mathcal P$ is called {\it concentrated} if the family of the spaces $\{\mathcal L_1(\X)\}_{\X \in  \mathcal P}$ of $1$-Lipschitz functions modulo constant addition on~the support~$\supp(\mm_\X)$ is precompact with respect to the Gromov-Hausdorff topology, where $\mathcal L_1(\X) \subset \L^0(\X, \mm_\X)/\R$ is endowed with the subspace distance of the (quotient) Ky Fan distance in $\L^0(\X)/\R$ (which will be recalled in \Cref{t:EPM}).

Although pyramids admit an intrinsic description through \textup{(P1)}--\textup{(P3)}, studying them as geometric objects remains challenging, as these descriptions do not explicitly encode any geometric structure.  This naturally leads to a fundamental question: 
\begingroup
\tagsleft
\begin{align} \label{Q} \tag*{(Q)}
\text{do pyramids always have associated geometric representations?}
\end{align}
\endgroup
If one stays in {the category of mm-spaces}, this question is negative because there are pyramids that do not have any representation by mm-space $\X$ in the form $\mathcal P=\mathcal P_{\X}$: one simple example is the pyramid consisting of the whole $\mathcal X$. 
More significant examples will be presented in \Cref{s:AP}.
To answer the question, therefore, we need to extend the scope of {\it geometry}  beyond mm-spaces.

\paragraph{Extended metric measure spaces.} 
From a perspective complementary to that of pyramids, the theory of metric measure spaces has also been enlarged in order to develop theories of Sobolev spaces and optimal transport encompassing a wider range of applications. In this direction, one is led to the notion of an {\it extended metric measure space}. A quadruple $(\X,\tau,\sfd,\mm)$ is called an {\it Polish extended normalised metric measure space (emm-space)} if the following conditions are satisfied:
    \begin{itemize}
        \item the topological space $(\X, \tau)$ is Polish;
        \item the extended distance $\sfd:\X\times\X\to [0, +\infty]$ is complete and $\tau\times \tau$-lower semicontinuous;
        \item the topology~$\tau_\sfd$ induced by $\sfd$ is stronger than or equal to $\tau$;
        \item $\mm$ is a $\tau$-Borel probability measure.
    \end{itemize}
The fundamental difference  from  ordinary metric measure spaces is the decoupling between the topology $\tau$ and the distance~$\sfd$. In general, $\tau$ is {\it strictly weaker} than the topology $\tau_{\sfd}$ induced by $\sfd$; the latter is typically {\it non-separable}; and the measure $\mm$ is not  {\it $\tau_{\sfd}$-Borel}. Due to the decoupling,  many standard tools from the usual metric measure setting are not directly available.
This extension is  motivated by classical infinite-dimensional examples: infinite products of Gaussian spaces, or  abstract Wiener spaces (see \Cref{d:WS}), where the natural product topology (or Banach space topology) is decoupled from the Cameron--Martin distance. Another important class is given by configuration spaces, where the topology induced by the duality with compactly supported continuous functions (called {\it the vague topology}) is decoupled from the $\ell_2$-optimal matching distance (see Section~\ref{subsec:CF}).

\subsection{Overview of our main results and their cross-disciplinary impact}
Our representation theorem (\Cref{t:1}) shows that every pyramid $\mathcal P$ has a geometric representation $\mathcal P=\mathcal P_\X$ by an emm-space~$\X$, and this representation is unique if $\mathcal P$ is concentrated. This provides an affirmative answer to~\ref{Q} and  establishes a systematic bridge between pyramids and emm-spaces, allowing structural properties and numerical invariants of pyramids to be reformulated as geometric properties of emm-spaces. In the present paper, we illustrate this principle through concentration, a notion that had previously been defined only at the level of pyramids. We transfer it to the emm-setting, and we introduce a new concept, {\it concentrated emm-spaces} (\Cref{t:t2}). We then define the observable distance for concentrated emm-spaces and establish three equivalent characterisations of concentration in terms of pyramids, Lipschitz observables and the observable distance (\Cref{t:2} and \Cref{c:t2}). Furthermore, by establishing an fibration approach (\Cref{t:FBR}), we  show the $\Gamma$-$\limsup$ inequality of Cheeger energies (\Cref{t:SMC}) as well as the stability of the log-Sobolev inequality (resp.~the Poincar\'e inequality) under the weak topology~$\tau_{\Pi}$ of pyramids associated with emm-spaces~(\Cref{c:SLSP}). 
We finally exhibit a broad class of genuinely infinite-dimensional examples possessing this property, and we establish their convergence in the observable distance. 

A particularly important feature of this perspective is that it brings into the setting of metric-measure geometry many infinite-dimensional spaces that have long played a central role in probability theory. Classical examples include abstract Wiener spaces, Gaussian product spaces, path spaces carrying Wiener or bridge measures, spatial white noise, Gaussian free fields, and more general Gaussian random distributions (see Sections~\ref{ss:SA} and~\ref{ss:GFF}). Although these objects have been studied extensively through stochastic analysis, Gaussian measure theory, Dirichlet forms, and infinite-dimensional dynamics, their extended metric geometry has remained comparatively less visible. Our theorem places them naturally within the theory of pyramids and provides a geometric language in which their concentration, finite-dimensional approximations, and observable structure can be studied.

The scope of the representation theorem extends beyond concentration. Various properties of pyramids introduced in~\cite{Gro06} as well as numerical invariants of pyramids established in the literature ---including observable diameter and separation distances (\cite{OzawaShioya2015}), and optimal Poincaré and logarithmic Sobolev constants (\cite{esakiInvariantsGromovsPyramids2024}) --- may be interpreted through their emm-space representatives. The representation also makes it natural to study structural properties of pyramids, such as their behaviour under scaling and $\ell_p$-products, including scale invariance and $\ell_p$-idempotence (\cite{esakiInvariantsGromovsPyramids2024}) as well as the phase transition property of pyramids discovered in~\cite{Shi17}. In this way, abstract pyramid invariants and operations can acquire concrete geometric, analytic, and probabilistic interpretations on infinite-dimensional emm-spaces. 

Conversely, the passage from emm-spaces to pyramids opens a natural program for investigating how synthetic curvature lower bounds (curvature-dimension condition)  such as $\CD(K,N)$ and $\RCD(K,N)$, and metric curvature bounds, including Alexandrov lower-curvature and ${\rm CAT}(\kappa)$ upper-curvature conditions, are reflected at the level of pyramids. In particular, it would be interesting to determine whether these structures admit intrinsic pyramid formulations and how they behave under pyramid convergence.

We therefore expect this framework to provide a common platform on which the theories of pyramids, emm-spaces, and probability can be developed in tandem. Properties at the level of pyramids may acquire more transparent geometric realisations, while the rich collection of examples arising from Gaussian processes, random fields, stochastic dynamics, and statistical mechanics may furnish new and nontrivial objects for the theory of pyramids. In this way, the representation theorem not only connects two existing geometric frameworks, but also opens a route by which ideas with a long history in probability theory can enter and enrich the study of high- and infinite-dimensional metric measure geometry.

\subsection{Main results}
The first main result of this paper is to provide a geometric representation of arbitrary pyramids in terms of emm-spaces. Furthermore we prove that if pyramids are concentrated, the representation is unique up to {\it emm-isomorphism} in the sense of~\cite{SuYo25+}.

\paragraph{Pyramid representation.}
For two emm-spaces~$\X$ and $\Y$, we say that $\X$ and $\Y$ are {\it emm-isomorphic} if there exists a measure-preserving map $F: \X \to \Y$ and an~$\mm_\X$-measurable subset~$\tilde \X \subset \X$ such that  $$\mm_\X(\tilde \X)=1, \quad F|_{\tilde \X}: \tilde \X \to \Y \quad \text{is an isometry}.$$ 
Note that $F$ can be non surjective.
We write $\X \cong \Y$ if $\X$ and $\Y$ are emm-isomorphic, which is an equivalence relation (see \cite[Proposition 2.9]{SuYo25+}). We will discuss further details about emm-isomorphism in and after~\Cref{d:EMI}. 

We define $\mathfrak X$ as emm-isomorphism classes of emm-spaces, viz.,
$$\mathfrak X :=\{\text{all emm-spaces}\}/\cong.$$
We note that the class $\mathfrak X$ is a set (\Cref{p:set}).
We write $\Y \prec \X$ if $F$ is measure-preserving and $F|_{\tilde \X}: \tilde \X \to \Y$ is $1$-Lipschitz. 
We define 
$$\mathcal P_\X=\{\Y \in \mathcal X: \Y \prec \X\}.$$
In the generality of emm-spaces, we do not know whether~$\mathcal P_\X$ is $\square$-closed. We therefore consider the $\square$-closure~$\overline{\mathcal P_{\X}}^{\square}$
to discuss pyramids associated with an emm-space~$\X$. 
\begin{theorem}[Pyramid representation by emm-spaces] The following hold. \label{t:1} \
\begin{itemize}
    \item $($\Cref{t:IPT}$)$ For every pyramid $\mathcal P$, there exists an emm-space $\X$ such that 
    $$\mathcal P=\overline{\mathcal P_{\X}}^{\square}.$$

    \item $($\Cref{theorem: isomorphism from pyramid}$)$ If $\mathcal P$ is concentrated,  
    there exists a unique emm-space~$\X$ up to emm-isomorphism such that 
    $$\mathcal P=\mathcal P_{\X}=\overline{\mathcal P_\X}^\square.$$
    Namely, the correspondence $\X \mapsto \mathcal P_{\X}$ is one-to-one.
    \item $($\Cref{e:CEG}$)$ If $\mathcal P$ is not concentrated, the correspondence $\X \mapsto \overline{\mathcal P_{\X}}^{\square}$ is not one-to-one in general. Namely, there exist two emm-spaces $\X_1$ and $\X_2$ that are not emm-isomorphic, but $\overline{\mathcal P_{\X_1}}^{\square}=\overline{\mathcal P_{X_2}}^{\square}$. 
\end{itemize}    
\end{theorem}

\paragraph{Concentrated emm-spaces.}
A natural question would be whether we can characterise exactly when~the pyramid $\mathcal P_\X$ associated with an emm-space~$\X$ becomes concentrated in terms of~observables (functions) on~$\X$. 
To answer this question, we study the space~$\Lip_1(\X, \mm)$ of $\mm$-equivalence classes of $1$-Lipschitz functions on~an emm-space~$\X$. We write $$\mathcal L_1(\X)=\Lip_1(\X, \mm)/\R,$$ where the quotient mods out two functions that coincide by a scalar addition. We stress that the topology~$\tau$ does not play a direct role in the definition and topology of $\mathcal L_1(\X)$.
Let $\L^0(\X)$ be the $\mm$-equivalence classes of measurable functions in~$\X$ endowed with the convergence-in-probability topology. We can regard $\mathcal L_1(\X)$ as a subspace in $\L^0(\X)/\R$, the quotient (by $\R$-addition) of $\L^0(\X)$. 
We endow $\mathcal L_1(\X)$ with the subset topology from $\L^0(\X)/\R$. 

In the following, we characterise the concentration property of the pyramid~$\mathcal P_\X$ by the compactness of~$\mathcal L_1(\X)$ and we define a new concept of {\it concentrated emm-spaces}. 
\begin{theorem}[\Cref{c:ECP}] \label{t:t2}Let $\X \in \mathfrak X$ be an emm-space. The following are equivalent$:$
\begin{itemize}
    \item $\mathcal L_1(\X)$ is compact$;$
    \item $\mathcal P={\overline{\mathcal P_\X}}^\square$ is concentrated. 
\end{itemize}
{\rm In this case, we call $\X$ {\it concentrated emm-space}, and denote by~$\mathfrak X_{\sf conc}$ the emm-isomorphism classes of cocentrated emm-spaces. By the second statement in~\Cref{t:1}, the map $$\mathcal P_{\bullet}: \mathfrak X_{\sf conc} \ni \X \mapsto \mathcal P_{\X} \in \Pi$$ is one-to-one. }
\end{theorem}

\paragraph{Observable distance.}
We can extend the observable distance to~$\mathfrak X_{\sf conc}$: for $\X, \Y \in \mathfrak X_{\sf conc}$, the \emph{observable distance} between $\X$ and $\Y$ is defined as 
\begin{equation}
    \sfd_{\sf conc}(\X, \Y)=\inf_{\iota_\X, \iota_\Y} \sfd^{\L^0}_{\sf H} (\iota_X^*\Lip_1(\X, \mm_\X), \iota_\Y^*\Lip_1(\Y, \mm_\Y)),
\end{equation}
where the infimum runs over all parameters~$\iota_\X, \iota_\Y$~for $\X$ and $\Y$ -- i.e., measure-preserving maps from the unit interval $(0,1)$ endowed with the Lebesgue measure~$\mathsf{Leb}$ to the mm-spaces~$\X$ and $\Y$--, $\iota_\X^*$ (resp.~$\iota_\Y^*$) denotes the pullback of functions in $\X$ (resp.~$\Y$) to $(0,1)$ and $\sfd_{\mathsf H}^{\L^0}$ denotes the Hausdorff distance in the metric space~
$$\Bigl(\L^0((0,1), \mathsf{Leb}), \sfd_{\L^0}\Bigr),$$ where $\sfd_{\L^0}$ is the Ky-Fan distance.
See \Cref{d:Obs} for more details. 
\begin{theorem}[\Cref{t:Vt2}] \label{t:2} The following hold: 
\begin{itemize}
\item $(\mathfrak X_{\sf conc}, \sfd_{\sf conc})$ is complete and separable;
\item $(\mathfrak X_{\sf conc}, \sfd_{\sf conc})$ is a completion of $(\mathcal X, \sfd_{\sf conc})$. The canonical isometric embedding of $\mathcal X$ into~$\mathfrak X_{\sf conc}$ for the completion is the trivial map given by identifying each mm-isomorphism class in~$\mathcal X$ with the associated emm-isomorphism class~in~$\mathfrak X_{\sf conc}$.
\item the map $\mathcal P_\bullet: (\mathfrak X_{\sf conc}, \tau_{\sfd_{\sf conc}}) \to (\Pi, \tau_{\Pi})$ is homeomorphic onto its image.
\end{itemize}
In particular, the following are equivalent for $(\X_n)_{n \in \N}\subset \mathfrak X_{\sf conc}$ and $\X \in \mathfrak X_{\sf conc}$$:$
\begin{enumerate}
\item $\X_n \xrightarrow{\sfd_{\sf conc}} \X$;
\item $\mathcal P_{\X_n} \xrightarrow{\tau_{\Pi}} \mathcal P_\X$.
\end{enumerate}
\end{theorem}
As a corollary we have the following equivalent characterisations of concentrated emm-spaces:
\begin{corollary}[\Cref{c:ECE}] \label{c:t2}Let $\X \in \mathfrak X$ be an emm-space. The following are equivalent$:$
\begin{itemize}
    \item $\mathcal L_1(\X)$ is compact$;$
    \item  $\X \in \overline{\mathcal X}^{\sfd_{\sf conc}}$$;$
    \item $\mathcal P={\overline{\mathcal P_\X}}^\square$ is concentrated. 
\end{itemize}
\end{corollary}
\Cref{c:t2} should be compared to \cite[Theorem~7.25 and Proposition 7.29]{Sh16} (which will be recalled in \Cref{t:EPM}). 
Our novelty here is that we explicitly realise $\mathcal L_1(\X)$ and $P=\overline{\mathcal P_\X}^{\square}$ as $1$-Lipschitz observables and pyramids associated with an emm-space~$\X$.
On the contrary, in \cite{Sh16} both of these objects are obtained via the contraction properties of the maps~$\mathcal  L_1: \mathcal X \to \mathcal H$ and $\mathcal P_\bullet: \mathcal X \to \Pi$ (where $\mathcal H$ is the isometry classes of compact metric spaces) and associated to the abstract elements in the completion.
In particular, the $\mathcal L_1$ spaces defined in \cite[Chapter 7]{Sh16} are just elements of $\mathcal H$, while ours are spaces of (equivalence classes of) functions.

\paragraph{Ideal mm-spaces.}
Thanks to \Cref{t:1}--\ref{t:2}, ideal mm-spaces
$$\partial \mathfrak X_{\sf conc}:=\mathfrak X_{\sf conc} \setminus \mathcal X=\overline{\mathcal X}^{\sfd_{\sf conc}}\setminus \mathcal X$$ are no longer just elements in the abstract completion, rather they are geometrically structured as concentrated emm-spaces. 
Some of the most fundamental questions would be whether $\partial \mathfrak X_{\sf conc}$ is non-empty, and if so, what geometric properties can be possessed by ideal mm-spaces. 
\begin{theorem} \label{t:3}
    The following hold: 
\begin{itemize}
\item $\partial \mathfrak X_{\sf conc}$ is not empty. 
Namely, there is an (indeed many) emm-space~$\X$ such that $\mathcal L_1(\X)$ is compact, but $\X$ is not emm-isomorphic to any mm-space. 
    \item $($\Cref{p:FDP}$)$ If $\X \in \partial \mathfrak X_{\sf conc}$, then there exists no set~$\Z \subset \X$ with $\mm(\Z)=1$ such that $(\Z, \tau_\sfd)$ is separable. 
\end{itemize}
\end{theorem}
In particular, if $\X \in \partial \mathfrak X_{\sf conc}$, then there exists no set~$E \subset \X$ with $\mm(E)=1$ such that the Hausdorff dimension of $E$ is finite.
This tells us that ideal mm-spaces $\X \in \partial \mathfrak X_{\sf conc}$ have to be essentially infinite-dimensional emm-spaces. 
One of the simplest examples belonging to the boundary $\partial \mathfrak X_{\sf conc}$ is the infinite-product of the $n$-dimensional unit sphere with the normalised volume measure
$$\X=\prod_{n=1}^\infty \mathbb S^n(1),$$ see \Cref{e:IPS}. 
This example has already been mentioned, e.g., in~\cite[Example 7.36]{Sh16} as ``considered to be an ideal mm-space". However, to show that it is an element in $\mathfrak X_{\sf conc}\setminus \mathcal X$ rigorously,  one needs to show that it is {\it not  emm-isomorphic to any mm-spaces}, where one should also be  aware that the {emm-isomorphism} is weaker than the ordinary mm-isomorphism.

\paragraph{Dichotomy for Gaussian product.}
We can provide a clear dichotomy for Gaussian product spaces to be concentrated in terms of tail behaviour of variances. 
Let
$\X=\mathbb{R}^{\infty}$ with the product topology $\tau^\infty$ and for $x=(x_i)_{i \in \N}$ and $y=(y_i)_{i \in \N}$, 
\[\ell_2(x,y)\coloneqq\Big(\sum_{n=1}^\infty(x_n-y_n)^2\Big)^{1/2}\in[0,\infty].
\]
Let
$\gamma^\infty := \bigotimes_{n=1}^\infty\gamma_{\sigma_n^2},$
where $\gamma_{\sigma_n^2}$ is  the one-dimensional centred Gaussian measure with variance $\sigma_n^2$. 
\begin{corollary}[\Cref{corollary: ideality gaussian}]\label{corollary-it: ideality gaussian} Let $\X=(\R^\infty, \tau^\infty,\ell_2, \gamma^\infty)$. Then
the following hold:
\begin{enumerate}
    \item $\X \in \mathfrak X_{\sf conc}$  if and only if  $\sigma_n \to 0$ as $n \to \infty$
    \item $\X \in \partial \mathfrak X_{\sf conc}$ if and only if $\sigma_n \to 0$ as $n \to \infty$ and 
    $$\sum_{n=1}^\infty \sigma_n^2=+\infty.$$
\end{enumerate}
{\rm In particular by the first assertion, the abstract Wiener space~$\mathbb W$ (see \Cref{d:WS}), which is emm-isomorphic to the case $\sigma_n=1$,  is not concentrated. }
\end{corollary}

\paragraph{Extrinsic approach.} 
For metric measure spaces, convergence in the concentration~\(
\X_n \xrightarrow{\sfd_{\sf conc}} \X
\)
can be characterised by the existence of maps $p_n:\X_n\to\X$ satisfying suitable approximation properties in terms of distances and measures.
This characterization provides a powerful tool for studying stability phenomena, including the stability of entropies, Cheeger energies, and curvature lower bounds. 
This idea was pioneered by Funano and Shioya (\cite{funanoConcentrationRicciCurvature2013}) who proved the stability of the Curvature-Dimension condition and then followed by Ozawa and Yokota (\cite{ozawaStabilityRCDCondition2019a}) to prove the stability of the Riemannian-Curvature-Dimension condition and by the first and third authors (\cite{gigliStabilityHeatFlow2024}) to prove the Mosco convergence of the Cheeger energy, the stability of the heat flow and of the spectrum of the Laplacian. 

Here we extend this approach in relation to the weak convergence $\tau_{\Pi}$ of pyramids as well as the convergence in~$(\mathfrak X_{\sf conc}, \sfd_{\sf conc})$ based on fibration maps. 
We say that the sequence $n\mapsto \X_n$ {\it fibrates} over $\X$ with projections $n\mapsto p_n:\X_n\to\X$ if the maps $p_n$ are Borel and it holds $(p_n)_\#\mm_n\rightharpoonup \mm$.

Fibrations, introduced in \cite{gigliStabilityHeatFlow2024}, are a framework that allows to define natural notions of convergence of measures or functions defined on varying spaces.
Postponing the precise definitions to \Cref{d:DOC}, we say that projections $n\mapsto p_n$ induce a domination if they allow to well approximate the observables, i.e.~1-Lipschitz functions, in the limit space $\X$ with observables in the varying spaces $\X_n$.
We say that projections $n\mapsto p_n$ induce concentration if they induce a domination and sequences of observables are compact along the sequence of spaces.
Roughly speaking, a sequence of projections induces a domination if the limit space has ``fewer'' 1-Lipschitz functions than the spaces in the sequence; they induce a concentration if $1$-Lipschitz functions in the limit are exactly the limit points of $1$-Lipschitz functions in the sequence of spaces. 
\begin{theorem} \label{t:FBR} The following hold:
\begin{itemize}
    \item (\Cref{proposition: definition via fibration concentrated}) Let $n\mapsto \X_n$ be a sequence of concentrated emm-spaces. Then $n\mapsto \X_n$ converges in concentration to a limit concentrated emm-space $\Y$ if and only if there exists an extended topological-metric-measure space~$\tilde{\Y}$ emm-isomorphic to $\Y$ such that $n\mapsto \X_n$ fibrates over $\tilde\Y$ with projections $n\mapsto p_n$ inducing concentration.
    \item (\Cref{proposition: general projections})     Let $n\mapsto \X_n$ be a sequence of emm-spaces and let $\Y$ be an emm-space.
    Then $$\overline{\mathcal P_\X}^\square\subset K-\liminf_n \overline{\mathcal P_{\X_n}}^\square$$ if and only if there exists an extended topological-metric-measure space $\tilde{\X}$ isomorphic to $\X$ such that $\X_n$ fibrates over $\tilde{\X}$ with projections $p_n$ inducing a domination. 
    \end{itemize}
       The notation $K-\liminf_n \overline{\mathcal P_{\X_n}}^\square$ denotes the inferior limit in the sense of Painlev\'e--Kuratowski, viz.~
       $$K-\liminf_n \overline{\mathcal P_{\X_n}}^\square:=\{\Z\in\mathcal \X: \exists n\mapsto \Z_n\in \overline{\mathcal P_{\X_n}}^\square, \Z_n\xrightarrow{\square}\Z\}.$$
\end{theorem}

We remark that being an extended topological-metric-measure space is a slightly more restrictive condition than being an emm-space, see \Cref{d:ETMS}. 

\paragraph{Stability.} 
Using the fibration approach, we obtain the $\Gamma$-$\limsup$ inequality for the Cheeger energy~$\Ch_\X: \L^2(\X) \to \R_+ \cup \{+\infty\}$ (see~\Cref{d:Che}). 
Notice that the following theorem does not require the convergence in $\sfd_{\sf conc}$. We only requires projections inducing a domination, which is weaker than the weak convergence of pyramids associated with emm-spaces as seen in the second statement in~\Cref{t:FBR}. 
\begin{theorem}[\Cref{theorem: gamma limsup}] \label{t:SMC}
    Let $(\X_n)_{n\in \N}$ be a sequence of emm-spaces and assume that $\X_n$ fibrates over an extended topological-metric-measure space $\X$ with projections $n\mapsto p_n$ inducing a domination.
    Then for all $f\in\L^2(\X)$ there exists a sequence $n\mapsto f_n\in \L^2(\X_n)$ such that $f_n\to f$ in $\L^2$ and
    \begin{equation}
        \ch_\X(f)\geq \limsup_n\ch_{\X_n}(f_n).
    \end{equation}
    In other words, $\Gamma-\limsup_n \ch_{\X_n}\leq \ch_\X$.
\end{theorem}
The $\Gamma$-$\limsup$ inequality for the Cheeger energies implies the stability of first order functional analytic inequalities (see also the discussion after \cite[Theorem 4.4]{zotero-item-281}). 
In particular, we show that under the weak convergence of pyramids the log-Sobolev inequality (resp.~the Poincar\'e inequality) is stable for the associated emm-spaces.
Recall that an emm-space $\X$ supports the {\it log-Sobolev inequality} if 
\begin{equation} \tag*{$\mathsf{LS}(C)$}
    \int_\X f^2 \log f^2 \diff \mm \le 2C\ch_{\X}(f)
\end{equation}
for every~$f$ with $\|f\|_{\L^2}=1$, 
and the {\it Poincar\'e inequality} if 
\begin{equation} \tag*{$\mathsf{P}(C)$}
    \int_\X \Bigl|f-\int_\X f\diff \mm\Bigr|^2 \diff \mm \le 2C\ch_{\X}(f), \qquad f \in \L^2(\X).
\end{equation}
\begin{corollary}[Stability of log-Sobolev and Poincar\'e.~{\Cref{corollary: stability constants}, \Cref{c:SLP}}] \label{c:SLSP}
Let $(\X_n)_{n \in \N}$ be a sequence of emm-spaces supporting  the log-Sobolev inequality~$\LS(C)$ {\rm(}resp.~the Poincar\'e inequality $\P(C)${\rm)}. Suppose that 
$$\overline{\mathcal P_{X_n}}^\square \xrightarrow{\tau_{\Pi}} \mathcal P \in \Pi.$$
Then, there exists an emm-space $\X$ such that $\X$ supports~$\LS(C)$ {\rm (}resp.~$\P(C)${\rm )} and  $$\mathcal P=\overline{\mathcal P_{X}}^\square.$$
Furthermore, any emm-space $X$ representing $\mathcal P=\overline{\mathcal P_{X}}^\square$ supports $\LS(C)$ {\rm (}resp.~$\P(C)${\rm )}. 
\end{corollary}

\subsection{Some applications} \label{ss:SA}
\paragraph{Pyramid representation for the Wiener space.}
Let $X=(\R, |\cdot|, \gamma)$ be the one-dimensional centred Gaussian space with unit variance, and $\X^n$ be the $n$-Cartesian product of $\X$. Consider the pyramid given by $$\overline{\bigcup_{n \in \N} \mathcal P_{\X^n}}^\square.$$ In \cite{Sh16}, this pyramid is formally denoted by $\mathcal P_{\Gamma^\infty}$ and called {\it virtually infinite-dimensional Gaussian space (pyramid)} while the rigorous construction and realisation of $\mathcal P_{\Gamma^\infty}$ as a pyramid associated with emm-spaces have remained open. 

The following \Cref{cI:Int2} resolves this problem: we rigorously identify the virtually infinite-dimensional Gaussian space~$\mathcal P_{\Gamma^\infty}$ to the ($\square$-closure of) the pyramid associated with the abstract Wiener space $\mathbb W$ endowed with the Cameron--Martin extended distance (see \Cref{d:WS}). Let $X^\infty:=(\R^\infty, \tau^\infty, \ell_2, \gamma^\infty)$ be the infinite-product of~$\X$. 
\begin{corollary}[\Cref{c:PWI}] \label{cI:Int2}
 $$\overline{\mathcal P_{\mathbb W}}^\square = \overline{\mathcal P_{\X^\infty}}^\square=\overline{\bigcup_{n \in \N} \mathcal P_{\X^n}}^\square.$$
\end{corollary}

\paragraph{High-dimensional sphere.}
Combined with~\cite[Theorem 1.1]{Shi17}, we can describe the convergence of high-dimensional spheres by explicit emm-spaces. We denote by $\star$ the one-point mm-space (the mm-space consisting of one point). Define 
$$\mathbb I_\infty=([0,1], \tau_{[0,1]}, \sfd_\infty, \mathsf{Leb}),$$
where $\tau_{[0,1]}$ is  the standard Euclidean topology, $\mathsf{Leb}$ is the Lebesgue measure on~$[0,1]$, and 
\begin{equation}
    \sfd_\infty(x,y) =
    \begin{cases}
        0 & \qquad \text{if $x=y$}
        \\
        +\infty & \qquad \text{otherwise}.
    \end{cases}
\end{equation}
Let $\Gamma_{\lambda^2}=(\R^\infty, \tau^\infty,\ell_2, \gamma^\infty_{\lambda^2})$ be the infinite-product centred Gaussian space with variance $\lambda>0$. Let $\mathbb S^n(r_n)$ denote the $n$-dimensional sphere with the standard geodesic distance with the normalised volume measure and radius $r_n$.
We have the following dichotomy. 
\begin{corollary}[\Cref{c:HDS}]
   \begin{equation}
   \mathbb S^n(r_n) \xrightarrow{n \to \infty}
       \begin{cases}
           \star \qquad &\text{in $\sfd_{\sf conc}$} \quad\ \text{if} \quad \frac{r_n}{\sqrt{n}} \to 0
           \\
           \Gamma^\infty_{\lambda^2} \qquad &\text{in $\tau_{\Pi}$} \qquad \text{if} \quad \frac{r_n}{\sqrt{n}} \to \lambda
           \\
           \mathbb I_\infty\qquad &\text{in $\tau_{\Pi}$} \qquad \text{if} \quad \frac{r_n}{\sqrt{n}} \to \infty
       \end{cases}
   \end{equation}
   where by convergence in $\tau_{\Pi}$ we mean the convergence of the associated pyramids. In particular, if $\frac{r_n}{\sqrt{n}} \to 1$, then $\mathbb S^n(r_n)$ converges to the abstract Wiener space:
   $$\mathbb S^n(r_n) \xrightarrow{n \to \infty} \mathbb W \quad \text{in $\tau_{\Pi}$}.$$
\end{corollary}

\paragraph{Pyramid representation for the configuration space.} 
Another example is the configuration space, which is given as an inverse limit of mm-spaces. 
Let $\X \subset \mathbb R$ be a locally compact Polish subspace. Let $\U(\X)$ be the space of locally finite discrete measures on~$\X$, i.e., the totality of $\gamma=\sum_{i=1}^N \delta_{x_i}$ with $N \in \mathbb N_0 \cup \{+\infty\}$, $x_i \in \X$, and $\gamma(K)<+\infty$ for every compact $K \subset \X$.  We endow $\U(X)$ with the topology $\tau_{\sf vague}$ by duality of compactly supported continuous functions in $X$. We simply write $\U$ when $\X=\R$. For $\gamma, \eta \in \U(\X)$, we define the following $\ell_2$-matching extended distance
$$\mssd_{\U(\X)}(\gamma, \eta)^2\coloneqq\inf_{\mssc \in \mathsf{Cpl(\gamma, \eta)}}\int_{\X \times \X} |x-y|^2 \diff \mssc (x, y), \qquad \inf \emptyset =+\infty,$$
where $\mathsf{Cpl}(\gamma, \eta)\coloneqq\{\mathsf c \in \U(\X \times \X): \mathsf c(A \times \X)=\gamma(A),\ \mathsf c(\X \times A)=\eta(A) \}.$
Let $B_n\coloneqq(-n ,n)$, $\bar{B}_n\coloneqq[-n, n]$ and $\partial B_n=\bar B_n \setminus B_n$. Define 
$$\mssd_n(\gamma, \eta)=\inf_{\alpha, \beta \in \U(\partial B_n)} \mssd_{\U(\bar{B}_n)}(\gamma|_{B_n}+\alpha, \eta|_{B_n}+\beta).$$ 
We write $\gamma \overset{\mssd_n}\sim \eta$ if $\mssd_n(\gamma, \eta)=0$. Let $\U_n$ be the quotient metric space of $\U$ with respect to the equivalence relation $\overset{\mssd_n}\sim$ and we denote by $\tilde\mssd_n$ for the quotient distance.  
Let $\pi^{(n)}$ be the Poisson measure in~$\U(\bar B_n)$ with unit intensity (see~\eqref{d:PM}) and $\pi$ be the Poisson measure in~$\U$ obtained by the inverse limit of $\pi^{(n)}$ with bonding map $p_{m,n}: \U(\bar B_n) \to \U(\bar B_m)$ as the restriction $p_{m,n}(\gamma):=\gamma|_{\bar B_m}$ with $m \le n$. 
We may regard $\pi^{(n)}$ as a measure on $\U$ by the inclusion~$\U(\bar B_n) \subset \U$.
Let $\pi_n$ be the quotient measure of $\pi^{(n)}$ on the quotient space~$\U_n$. We denote by $(\bar\U_n, \bar \sfd_n)$ be the completion of $(\U_n, \tilde{\sfd}_n)$ and $\bar \pi_n$ be the push-forward measure to the completion. 
\begin{corollary}[Configuration space] 
Let $\U=(\U, \tau_{\sf vague}, \mssd_\U, \pi)$  and $\U_n=(\bar\U_n, \bar\mssd_n, \bar\pi_n)$ be as above. Then, 
\begin{itemize}
    \item $($\Cref{c:CFP}$)$ $\overline{\mathcal P_{\U}}^\square=\overline{\cup_{n \in \N} \mathcal P_{\bar\U_n}}^\square.$
    \item $($\Cref{p:CNC}$)$ $(\U, \tau_{\sf vague}, \mssd_\U, \pi)$ is not concentrated.
\end{itemize}
\end{corollary}

\paragraph{Gaussian fields.} Here we discuss emm-spaces associated with Gaussian field models, which have a long history in probability theory. 
Let $H$ be a separable real Hilbert space with orthonormal basis
$(e_k)_{k \in \N}$.  Let
\[
1\le \lambda_1\le\lambda_2\le\cdots\lambda_k\longrightarrow\infty,
\]
and assume the Weyl-type growth
\[
\lambda_k\asymp k^{2/d}.
\]
For $t\in\mathbb R$, define
\[
H^t
:=
\left\{
x=\sum_{k=1}^\infty x_ke_k:
\sum_{k=1}^\infty\lambda_k^t x_k^2<\infty
\right\},
\]
with Hilbert norm
$
\|x\|_{H^t}^2
:=
\sum_{k=1}^\infty\lambda_k^t x_k^2.$
Note that for $t \ge 0$, $H^t$ is a subspace of $H$, and for negative~$t<0$, $H^t$ is the completion of $H$ with respect to the norm~$\|\cdot\|_{H^t}$.  Take $r, \alpha \in\mathbb R$ and assume  $\alpha+r >d/2$. Then by $\lambda_k \asymp k^{2/d}$ we have
\begin{equation}\label{eq:support-assumption}
\sum_{k=1}^{\infty}\lambda_k^{-r-\alpha}<\infty,
\end{equation}
and the series
\begin{equation}
\Phi:=\sum_{k=1}^{\infty}\lambda_k^{-\alpha/2}\,\xi_k e_k
\end{equation}
converges
 in $L^2(\Omega;H^{-r})$ and almost surely in $H^{-r}$ (see Subsection~\ref{ss:GFF}), where  $\xi_k: \Omega \to \R$ are independent real-valued centred Gaussian random variables with variance~$1$ defined in a probability space $(\Omega, \mathbb P)$.
Let $\mathfrak m$ be its law (i.e., the push-forward~$\Phi_\#\mathbb P$) on the Hilbert space
\[
\X:=H^{-r}
\]
equipped with its norm topology $\tau_{H^{-r}}$.
By $\Phi \in H^{-r}$ $\mathbb P$-a.s.~as seen above, we have $\mm(H^{-r})=1$. 
Let $\X_s=(H^{-r}, \tau_{H^{-r}}, \|\cdot\|_{H^s}, \mm)$ be the corresponding emm-space for $s \in \R$ and assume $s+r \ge 0$ to have $\tau_{H^{-r}} \subset \tau_{\mssd_s}$. 
\begin{corollary}[\Cref{cor:frac}]\label{corI:frac} Under the above assumption, the following hold:
\begin{enumerate}
 \item $\X_s \in \mathfrak X_{\sf conc}$  if and only if $s<\alpha$;
    \item $\X_s \in \partial \mathfrak X_{\sf conc}$ if and only if $\alpha - \frac{d}{2} \le s <\alpha$;
    \item if $s_n \to s \in (-\infty, \alpha)$ with $s_n + r \ge 0$ for sufficiently large $n$, then 
    $$\X_{s_n} \xrightarrow{\sfd_{\sf conc}} \X_s.$$
\end{enumerate}
\end{corollary}
According to the choice of $\lambda_k$ and parametres $d, \alpha, s, r$, we can associate~$\X_s$ with various Gaussian field models such as path spaces with Brownian motion/Brownian bridge, massive Gaussian free field, spatial white noise and massive membrane (massive bi-Laplacian fields). We refer the readers to \Cref{ss:EGF} and Table~\ref{tab:gaussian-emm-examples} at the end of the paper. 

\subsection{Further remarks}
\paragraph{The role of the topology.}
In an extended metric measure space, the topology plays a somewhat different role from the one it plays in the classical metric measure setting. For an ordinary metric measure space, the metric topology simultaneously determines the Borel structure, the support of the measure, and the underlying metric geometry. 
In the extended setting these roles are separated. 
The topology~\(\tau\) supplies the measurable structure and determines the support of \(\mu\), whereas the extended distance \(\mathsf d\) carries the metric information. 
This separation is the reason why our notion of isomorphism is not required to be topological. 
The topology serves mainly as a framework in which measurability and support are defined; the metric-measure structures studied here depend, up to null sets, only on the measure and the extended distance. 
Consequently, a measure-preserving map that preserves \(\mathsf d\) on a full-measure set is sufficient to transfer the relevant
structures such as Cheeger energies, curvature bounds and related functional inequalities, see~\cite{SuYo25+}. 

That said, there are still a few places in this paper where  a choice of  a representative in emm-isomorphism classes having a nicer topological structure  does matter.  For instance,  in the extrinsic formulation of the convergence in \Cref{subsection: extrinsic convergence}, choosing an emm-representative from an extended topological-metric-measure space (\Cref{d:ETMS}), which has a slightly better topological condition than emm-spaces,  plays a key role. However,  this choice does not change the level of generality for our statements since \cite[Theorem 1.1]{SuYo25+} (recalled in~\Cref{lemma: inverse system}) guarantees that every emm-space has an emm-representative by an inverse limit emm-space, which is, in particular, an extended topological-metric-measure space. Thus, we can always choose an extended topological-metric-measure space as an emm-representative for every emm-space. 

\paragraph{Comparison to other modes of convergences.}
In metric geometry, many notions of convergence for spaces have been extensively studied such as  Lipschitz distances, the measured Gromov--Hausdorff distance, the Gromov--Prokhorov distance, the box distance and related variants. 
These frameworks have proved remarkably powerful in finite-dimensional settings. 
They are, however, not designed to capture families of spaces with genuinely unbounded dimension, nor do they adequately address the extended metric-measure setting considered here. 
Our results are  notably robust enough to accommodate a broad class of infinite-dimensional spaces that arise, for instance, in probability theory.

\paragraph{Related work on pyramid representations.}
A relation between pyramids and metric measure spaces with an extended distance has recently been addressed in~\cite[Theorem 1.6]{miyamoto_direct_2026}, where the terminology~``extended metric measure space~$(\X, \sfd, \mm)$"~there was used in a  more restrictive way than~this paper: $(\X, \sfd)$ is a complete separable extended metric space and $\mm$ is a Borel probability measure on~$(\X, \sfd)$. In particular, the distance and the topology inducing the Borel~$\sigma$-algebra for the measure are not allowed to decouple, and the distance topology has to be separable.   In~\cite[Theorem~1.6]{miyamoto_direct_2026}, a characterisation for pyramids to have geometric representations within this framework has been given in terms of {\it covering numbers of pyramids}, which is a sort of quantitative way to measure the separability of pyramids. For general pyramids, the covering numbers can blow up, hence one cannot obtain geometric representation of these pyramids within the framework of~\cite{miyamoto_direct_2026}. This provides another evidence that the decoupling between the topology and the distance is essential to obtain geometric representations of {\it every} pyramid.

\paragraph{Acknowledgement}
The first author was supported by the Italian Ministry of University and Research (MUR) under the 'Fondo Italiano per la Scienza' (FIS 3) program, PI: Prof. Nicola Gigli, Project Title: Modern challEnges in Geometric Analysis - MEGA, CUP: G53C25000920001
The third author was funded by the Austrian Science Fund (FWF) 10.55776/EFP6. 

\section{Preliminaries} \label{s:Pre}

\subsection{Lipschitz Order for Metric measure spaces}\label{sub: Lipschitz order mm-spaces}
\paragraph{Measure-theoretic preliminary.} Let~$(\X, \mm_\X)$ be a probability space. We simply write $\mm$ instead of $\mm_\X$ if no confusion could occur. 
We denote by $\L^0=\L^0(\X)$ the space of $\mm$-equivalence classes of real-valued $\mm$-measurable functions.  
For $1 \le p <+\infty$,  $\L^p(\X) \subset \L^0(\X)$ denotes the subspace consisting of $p$-absolutely integrable functions $\|f\|_{\L^p}:=(\int_\X |f|^p \diff \mm)^{1/p} <+\infty$. For $p=+\infty$, $\L^\infty(\X)$ denotes the subspace consisting of  $\mm$-essentially bounded functions $\|f\|_{\L^\infty}:=\essup_{\X}f <+\infty$.

The space $\L^0$ becomes a complete and separable metric space when endowed with the $\sfd_{\L^0}$-distance, a.k.a.~\emph{Ky-Fan distance}, defined as
\begin{equation} \label{d:KF}
    \sfd_{\L^0}(f, g)\coloneqq \inf\{\varepsilon>0: \mm(\{|f-g|>\varepsilon\})\leq\varepsilon\}.
\end{equation}
We say that a sequence $(f_n)_{n\in \N}\subset \L^0(\X)$ converges in $\L^0$ to $f$ and write $$f_n\xrightarrow{\L^0} f,$$ when $\sfd_{\L^0}(f_n, f)\to 0$.
Notice that this is equivalent to $f_n\to f$ {\it in probability}, and also equivalent to that every subsequence $k\mapsto f_{n_k}$ has a further subsequence $l\mapsto f_{n_{k_l}}$ converging to $f$~$\mm$-a.e..

For a metric space~$\Z$, we define $\L^0(\X; \Z)$ as the space of $\mm$-equivalence classes of $\Z$-valued $\mm$-measurable functions $T$ that are essentially separably-valued, i.e. there is a full measure set $E\subset \X$ such that $T(E)$ is measurable.
We endow it with the distance $\sfd_{\L^0}$, defined as
\begin{equation}
    \sfd_{\L^0}(S, T)\coloneqq \inf\bigg\{\varepsilon>0: \mm_{\X}\Big(\big\{\sfd_\Z(S, T)>\varepsilon\big\}\Big)<\varepsilon\bigg\}.
\end{equation}
Notice that every equivalence class in $\L^0(\X; \Z)$ contains a separably valued Borel map. 
The space~$\L^0(\X; \Z)$ is separable (respectively, complete) if $\Z$ is separable (respectively, complete).
Similarly to the case~$\L^0(\X)$, we say that a sequence $(T_n)_{n\in \N}\subset \L^0(\X; \Z)$ converges in $\L^0$ to $T$ and write $$T_n\xrightarrow{\L^0} T$$ when $\sfd_{\L^0}(T_n, T)\to 0$.
Notice that this is equivalent to $T_n\to T$ {\it in probability}, and also equivalent to that every subsequence $k\mapsto T_{n_k}$ has a further subsequence $l\mapsto T_{n_{k_l}}$ converging to $T$~$\mm$-a.e..

A real number $m$ is called {\it median of $f \in \L^0(\X)$}  if 
\begin{align} \label{d:median}
\mm(\{f\geq m\})\geq \frac 12,  \qquad \mm(\{f\leq m\})\geq \frac 12.
\end{align}
The set of medians is a compact interval in~$\R$, in particular, we can pick the maximum $m_+$ and the minimum~$m_-$ among the set of medians for $f$. 
The \emph{Lévy mean} of $f$ is defined as 
\begin{align} \label{def: lévy mean}
\mssm_f:=\frac{m_++m_-}{2}.
\end{align}

\paragraph{Weak topology.} Let $\X$ be a Polish space, i.e., a separable topological space whose topology is metrisable by a complete distance. Let $C^0_{\sf b}(\X)$ denote the space of continuous and bounded functions on $\X$. We denote by $\pr(\X)$ the space of Borel probability measures on $\X$.
Given a sequence $n\mapsto\mu_n\in\pr(\X)$, we say that $\mu_n$ weakly converges to $\mu$, denoted by $\mu_n\rightharpoonup\mu\in\pr(\X)$, if
\begin{equation}
    \int_\X f\,\d\mu_n\to\int_\X f\,\d\mu \qquad f\in C^0_{\sf b}(\X).
\end{equation}
The space $\pr(\X)$ endowed with the topology of weak convergence is metrizable and complete.
The Prohorov distance is a distance on $\mathcal P(\X)$ metrising the weak topology defined as follows: fixing a distance $\sfd$ metrising $\X$, for $\mu, \nu\in\pr(\X)$ the Prohorov distance between $\mu$ and $\nu$ is defined as
\begin{equation}
    \sfd_{\sf P}(\mu, \nu)=\inf\bigg\{\varepsilon>0: \forall A\in \mathcal B(\X), \mu\Big(B^{\sfd}(A, \varepsilon)\Big)\geq \nu(A)-\varepsilon\bigg\},
\end{equation}
where $\mathcal B(\X)$ is the Borel~$\sigma$-algebra in $\X$ and $B^{\sfd}(A, \varepsilon)=\{x\in \X:\sfd(x, A)<\varepsilon\}$ is the open $\varepsilon$-neighbourhood of the set $A$. 

Here we record a useful inequality between $\sfd_{\L^0}$ and $\sfd_{\sf P}$ for later use (see e.g.,~\cite[Lemma 1.26]{Sh16}): let $\X$ be a topological space with a Borel probability measure~$\mm$ and $\Y$ be a metric space. For any $\mm$-measurable maps~$f, g: \X \to \Y$, 
\begin{align} \label{in:LP}
   \sfd_{\sf P}(f_\#\mm, g_\#\mm) \le \sfd_{\L^0}(f, g). 
\end{align}

\paragraph{Metric measure space.} A \emph{metric measure space} (simply \emph{mm-space}) is a triplet~$(\X, \sfd_\X, \mm_\X)$ of a complete separable metric space $(\X, \sfd_\X)$ endowed with a locally finite Borel measure $\mm_\X$, namely, $\mm_\X(B)<\infty$ for every bounded Borel set $B \subset X$.
In the following we will always assume that our mm-spaces are normalised, i.e.~$\mm_\X$ is a probability measure. We denote by $\supp(\mm_\X)$ the topological support of $\mm_\X$. For simplicity, we shortly write $(\X, \sfd, \mm)$ instead of $(\X, \sfd_\X, \mm_\X)$ when no confusion could occur.

We denote by $\Lip(\X, \mm)$ the space of Lipschitz functions on $\supp(\mm)$ and by $\Lip_1(\X, \mm)$ the set of~$1$-Lipschitz functions~on~$\supp(\mm)$, namely those functions $f:\X\to\mathbb R$ satisfying
\begin{equation}\label{eq: contraction}
   |f(x)-f(y)|\leq \sfd(x, y) \quad x, y \in \supp(\mm) .
\end{equation}
Notice that convergence in measure for 1-Lipschitz functions coincides with uniform convergence on compact sets in~$\supp(\mm)$.
In particular, $\Lip_1(\X, \mm)$ endowed with $\sfd_{\L^0}$ becomes a complete and separable metric space.

We note that $\mm$-equivalence classes of $\Lip_1(\X, \mm)$, which, by a small abuse of notation we will denote in the same way, is identified with the space of all $\mm$-equivalence classes of functions $f$ such that there exists a set $E$ of full measure on which $f$ is $1$-Lipschitz.

The space $(\Lip_1(\X, \mm), \sfd_{\L^0})$ is not compact as it contains a copy of $\mathbb R$.
That is, though, the only obstacle to compactness: the quotient space 
\begin{equation} \label{d:L1X}
    \mathcal L_1(\X)\coloneqq \Lip_1(\X, \mm)/\mathbb R,
\end{equation}
obtained by modding out translations and endowed with the quotient distance, is a compact metric space, see \cite[Proposition 4.6]{Sh16}.
The compactness of $\mathcal L_1(\X)$ can be characterised in terms of L\'evy mean: a sequence $(f_n)_{n\in \N}\subset\Lip_1(\X, \mm_\X)$ of 1-Lipschitz functions on an mm-space $\X$ such that the Lévy mean of $f_n$ is 0 for all $n\in \N$ is precompact in $\L^0$.

\paragraph{Lipschitz order.}
Given two mm-spaces $\X$ and $\Y$, we write $\Y\prec \X$ if there exists an $\mm_\X$-measurable map $\varphi:\X\to \Y$ satisfying 
$$\varphi_\#\mm_\X=\mm_\Y, \qquad  \sfd_\Y(\varphi(x_1), \varphi(x_2))\leq \sfd_\X(x_1, x_2) \qquad x_1, x_2 \in \supp(\mm_\X). $$
We say that $\X$ and $\Y$ are {\it mm-isomorphic} (or simply {\it isomorphic}) if there exists $\varphi:\X\to \Y$ such that 
$$\varphi_\#\mm_\X=\mm_\Y, \qquad  \sfd_\Y(\varphi(x_1), \varphi(x_2)) = \sfd_\X(x_1, x_2) \qquad x_1, x_2 \in \supp(\mm_\X). $$

Notice that every $\mm_\X$-a.e.~defined map $\psi$ from $\X$ to $\Y$ which is measure preserving and an isometry on a full measure set can be  extended to an isomorphism $\varphi$ as above. 
In other words, two mm-spaces are mm-isomorphic if and only if there exists a set $E \subset \X$ with $\mm_\X(E)=1$ and a map $\psi: E \to Y$ such that 
$$\psi_\#(\mm_\X|_E)=\mm_\Y, \qquad  \sfd_\Y(\psi(x_1), \psi(x_2)) = \sfd_\X(x_1, x_2) \qquad x_1, x_2 \in E. $$

Indeed we  can extend the isometry~$\psi$ from $E \cap  \supp(\mm_\X)$ to the isometry on the  entire~$\supp(\mm_\X)$ uniquely by the density of $E \subset \supp(\mm_\X)$. We use the same notation~$\psi$ for this map. 
We then extend $\psi$ to a map~$\varphi$ defined on the entire space~$\X$, e.g., by choosing one point $y_* \in \Y$ and define 
\begin{equation} \label{e:EMI}
\varphi(x)=
\begin{cases}
y_* & \qquad \text{$x \in \X \setminus \supp(\mm_\X)$,} 
\\
\psi(x) & \qquad \text{$x \in \supp(\mm_\X)$,}
\end{cases}
\end{equation}
which gives a desired mm-isomorphism.

Moreover two mm-spaces $\X$ and $\Y$ are mm-isomorphic if and only if $\X\prec \Y$ and $\Y\prec \X$.
In particular, $\prec$ is a preorder relation, called \emph{Lipschitz order},  on the space $\mathcal X$ of isomorphism classes of mm-spaces (see e.g.~\cite[Section 2]{Sh16}).

\subsection{Convergence of metric measure spaces} \label{subsec:CM}
\paragraph{Measured Gromov--Prokhorov/Gromov's box distance.}
Here we recall two distance functions in the space~$\mathcal X$ of isomorphism classes of mm-spaces. 
{\it The measured Gromov--Prokhorov distance} is defined by the Prokhorov distance  between measures embedded into $\ell^\infty$: for mm-spaces~$\X$ and~$\Y$, 
\begin{equation}
    \sfd_{\sf GP}(\X, \Y)=\inf_{\varphi_\X, \varphi_\Y} \sfd_{\sf P}((\varphi_\X)_\#\mm_\X, (\varphi_\Y)_\#\mm_\Y),
\end{equation}
where $\varphi_\X, \varphi_\Y$ are isometric embeddings in $\ell^\infty$ and $\sfd_{\sf P}$ is the Prokhorov distance.

For a topological space~$\X$ with a Borel probability measure $\mm$, a Borel-measurable map $\iota: [0,1) \to \X$ is called {\it a parameter} if 
$$\iota_\# \mathsf{Leb} = \mm,$$
where $\mathsf{Leb}$ is the Lebesgue measure on $[0,1)$. Every mm-space has a parameter (see, e.g.,~\cite[Lemma~4.2]{Sh16}).
{\it The Gromov's box distance} is defined by comparing distance functions through parameters: for mm-spaces~$\X$ and~$\Y$, 
\begin{align} \label{e:Box}
    \square(\X, \Y)=\inf_{\iota_\X, \iota_\Y} \square(\iota_\X^*\sfd_\X, \iota_\Y^*\sfd_\Y),
\end{align}
where $\iota_\X$ (resp.~$\iota_\Y$) is a parameter, and the right-hand side $\square$ is defined to be the infimum of $\e \ge 0$ satisfying that  there exists a Borel subset $I_0 \subset [0,1)$ such that for every $s, t \in I_0$, 
$$|\iota_\X^*\sfd_\X(s, t)-\iota_\Y^*\sfd_\Y(s,t)| \le \e, \qquad \mathsf{Leb}(I_0) \ge 1-\e.$$
We refer the readers to \cite[Section~4.1]{Sh16} for further details.

These two distances in $\mathcal X$ are  complete and  induce the same separable topology in~$\mathcal X$ (see \cite{Lo13}). 
In particular, given a sequence of mm-spaces~$\X_n$ such that $\square(\X_n, \X) \to 0$ for some $\X\in \mathcal X$, there exists a sequence of isometric embeddings $\varphi_n:\X_n\to \ell^\infty$ and an isometric embedding $\varphi:\X\to \ell^\infty$ such that $(\varphi_n)_\#\mm_{\X_n}\rightharpoonup\varphi_\#\mm_\X$.

The topology induced by the box distance $\square$ is compatible with Lipschitz order $\prec$.
For instance, $\prec$ is a closed relation on $\mathcal X$ endowed with the $\square$-topology, i.e., if $\X_n, \Y_n\to \X, \Y$ in the box topology and $\X_n\prec \Y_n$,  
then 
\begin{align} \label{e:BC}
\X\prec \Y.
\end{align}
In the following, we record further useful results.
\begin{lemma}[{\cite[Lemma 6.10]{Sh16}}]\label{lemma: closure of pyramids} \ 
\begin{itemize}
    \item Let $\Y\prec \X$ be mm-spaces. Let $\X_n\to \X$ in the $\square$-topology. Then there exist mm-spaces $\Y_n\prec \X_n$ such that $\Y_n\to \Y$ in the $\square$-topology.
    \item Let $\X_n, \Y_n, \Z_n$ be sequences of mm-spaces~such that $\X_n, \Y_n\prec \Z_n$ and such that both $\X_n$ and $\Y_n$ are $\square$-precompact. 
    Then there exists a $\square$-precompact sequence of mm-spaces $\tilde{\Z}_n$ such that $\X_n, \Y_n\prec \tilde{\Z}_n\prec \Z_n$.
\end{itemize}
\end{lemma}
The precompactness with respect to the topology induced by the $\square$ distance can be characterised by the preorder relation $\prec$ due to the following result in \cite{KaYo21}: a subset $\mathcal Y\subset \mathcal X$ is $\square$-precompact if and only if it is $\prec$-bounded, i.e.~there exists $\X\in\mathcal X$ such that $\Y\prec \X$ for all $\Y\in\mathcal Y$. 

We record a useful inequality between $\square$ and $\sfd_{\sf P}$: Let $\X$ be a complete separable metric space. For any two Borel probability measures $\mm_1$ and $\mm_2$ on~$\X$, 
\begin{align} \label{in:BP}
    \square\bigl((\X, \mm_1), (\X, \mm_2)\bigr) \le 2 \sfd_{\sf P}(\mm_1, \mm_2).
\end{align}
See~e.g.,~\cite[Proposition 4.12]{Sh16}.

\paragraph{Observable distance.}
 To study families of mm-spaces having a concentration phenomenon, typically with unbounded dimensions, Gromov introduced the {\it observable distance} in~\cite[Section $3\frac{1}{2}$]{Gro06}~by comparing the space $1$-Lipschitz functions on different mm-spaces. 
We recall that the Hausdorff distance between two closed subsets $A, B$ of a metric space $(\Z, \sfd_\Z)$ is defined as $$\sfd_{\sfd_\Z}^{\sf H}(A, B)\coloneqq\inf\{\varepsilon>0: A\subset B^{\sfd_\Z}(B, \varepsilon), B\subset B^{\sfd_\Z}(A, \varepsilon)\}.$$

The observable distance between two mm-spaces~$\X$ and $\Y$ is defined by pulling back $1$-Lipschitz functions on $\X$ and $\Y$ to~$[0,1)$ through parameters and measuring their distance  by the Hausdorff distance on the metric space~$(\L^0, \sfd_{\L^0})$. 
\begin{definition}[Observable distance] \label{d: conc mm}
    Given $\X, \Y$ in $\mathcal X$, 
\begin{equation}\label{eq: conc mm}
        \sfd_{\sf conc}(\X, \Y)\coloneqq \inf_{\iota_\X, \iota_\Y} \sfd_{\L^0}^{\sf H}(\iota_\X^*\Lip_1{(\X, \mm_\X)}, \iota_\Y^*\Lip_1{(\Y, \mm_\Y)}),
    \end{equation}
where $\iota_\X, \iota_\Y$ are parameters of $\X$ and $\Y$ respectively, $\iota_\X^*, \iota_\Y^*$ are the corresponding pullbacks, and  $\sfd_{\L^0}^{\sf H}$ is the Hausdorff distance among closed sets of $(\L^0([0,1), \mathsf{Leb}), \mssd_{\L^0})$. 
\end{definition}

The observable distance~$\mssd_{\sf conc}$ metrises the measure-concentration topology, e.g., $\mathbb S^n(1)$ ($n$-sphere with radius $1$) converges to one-point mm-space as $n \to \infty$ under~$\mssd_{\sf conc}$ (e.g., \cite[Corollary~5.8]{Sh16}).
Note that the topology induced by $\mssd_{\sf conc}$ is strictly weaker than that of $\square$ (e.g., \cite[Proposition 5.5]{Sh16}):
\begin{align} \label{e:BXD}
\sfd_{\sf conc}(\X, \Y) \le \square(\X, \Y) \qquad \X, \Y \in \mathcal X.
\end{align}
The \emph{Lipschitz order} relation $\prec$ is closed also with respect to convergence in concentration (for instance, as a consequence of \cite[Proposition 6.2]{Sh16}). 

There is another equivalent definition of convergence in concentration, which is  relevant to the discussion later in order to extend the measure-concentration topology to a broader setting of extended metric measure spaces.  
The following definition is due to Nakajima (\cite{nakajimaBoxDistanceObservable2022}), based on ideas in Optimal Transport:
given $\X, \Y$ in $\mathcal X$ and a coupling $\ppi\in \adm(\mm_\X, \mm_Y)$, we write
\begin{equation}
    \sfd^\sppi_{\sf conc}(\X, \Y)=\sfd_{\L^0(\sppi)}^{\sf H}(\proj_\X^*\Lip_1{(\X, \mm_\X)}, \proj_\Y^*\Lip_1{(\Y, \mm_\Y)}),
\end{equation}
where $\proj_\X, \proj_\Y$ are the projections from the product space~$\X \times \Y$ to each factors and $\adm(\mm_\X, \mm_Y)\coloneqq \{\ppi\in \pr(\X\times\Y): (\proj_1)_\#\ppi=\mm_\X, (\proj_2)_\#\ppi=\mm_\Y\}$.  
\begin{proposition}[{\cite{nakajimaBoxDistanceObservable2022}}]
    Given $\X, \Y$ in $\mathcal X$, the set $\{\sfd^\pi_{\sf conc}(\X, \Y): \pi\in\adm(\mm_\X, \mm_\Y)\}$ admits a minimum.
    Moreover, 
    \begin{equation}
        \sfd_{\sf conc}(\X, \Y)=\min_{\pi\in\adm(\mm_\X, \mm_\Y)}\sfd^\pi_{\sf conc}(\X, \Y).
    \end{equation}
\end{proposition}

\begin{remark}[Incompleteness of $\sfd_{\sf conc}$] \label{r:IC}
$(\mathcal X, \sfd_{\sf conc})$ is incomplete as a metric space. This means that elements in the metric completion $\overline{\mathcal X}^{\sfd_{\sf conc}}$ do not generally have metric-measure structures.
An example of such incompleteness is e.g., due to~\cite[Example 7.36]{Sh16}: let~$\mathbb S^j(1)$ be the $j$-dimensional unit sphere with intrinsic distance and normalised volume measure.  The sequence $\X_n=\prod_{j=1}^n \mathbb S^j(1)$ is then a $\sfd_{\sf conc}$-Cauchy sequence. In this case, the limit candidate~$\prod_{j=1}^\infty \mathbb S^j(1)$ is not an mm-space because the product distance does not metrise the product topology. 
\end{remark}

\subsection{Pyramids}\label{subsection: pyramids}
In~\cite{Gro06}, Gromov introduced a generalisation of mm-spaces called {\it pyramids}: 
\begin{definition}
    A {\it pyramid} $\mathcal P$ is a subset in $\mathcal X$ such that: 
    \begin{enumerate}[label=\textnormal{(P\arabic*)}]
        \item $\mathcal P$ is non empty and closed under $\square$-topology (simply called $\square$--closed);
        \item\label{it: left closure} if $\X\in \mathcal P$ and $\Y\prec \X$, then $\Y\in \mathcal P$;
        \item\label{it: filtered} if $\X, \Y\in\mathcal P$, then there exists $\Z\in\mathcal P$ such that $\X, \Y\prec \Z$.
    \end{enumerate}
\end{definition}
\begin{lemma}
    Let $\mathcal Q\subset \mathcal X$ be a non-empty set satisfying \ref{it: left closure} and \ref{it: filtered}. Then $\overline{\mathcal Q}^\square$ is a pyramid.
\end{lemma}
\begin{proof}
    It is enough to prove that $\overline{\mathcal Q}^\square$ still satisfies \ref{it: left closure} and \ref{it: filtered}, which follows at once from \Cref{lemma: closure of pyramids} and the fact~\eqref{e:BC} that the $1$-Lipschitz order is preserved by the $\square$-convergence.
\end{proof}

The space of pyramids $\Pi$ can be endowed with the topology of Painlevé-Kuratowski convergence~$\tau_\Pi$ (see \cite{Be93}), which is  also called \emph{weak convergence of pyramids}. Athough the Painlev\'e-Kuratowski convergence is well-defined in any metric space,  we here recall the definition adapted to our specific setting of pyramids.
Given a sequence $(F_n)_{n\in \N}$ of $\square$--closed subsets of $\mathcal X$, the set of limit points is defined as 
\begin{equation}
    \underbar{F}_\infty\coloneqq \{\X\in\mathcal X:\, \exists \X_n\in F_n, \  \X_n\xrightarrow{\square}\X\}.
\end{equation}
Similarly, the set of cluster points is defined as
\begin{equation}
    \bar{F}_\infty\coloneqq \{\X\in\mathcal X:\, \exists n_k \uparrow \infty,\  \exists \X_{n_k}\in F_{n_k}, \ \X_{n_k}\xrightarrow{\square}\X  \}.
\end{equation}
The inclusion~$\underbar{F}_\infty \subset\bar{F}_\infty$ is always true.
The sequence $(F_n)_{n\in \N}$ {\it converges in the Painlevé-Kuratowski sense} to the limit set $F$ if and only if $F=\underbar{F}_\infty=\bar{F}_\infty$.
\begin{theorem}[{\cite{Gro06}, \cite[Theorem 6.22]{Sh16}}]
    The space of pyramids $\Pi$ is compact with respect to the toplogy~$\tau_\Pi$.
    Moreover, $\tau_\Pi$ is metrizable.
\end{theorem}

There is a natural embedding of $\mathcal X$ into $\Pi$ by which we can regard the space~$\Pi$ of pyramids as a compactification of the space $\mathcal X$ of mm-spaces.  
\begin{theorem}[{\cite{Gro06}\cite[Theorem 6.23, 6.25]{Sh16}}]
    The map $\mathcal P_\bullet:(\mathcal X, \sfd_{\sf conc}) \to (\Pi, \tau_\Pi)$ defined as 
    \begin{align}
        \X\mapsto\mathcal P_\X\coloneqq \{\Y\prec \X\}
    \end{align}
    is a topological embedding (homeomorphism on the image), i.e.,  it is injective and $\mathcal X \ni \X_n\concto \X \in \mathcal X$ if and only if $\mathcal P_{\X_n}\xrightarrow{\tau_\Pi} \mathcal P_\X$. 
    Furthermore, there exists a metric $\rho$ inducing $\tau_\Pi$ such that $\mathcal P_\bullet: (\mathcal X, \sfd_{\sf conc})\to (\Pi, \rho)$ is 1-Lipschitz.
\end{theorem}
The last part above about the $1$-Lipschitz property implies that the embedding can be extended to the $\mssd_{\sf conc}$-completion 
\begin{align} \label{d:EPY}
\mathcal P_\bullet:\overline{\mathcal X}^{\sfd_{\sf conc}}\to \Pi,
\end{align}
which is also a topological embedding (\cite{Gro06}\cite[Theorem 7.27]{Sh16}).
\begin{lemma}[{\cite[Definition 7.13, Lemma 7.14]{Sh16}}]\label{lemma: approximation of pyramids} 
    Let $\mathcal P\in \Pi$ be any pyramid. Then there exists a sequence of mm-spaces $(\X_n)_{n\in \N}$ such that 
    \begin{itemize}
        \item  $\X_n \prec \X_{n+1}$ for every $n \in \N$;
        \item $\mathcal P=\overline{\bigcup_n \mathcal P_{\X_n}}^\square$
        or, equivalently,  $\mathcal P_{\X_n}\to \mathcal P$ in $\tau_\Pi$.
     \end{itemize}     
\end{lemma}

Let $\mathcal H$ be the collection of  isometry classes of compact metric spaces. Recalling that $\mathcal L_1(\X) = \Lip_1(\X, \mm)/\R$ was given in \eqref{d:L1X}, we can think of $\mathcal L_1$ as a map from $\mathcal X$ to $\mathcal H$ by the assignment~$X \mapsto \mathcal L_1(X) \in \mathcal H$. 
\begin{definition}[Asymptotic mm-spaces/concentrated pyramids] \label{d:Asy} \ 
\begin{itemize}
\item A sequence of mm-spaces~$(\X_n)_{n\in \N}$ is {\it asymptotic (to a pyramid $\mathcal P$)} if the associated pyramids $\mathcal P_{\X_n}$ converge in $\tau_\Pi$ (to $\mathcal P$).
\item    A pyramid $\mathcal P$ is {\it concentrated} if 
\begin{equation}
    \{\mathcal L_1(\X):\, \X\in \mathcal P\}
\end{equation}
is Gromov-Hausdorff precompact in $\mathcal H$. 
\end{itemize}
\end{definition}

\begin{theorem}[{\cite[Theorem~7.25, Lemma 7.7, Proposition 7.29]{Sh16}}] \label{t:EPM}
    The following hold:
    \begin{itemize}
        \item A pyramid $\mathcal P$ is concentrated if and only if $\mathcal P=\mathcal P_\X$ for $\X \in \overline{\mathcal X}^{\sfd_{\sf conc}}$, i.e., $\mathcal P$ belongs to the image of the map $\mathcal P_\bullet: \overline{\mathcal X}^{\sfd_{\sf conc}}\to \Pi$.
        \item The map $\mathcal L_1: \mathcal X \to \mathcal H$ is 1-Lipschitz with respect to $\sfd_{\sf conc}$ on $\mathcal X$ and the Gromov-Hausdorff distance on $\mathcal H$. In particular, $\mathcal L_1$ can be extended to~$\overline{\mathcal X}^{\sfd_{\sf conc}}$.
        \item The map $\mathcal L_1: \overline{\mathcal X}^{\sfd_{\sf conc}} \to \mathcal H$  is proper, i.e.~the preimage of a compact set is compact. 
    \end{itemize}
\end{theorem}
In particular, we have the following characterisation. 
\begin{corollary}\label{cor: concentrated}
    Let $(\X_n)_{n\in \N}\subset \mathcal X$ be asymptotic to $\mathcal P\in\Pi$. Then 
    $$\mathcal P\in \{\mathcal P_\X: \X \in \overline{\mathcal X}^{\sfd_{\sf conc}}\}$$   if and only if $\{\mathcal L_1(\X_n): n \in \N\}$ is Gromov-Hausdorff precompact.
\end{corollary}

\subsection{Extended metric measure spaces}\label{subsection: emms}
There are infinite-dimensional spaces that do not fall into the category of metric measure spaces such as infinite-product spaces, the abstract Wiener space, the configuration space and general Gaussian fields, which will be seen in \Cref{s:AP} in details.  
A common feature in these examples is the decoupling of topologies and distances, i.e., distances do not generally metrise topologies. Towards an extension of the existing metric measure framework to this setting,  a few notions of {\it extended} metric measure spaces have been studied.
The first one we recall below  was introduced by Savaré (\cite{Sav19}) and later studied in depth by Pasqualetto--Schultz--Taipalus~\cite{PaScTa26} and Pasqualetto--Taipalus~\cite{PaTa25}, building on earlier works of Ambrosio--Erbar--Savaré (\cite{AmbErbSav16}).

\begin{definition}[{\cite{PaScTa26}}] \ \label{d:PST}
\begin{itemize}
    \item 
A function~$\sfd:\X\times \X\to [0, +\infty]$ is an \emph{extended distance} if it is symmetric,  satisfies the triangle inequality and $\mssd(x,y)=0$ iff $x=y$ for every $x, y \in X$; the couple $(\X, \sfd)$ is called an extended metric space.

\item A \emph{pre-extended topological-metric space} $(\X, \tau, \sfd)$ is a triple such that $(\X, \sfd)$ is an extended metric space and $(\X, \tau)$ is a topological space.

\item A \emph{pre-extended topological-metric-measure spaces} is a quadruples $(\X, \tau, \sfd, \mm)$ such that~$(\X, \tau, \sfd)$ is a pre-extended topological-metric space and $\mm$ is a $\tau$-Borel probability measure on~$\X$.
\item An extended distance~$\sfd$ is {\it $\tau$-recovered} 
if 
\begin{equation}
    \sfd(\cdot, \cdot\cdot)= \sup_{f\in \Lip_1(\X, \tau, \sfd)} |f(\cdot)-f(\cdot\cdot)|,
\end{equation}
where $\Lip_1(\X, \tau, \sfd)$ is the space of $\tau$-continuous and $1$-Lipschitz functions under $\mssd$.
\end{itemize}    
\end{definition}
\begin{definition}[{\cite{Sav19, PaScTa26}}] \label{d:ETMS}
    An \emph{extended topological-metric space} is a pre-extended topological-metric space $(\X, \tau, \sfd)$ such that 
    \begin{itemize}
        \item $\mssd$ is $\tau$-recovered;
        \item $\tau$ is Polish and generated by $\Lip_1(\X, \tau, \sfd)$.
    \end{itemize}
    An \emph{extended topological-metric-measure spaces} is a quadruplet $(\X, \tau, \sfd, \mm)$ such that $(\X, \tau, \sfd)$ is an extended topological-metric space and $\mm$ is a $\tau$-Borel probability measure on~$\X$. Note that $\mm$ is Radon because every Borel probability on a Polish space is Radon. 
\end{definition}
A class of examples of emm-spaces playing a key role in this paper is \emph{an inverse limit of metric measure spaces} discussed  in~\cite{SuYo25+}. 
An \emph{inverse system} is a sequence of mm-spaces $\{(\X_n, \sfd_n, \mm_n)\}_{n\in \N}$ along with a sequence of 1-Lipschitz measure preserving maps $p_{n, n+1}:\X_{n+1}\to \X_n$:
\begin{align}
    &\mssd_n(p_{n, n+1}(x), p_{n, n+1}(y)) \le \mssd_{n+1}(x, y) \qquad  x, y \in \supp(\mm),\\
&(p_{n, n+1})_\# \mfm_{\X_{n+1}}=\mfm_{\X_n} \quad n \in \N.
\end{align}
\begin{definition}[Inverse limit]
    The \emph{inverse limit} $\X\coloneqq\varprojlim \X_n$ of an inverse system 
    $$\{(\X_n, \sfd_n, \mm_n)\}_{n\in \N}$$ with $1$-Lipschitz measure-preserving maps $(p_{n, n+1})_{n\in \N}$ is the extended topological-metric-measure space $(\X, \tau, \sfd, \mm)$ defined as
    \begin{equation}
        \X=\bigg\{(x_n)_{n\in \N}\in \prod_{n=1}^\infty \X_n:\; x_n=p_{n, n+1}(x_{n+1})\text{ for all } n\in \N\bigg\},
    \end{equation}
    where $\tau$ is the subset topology induced by the product space, $\sfd$ is the $\ell^\infty$ distance in the product, i.e. $\sfd((x_n)_n, (y_n)_n)=\sup_n \sfd_n(x_n, y_n)$, and $\mm$ is the inverse limit measure given by Kolmogorov's extension theorem. We denote by $\p_n: \X \to \X_n$ the projection attached to the inverse limit given by 
    \begin{align} \label{d:IVS}
    \p_n((x_n)_{n \in \N}):=x_n.
    \end{align}
\end{definition}
\begin{remark}\label{remark: inverse system is regular}
    The inverse limit~$(\X, \tau, \sfd)$ is an extended topological-metric space, which can be readily seen by~observing that $\sfd$ is the supremum of $\sfd_n$ (see~\cite[Proposition~2.6]{SuYo25+}).
    In order to simplify the notation, when working with inverse systems we will omit the specification on the projections $p_{n, n+1}$ when not strictly necessary.
\end{remark}

Another class of extended metric measure spaces was introduced in \cite{AmbGigSav14}, which is more general than extended topological-metric-measure spaces in \Cref{d:ETMS}.
\begin{definition}[Emm-spaces]\label{definition: emm}
    A \emph{Polish extended normalised metric measure space}, in short emm-space, is a quadruplet $(\X, \sfd, \tau, \mm)$ where:
    \begin{itemize}
        \item the topological space $(\X, \tau)$ is Polish;
        \item the extended distance $\sfd:\X\times\X\to [0, +\infty]$ is complete and $\tau\times \tau$-lower semicontinuous;
        \item the topology~$\tau_\sfd$ induced by $\sfd$ is stronger than or equal to $\tau$;
        \item $\mm$ is a $\tau$-Borel probability measure.
    \end{itemize}
    Any emm-space is a pre-extended topological-metric-measure space, but it is not obvious whether the converse is true.
\end{definition}

The focal point in the rest of this subsection is 
to discuss a suitable notion of isomorphism between emm-spaces for our purpose.
We use the concept of isomorphism introduced by the second author and Yokota in~\cite{SuYo25+}, which is different from the one discussed in~\cite{Sav19} that imposes an isomorphism to be a measure-preserving isometry as well as  a homeomorphism.
Our isomorphism does not impose a topological condition directly, but the topology is involved rather indirectly through {\it Borel} measures. 
\begin{definition}[Emm-isomorphism {\cite[Definition 2.7]{SuYo25+}}] \label{d:EMI}
 Two emm-spaces    $\X$ and $\Y$ are {\it emm-isomorphic} (or simply {\it isomorphic}) if 
 \begin{itemize}
     \item there exists an $\mm_\X$-measurable subset $A \subset X$ with $\mfm_\X(A)=1$;
     \item there exists an $\mm_\X/\mm_\Y$-measurable  map $\varphi:\X\to \Y$ such that 
     $$\varphi_\# \mfm_\X = \mfm_\Y, \qquad \mssd_\Y(\varphi(x), \varphi(y)) = \mssd_\X(x, y) \qquad x, y \in A.$$
 \end{itemize} 
 We call such a map $\varphi$ an {\it emm-isomorphism} (or simply {\it isomorphism}) of emm-spaces. We write $\X \cong \Y$ if $\X$ and $\Y$ are emm-isomorphic. We remark that any emm-isomorphism $\varphi$ has an inverse emm-isomorphism $\psi$, i.e.~an emm-isomorphism such that $\psi\circ\phi=id_\X$ $\mm_\X$-a.e.~and $\varphi\circ \psi=id_\Y$ $\mm_\Y$-a.e.. In particular $\cong$ forms an equivalence relation,  see~\cite[Proposition 2.9]{SuYo25+}.
\end{definition}
\begin{remark} \ 
\begin{itemize}
\item (Borel representative) \cite[Definition 2.7]{SuYo25+} imposes~$A$ and $\varphi$ to be Borel, but it is not restrictive since we can always take Borel representatives of $A$ and $\varphi$ in \Cref{d:EMI} up to modifications on a set of measure zero. 
\item (Stable properties under emm-isomorphism) Due to~\cite{SuYo25+}, various properties and objects are invariant under emm-isomorphism including Cheeger energies, Wasserstein spaces, the curvature-dimension condition ($\CD$ and $\RCD$), funtional inequalities (log-Sobolev, Poincar\'e, Talagrand and several other). 
\end{itemize}
\end{remark}

\begin{example}\label{example: isomorphism} \ 
    \begin{itemize}
        \item Let $c>0$ and let $\mathbb I_c$ be the interval $(0, 1)$ endowed with the distance $\sfd_c(x, y)=c(1-\delta_{xy})$. 
        If $\mu, \nu$ are atomless measures on $(0, 1)$, then $(\mathbb I_c, \mu), (\mathbb I_c, \nu)$ are emm-isomorphic. 
        Indeed, it is easy to exhibit a strictly increasing measure preserving map.
        \item  Any infinite-dimensional abstract Wiener space (\Cref{d:WS}) is emm-isomorphic to the infinite product Gaussian emm-space~
        $$(\R^{\infty}, \tau^\infty, \ell_2, \gamma_1^{\otimes \infty}),$$ where $\gamma_1$ is the one-dimensional centred Gaussian with unit variance. 
        See, e.g., \cite[Proposition 8.3]{SuYo25+}. Thus, we may call {\it the} abstract Wiener space as long as its emm-isomorphism class is considered.  
    \end{itemize}
\end{example}
\begin{remark} \ 
\begin{itemize}
    \item Pre-extended topological-metric spaces with a lower semicontinuous distance and a topology weaker than the metric have already appeared in the literature as \emph{topometric spaces} in e.g.~\cite{benyaacovTopometricSpacesPerturbations2008}.

    \item We do not know an example of an emm-space that is not an extended topological-metric-measure space, but we believe that the two classes are distinct (compare to~\Cref{lemma: inverse system}).
    \end{itemize}
\end{remark}

\begin{proposition} \label{p:CEM}
Let $(\X,\sfd_\X,\mm_\X)$ and
$(\Y,\sfd_\Y,\mm_\Y)$  be mm-spaces. If $\X$ and $\Y$ are emm-isomorphic, then they are
mm-isomorphic.
\end{proposition}
\begin{proof}
Take~$f:\X\to \Y$ and $\X_0\subset \X$ such that they witness the extended-mm-space
isomorphism. Since~$\mm_\X(\X_0)=1$, the set
$\X_0$ is dense in $ \supp(\mm_\X)$. Similarly, $f(\X_0)$ is dense in~$\supp(\mm_\Y)$. Thus,  every
nonempty open set $U\subset \supp(\mm_\Y)$ has positive $\mm_\Y$-measure, and
\[
\mm_\X(f^{-1}(U)\cap \X_0)=\mm_\Y(U)>0.
\]
The restriction $f|_{\X_0}$ is an isometry between dense subsets of the
complete metric spaces $\supp(\mm_\X)$ and $\supp(\mm_\Y)$. It therefore extends uniquely to an isometry~$\bar f:\supp(\mm_\X)\to \supp(\mm_\Y)$. Since~$\bar f=f$
$\mm_\X$-a.e.,
\[
\bar f_\#\mm_\X=f_\#\mm_\X=\mm_\Y.
\]
We can extend $\bar f$ from $\supp(\mm_\X)$ to the entire~$\X$ by~\eqref{e:EMI}, which gives a desired mm-isomorphism. 
\end{proof}

\begin{prop} \label{p:set}
Let \(\mathfrak X\) denote the collection of emm-isomorphism classes of emm-spaces. Then \(\mathfrak X\) is a set. 
\end{prop}
\begin{proof}
Every Polish space is Borel
isomorphic to a Borel subset of a fixed standard Borel space, for instance \([0,1]\). Hence every
emm-space admits an Borel-isomorphic representative as
a Borel subset of \([0,1]\).

We write \(2^S\) for the power set of a set \(S\). Observe that the collection of all such
representatives is a set. The possible underlying sets \(B\) belong to \(2^{[0,1]}\). For each such
\(B\), a topology \(\tau_B\) on \(B\) can be regarded as a subset of \(2^B\), hence may be regarded as an element
of
\[
2^{\,2^{[0,1]}}.
\]
Likewise, an extended distance \(\sfd_B:B\times B\to[0,\infty]\) may be regarded as
\(
\bar \sfd_B:[0,1]\times[0,1]\to[0,\infty]
\)
by extending \(\sfd_B\) arbitrarily outside \(B\times B\). Thus all possible extended distances are
contained in 
\[
[0,\infty]^{[0,1]\times[0,1]}.
\]
Finally, a Borel probability measure $\mm_B$ on \(B\) can be regarded as 
\(
\bar\mm_B:2^{[0,1]}\to[0,1],
\)
by extending it arbitrarily to a function on all subsets of \([0,1]\). Therefore all possible data
\((B,\tau_B,\sfd_B,\mm_B)\) are contained in the fixed set
\[
2^{[0,1]}
\times
2^{\,2^{[0,1]}}
\times
[0,\infty]^{[0,1]\times[0,1]}
\times
[0,1]^{\,2^{[0,1]}}.
\]
It is, therefore, a set. Since the quotient of a set by an
equivalence relation is again a set, the collection \(\mathfrak X\) of isomorphism classes is a
set.
\end{proof}

Of fundamental importance for the rest of our arguments is the following result showing that any emm-space is emm-isomorphic to an inverse limit,  in particular, to an extended topological-metric-measure space.
\begin{theorem}[{\cite[Theorem 1.1]{SuYo25+}}]\label{lemma: inverse system}
    Let $\X$ be an emm-space. Then there exists an inverse system $(\X_n)_{n\in \N}$ of mm-spaces such that $\X$ is emm-isomorphic to the inverse limit~$\varprojlim \X_n$.
\end{theorem}

\section{A geometric representation of pyramids}
The main result in this section is a geometric representation theorem of pyramids in terms of emm-spaces. This representation theorem forges a bridge between pyramids and emm-spaces, by which we can define a convergence of emm-spaces by using the weak convergence of pyramids when pyramids have  unique emm-representations. This is indeed the case when pyramids are concentrated as we will see in \Cref{section: concentrated spaces}.
In the proof of the representation theorem, \Cref{lemma: approximation of pyramids} plays a crucial role.

\subsection{Lipschitz order for emm-spaces}
The Lipschitz order recalled in \Cref{sub: Lipschitz order mm-spaces} can be extended to emm-spaces.
\begin{definition} 
Let $\X$ and $\Y$ be emm-spaces. We write 
\begin{itemize}
    \item $\Y \prec \X$ if there exists a measure-preserving map $\varphi: \X \to \Y$ and $E \subset \X$ with $\mm_{\X}(E)=1$ such that
    \begin{equation}
        \sfd_\Y(\varphi(x_1), \varphi(x_2))\leq \sfd_\X(x_1, x_2)\qquad x_1, x_2\in E;
    \end{equation}
    \item  $\Y \prec_c \X$ if, furthermore,  $\varphi$ is $\tau_\X/\tau_\Y$-continuous and $E=\supp(\mm_\X)$.
\end{itemize}
\end{definition}

\begin{definition}
\label{d:PRM}
Let $\X$ be an emm-space. We define
\begin{equation}
    \mathcal P_\X\coloneqq\{\Y\in \mathcal X:\Y\prec \X\}, \qquad 
    \mathcal P_{\X, c}\coloneqq\{\Y\in \mathcal X:\Y\prec_c \X\}.
\end{equation}
\end{definition}
As opposed to the case where $\X$ is an mm-spaces, the sets $\mathcal P_\X$ and $\mathcal P_{\X, c}$ might not be $\square$-closed in general.  
The $\square$-closure is, however, always a pyramid. 
\begin{proposition}
    Let $\X$ be an emm-space. Then the $\square$-closures $\overline{\mathcal P_\X}^\square$ and $\overline{\mathcal P_{\X, c}}^\square$ are pyramids.
\end{proposition}
\begin{proof}
    Both $\overline{\mathcal P_\X}^\square$ and $\overline{\mathcal P_{\X, c}}^\square$ satisfy \cref{it: left closure} and \cref{it: filtered}.
    The claim then follows from \Cref{lemma: closure of pyramids}.
\end{proof}
\begin{remark}
    As mentioned before, the pyramid $\overline{\mathcal P_\X}^\square$ is invariant with respect to isomorphisms, but it disregards the topology of $\X$. In contrast, $\overline{\mathcal P_{\X, c}}^\square$ takes into account the topology, but it is not clear to be invariant with respect to isomorphisms.
\end{remark}

\begin{proposition}\label{proposition: piramidi - enmms}
    For every pyramid $\mathcal P$, there exists an inverse system $(\X_n)_{n \in \mathbb N}$ with inverse limit $\X = \varprojlim X_n$ such that 
    \begin{itemize}
        \item $(\X_n)_{n \in \N} \subset \mathcal P$;
        \item  $\mathcal P \subset \overline{\mathcal P_{\X, c}}^\square \subset \overline{\mathcal P_{\X}}^\square$.
    \end{itemize}
\end{proposition}
\begin{proof}
    By \Cref{lemma: approximation of pyramids} there exists an inverse system $(\X_n)_{n\in \N}\subset \mathcal P$ such that $\mathcal P=\overline{\bigcup_n \mathcal P_{\X_n}}^\square$. 
    Let $\X=\varprojlim \X_n$. 
    By construction, $\X_n\prec_c \X$, thus $\bigcup_n \mathcal P_{\X_n}\subset \mathcal P_{\X, c}\subset \mathcal P_\X$. By taking the closure, 
    \begin{equation}
        \mathcal P  \subset \overline{\mathcal P_{\X, c}}^\square \subset  \overline{\mathcal P_{\X}}^\square.
    \end{equation}
\end{proof}
In \Cref{c:IVI}, we will show that these inclusions of pyramids are actually equalities.
In order to prove that the inverse limit~$\X=\varprojlim X_n$ in the proof of  \Cref{proposition: piramidi - enmms} represents the pyramid as 
$$\mathcal P=\overline{\mathcal P_\X}^\square,$$ 
in the following sections we focus on an approximation of any mm-space~$\Y$ with $\Y\prec \X$ by some $\Y_n\prec \X_n$ in the $\square$-topology. For so doing, we revisit $1$-Lipschitz functions and discuss their approximations in the following subsection.

\subsection{Functions on emm-spaces}
As we have seen in \Cref{s:Pre}, the concentration topology is based on the idea that the behaviour of the observable (i.e., $\Lip_1(\X, \mm_\X)$) of a space~$\X$~determines the behaviour of the space itself. 
In the ordinary mm-space setting, the interplay between $\Lip_1(\X, \mm_\X)$ and the measure~$\mm$ is substantially simpler than in the emm-space setting since the topology is metrised by the distance, thanks to which Lipschitz functions are $\tau_\X$-continuous.
In particular the~$\mm$-equivalence classes of each elements in $\Lip_1(\X, \mm_\X)$ are automatically singletons. 

\smallskip
In this section, we extend the scope to the emm-setting with great attention to the delicate measurability issues caused by the decoupling between topologies and extended distances.   
\begin{definition}[Almost every~Lipschitz functions]
Let $\X$ be an emm-space. 
\begin{itemize}
    \item A  function $f: \X \to \R$ is {\it a.e.~$L$-Lipschitz} if $f$ is $\mm_\X$-measurable and there exists a set $E \subset \X$ with $\mm_\X(E)=1$ and $|f(x_1)-f(x_0)| \le L\sfd_\X(x_1, x_0)$ for $x_0,x_1 \in E$. 
    \item ${\it Lip}_1(\X, \mfm_\X)$ is the space of all a.e.~$1$-Lipschitz functions. 
    \item $\Lip_1(\X, \mm_\X)=\mathit {Lip}_1(\X, \mfm_\X)/\mfm_\X$ is the space of $\mm_\X$-classes of {\it a.e.~$1$-Lipschitz} functions, i.e., we identify any two functions coincide $\mm_\X$-a.e..
\end{itemize}
In particular, notice that $\Lip_1(\X, \mm_\X)\subset \L^0(\X, \mm_\X)$.

Similarly, if $\Y$ is a pre-extended topological-metric space (see \Cref{d:PST}), we say that 
\begin{itemize}
    \item a  function $T: \X \to \Y$ is {\it a.e.~$L$-Lipschitz} with values in $\Y$ if $T$ is $\mm_\X/\mm_\Y$-measurable and there exists a set $E \subset \X$ with $\mm_\X(E)=1$ and $\sfd_\Y(f(x_1),f(x_0)) \le L\sfd_\X(x_1, x_0)$ for $x_0,x_1 \in E$. 
    \item ${\it Lip}_1(\X, \mm_\X;\Y)$ is the space of all a.e.~$1$-Lipschitz functions with values in $\Y$. 
    \item $\Lip_1(\X, \mm_\X; \Y)=\mathit {Lip}_1(\X, \mm_\X;\Y)/\mm_\X\text{-a.e.}$ is the space of $\mm_\X$-classes of {\it a.e.~$1$-Lipschitz} functions with values in $\Y$. 
\end{itemize}
\end{definition}
\begin{remark}
    Notice that in the (rare) cases in which we deal with extended valued 1-Lipschitz functions, we stick to the convention 
    \begin{align}
        |(\pm\infty)-(\pm\infty)|&=0\\
        |(\pm\infty)-(\mp\infty)|&=+\infty.
    \end{align}
    In particular, extended valued functions are considered as functions with values in the extended topological-metric space $\overline{\mathbb R}=([-\infty, +\infty], \tau_{\sf ext}, \sfd_{\sf ext})$, where $\tau_{\sf ext}$ comes by the identification with $[-1, 1]$ via $x\mapsto \frac2\pi\arctan(x)$ and $\sfd_{\sf ext}$ coincides with the euclidean distance on $(-\infty,+\infty)$ and satisfies $\sfd_{\sf ext}(-\infty, +\infty)=+\infty$.
\end{remark}
As in the case of  mm-spaces, we can extend  a Lipschitz function defined in a subset of an emm-space to the entire space as a measurable Lipschitz function whose  Lipschitz constant is  preserved.
\begin{lemma}[McShane extension]\label{lemma: McShane}
    Let $\X$ be an emm-space, $A\subset \X$ be a Borel set and $f:A\to (-\infty, +\infty]$ be a 1-Lipschitz Borel function. 
    Then there exists a universally measurable (extended valued) 1-Lipschitz function $\tilde{f}$ such that $\tilde{f}=f$ on $A$.
    If $f$ is bounded, $\tilde{f}$ can be chosen to have the same upper and lower bounds.
\end{lemma}
\begin{proof}
If $A$ is the empty set, there is nothing to prove. 
Assume that $A$ is not empty and let us extend, without renaming it, $f$ to be $+\infty$ on $\X\setminus A$.
Set 
\begin{equation} \label{de:MCE}
    \tilde f(x)\coloneqq\inf_{a\in A}\bigl(f(a)+\mssd_\X(x,a)\bigr),
\qquad x\in \X,
\end{equation}
with the convention $\inf\varnothing=+\infty$.
Notice moreover that $\tilde f=\inf_{a\in \X}\bigl(f(a)+\mssd_\X(x,a)\bigr)$ for all $x\in \X$.

We first check that $\tilde f$ is well-defined as a map
$\X\to(-\infty,+\infty]$, i.e.\ it never takes the value $-\infty$.
Fix $x\in X$.
If $\sfd_\X(x,a)=+\infty$ for every $a\in A$, then every term
$f(a)+\sfd_\X(x,a)$ is equal to $+\infty$, hence $\tilde f(x)=+\infty$.
Otherwise, choose $a_0\in A$ such that $\mssd_\X(x,a_0)<\infty$.
Whenever $\mssd_\X(x,a)<\infty$, we have $\mssd_\X(a, a_0) \le \mssd_\X(a, x)+\mssd_\X(x, a_0)<+\infty$, hence, using the $1$-Lipschitz inequality~$f(a)\ge f(a_0)-\mssd_\X(a,a_0)$,
\[
f(a)+\mssd_\X(x,a)\ge f(a_0)-\mssd_\X(a,a_0)+\mssd_\X(x,a)\ge f(a_0)-\mssd_\X(x,a_0).
\]
Thus, the value~$f(a)+\mssd_\X(x,a)$ is bounded from below by 
$f(a_0)-\mssd_\X(x,a_0)$ on $\{a \in A: \mssd_\X(x, a)<+\infty\}$, while it takes $+\infty$ on the complement~$\{a \in A: \mssd_\X(x, a)=+\infty\}$.
Therefore $\tilde f(x)>-\infty$ for every $x \in \X$.
So indeed $\tilde f:X\to(-\infty,+\infty]$.

\medskip

\noindent
\emph{$\tilde f$ extends $f$.}
Let $x\in A$.
Taking $a=x$ in \eqref{de:MCE},
\[
\tilde f(x)\le f(x)+\mssd_\X(x,x)=f(x).
\]
Conversely, for every $a\in A$, the $1$-Lipschitz property of $f$ yields
\[
f(x)\le f(a)+\mssd_\X(x,a).
\]
Taking the infimum over $a\in A$ we obtain
$f(x)\le \tilde f(x)$.
Therefore $\tilde f(x)=f(x)$ for every $x\in A$.

\medskip

\noindent
\emph{$\tilde f$ is $1$-Lipschitz on $\X$.}
Fix $x,y\in \X$.
For every $a\in A$, the triangle inequality gives
\[
f(a)+\mssd_\X(x,a)\le f(a)+\mssd_\X(y,a)+\mssd_\X(x,y).
\]
Taking the infimum over $a\in A$ we get
\[
\tilde f(x)\le \tilde f(y)+\mssd_\X(x,y).
\]
Since $x,y$ are arbitrary, $\tilde f$ is $1$-Lipschitz.

\medskip

\noindent
\emph{$\tilde f$ is universally measurable.}
Since $f:\X\to(-\infty,+\infty]$ is Borel and
$\mssd_\X:\X\times \X\to[0,+\infty]$ is $\tau\times\tau$-lower semicontinuous,
the map
\[
(x,a)\longmapsto f(a)+\mssd_\X(x,a),
\qquad (x,a)\in  \X\times \X,
\]
is a Borel map from $\X\times A$ to $(-\infty,+\infty]$.

Fix $\alpha\in\mathbb R$ and consider
\[
E_\alpha\coloneqq\{(x,a)\in X\times \X:\ f(a)+\mssd_\X(x,a)<\alpha\}.
\]
Then $E_\alpha$ is a Borel subset of $\X\times \X$.
By definition of $\tilde f$,
\[
\{x\in \X:\tilde f(x)<\alpha\}=\pi_\X(E_\alpha),
\]
where $\pi_\X:\X\times \X\to \X$ is the projection to the first coordinate.
Since the projection of a Borel subset of a product of Polish spaces is analytic,
$\{x\in \X:\tilde f(x)<\alpha\}$ is an analytic subset of $\X$.
Every analytic subset of a Polish space is universally measurable.
Hence
$\{x\in \X:\tilde f(x)<\alpha\}$ is universally measurable for every $\alpha\in\mathbb R$.
Therefore $\tilde f$ is  universally measurable.

Finally, if $f$ is bounded, $\inf f\vee \tilde{f}\wedge \sup f$ is bounded and satisfies the same properties of $\tilde{f}$.
\end{proof}

\begin{remark}
In~\cite[Lemma 2.1]{LzDSSuz20}, a McShane extension theorem has been shown at the level of extended metric spaces (without measurable/measure structure). \Cref{lemma: McShane} above is an enhancement of it showing that we can take universally measurable McShane extensions, which we believe is a useful result of independent interest.  
\end{remark}
\begin{corollary}
    Let $\X$ be an emm-space.
    Every element $f \in \Lip_1(\X, \mfm_\X)$ has a universally measurable everywhere-defined~$1$-Lipschitz representative $\tilde f$.
\end{corollary}
\begin{remark}
    Notice that a priori the everywhere defined  1-Lipschitz representative might only be $\mm_\X$-a.e.~finite.
\end{remark}

\paragraph{Box-comparison.}
For two-variable measurable functions $\varphi_1, \varphi_2: X\times X \to \R \cup \{\pm \infty\}$, we have (at least) two ways to compare them  from the measure-theoretic viewpoint: one way is $\varphi_1=\varphi_2$ (or $\varphi_1 \le \varphi_2$) $\mm^{\otimes 2}$-a.e.; the other way is that there exists $\E \subset \X$ with $\mm(\E)=1$ such that $\varphi_1=\varphi_2$ (or $\varphi_1 \le \varphi_2$) everywhere on~$E \times E$. The latter case is obviously stronger than the former, and  similar to the idea used for the definition of the~$\square$-distance in~\eqref{e:Box}, so we call it {\it box comparison}.  In terms of identification of extended distances for emm-spaces, we need the stronger (latter) one. Indeed, we have the following pathological example with the weaker one, which is not suitable particularly when extended distances often take $+\infty$  measure-theoretically. 
\begin{example}\label{example: choice of negligible sets}
Consider the Gaussian space $\X=(\R^\infty, \tau^\infty, \ell_2, \gamma^\infty)$ and $\Y=(\R^\infty, \tau^\infty, \sfd_\infty, \gamma^\infty)$, where $\gamma^\infty=\gamma_1^{\otimes \infty}$ is the infinite-tensor of the centred Gaussian probability measure with unit variance, and $\sfd_\infty(x, y)=\infty$ if $x\neq y$ and 0 otherwise.
It is easy to see that both $\X$ and $\Y$ are extended topological-metric-measure spaces.
On one hand, by \Cref{corollary: ideality gaussian}, we have 
$$\gamma^\infty\otimes\gamma^\infty\Bigl(\{\ell_2=+\infty\}\Bigr)=1,$$ thus $\ell_2=\sfd_{\infty}$ holds $\gamma^\infty\otimes \gamma^\infty$-a.e..
On the other hand, we can see that the space of $a.e.$~$1$-Lipschitz functions is distinct between $\X$ and $\Y$:
$$\Lip_1(\X, \gamma^\infty) \neq \Lip_1(\Y, \gamma^\infty).$$
Indeed, the definition of $\sfd_\infty$ immediately implies~$\Lip_1(\Y, \gamma^\infty)=\L^0(\Y)$ while $\Lip_1(\X, \gamma^\infty)\neq \L^0(\X)$. To see the latter, 
for instance, let $\proj_1$ be the projection sending $(x_i)_{i \in \N} \in \R^\infty$ to the first coordinate~$x_1 \in \R$. Then $2\proj_1\in \L^0(\X)\setminus\Lip_1(\X, \gamma^\infty)$. 
This implies that $\X$ and $\Y$ are not emm-isomorphic despite $\ell_2=\sfd_{\infty}$ $\gamma^\infty\otimes \gamma^\infty$-a.e..
By a similar argument above, we can also see $$\overline{\mathcal P_\X}^\square\neq \overline{\mathcal P_\Y}^\square.$$
\end{example}

In the following, based on the stronger (box) comparison, we study pre-orders and equivalence of measurable two variable functions. 
\begin{definition} 
    Let $\X$ be a probability space.
    Let $\varphi, \psi:\X\times \X\to[0, +\infty]$ be measurable functions. 
    We write that $\varphi\leq_\square \psi$ if there exists $E\subset \X$ of full measure such that $\varphi\leq \psi$ on $E\times E$. 
    We also say that $\varphi$ is $\square$-equivalent to $\psi$ if there exists $E$ as above such that $\varphi=\psi$ on $E\times E$. 
    Equivalently, $\varphi$ is $\square$-equivalent to $\psi$ iff $\varphi\leq_\square\psi$ and $\psi\leq_\square\varphi$.
\end{definition}
\begin{remark} \  
\begin{itemize}
\item   The $\leq_\square$ relation is a partial order and $\square$-equivalence is an equivalence relation.
\item Since $\leq_\square$ is a partial order, one can define the $\square$-supremum $\square\text{-}\sup \mathcal F$ of a family $\mathcal F$ of measurable functions on $\X\times \X$ as the $\square$-minimal measurable function $\psi$ on $\X\times \X$ satisfying $\varphi\leq_\square\psi$ for all $\varphi\in \mathcal F$. 
Note that $\square\text{-}\sup \mathcal F$ might not exist in general, but if it  exists,  then it is unique up to $\square$-equivalence. 
\item We use the box notation $\le_\square$ since the idea behind it is similar to the box distance~\eqref{e:Box},  where both compare functions on~$\X \times \X$ on subsets $E \times E$ of large measure. 

\item Let $(\X, \tau, \sfd_1, \mfm)$ and $(\X, \tau, \sfd_2, \mfm)$ be emm-spaces. If $\sfd_1$ and $\sfd_2$ are $\square$-equivalent, then $(\X, \sfd_1, \mfm)$ and $(\X, \sfd_2, \mfm)$ are emm-isomorphic. 
\item  By definition of a.e.~Lipschitz property, whenever $f\in\Lip_1(\X, \mfm_\X)$, it holds
\begin{equation}
    \sfd(\cdot,\cdot\cdot)\geq_\square |f(\cdot)-f(\cdot\cdot)|.
\end{equation}
\end{itemize}
\end{remark}

In the case of mm-spaces, one can recover $\sfd(\cdot, \cdot\cdot)$ as the supremum of $|f(\cdot)-f(\cdot\cdot)|$ when $f$ runs among all 1-Lipschitz functions, or a large enough subset of them. 
Here we show the emm-counterpart in terms of the $\square$-comaprision. 
We start with an intermediate lemma.
\begin{lemma}\label{lemma: sup}
    Let $\X$ be a probability space and let $(g_n)_{n\in \N}$ be a sequence of measurable functions from $\X\times \X$ to $[0, +\infty]$.
    Then the pointwise supremum $\sup_n g_n$ is the $\square$-supremum of the sequence $(g_n)_{n\in \N}$.
\end{lemma}
\begin{proof}
    Let $g=\sup_n g_n$. 
    Clearly $g \ge_\square g_n$. Let us show that it is minimal.
    Let $h$ be any element satisfying $h \ge_\square g_n$ for every $n$.  By definition, there exist $E_n$ of full measure such that $h\geq g_n$ on $E_n\times E_n$. Now, $E=\bigcap_n E_n$ is a full measure set and $h\geq g_n$ on $E\times E$ for all $n\in \N$. This implies $h\geq g$ on $E\times E$ and thus $h\geq_\square g$.
\end{proof}
\begin{corollary}
    Let $\X$ be a probability space. Suppose that  $D\subset \L^0(\X)$ has a countable dense subset  $(g_n)_{n\in \N}$. Let $\mathcal D\coloneqq\{|g(\cdot)-g(\cdot\cdot)|: g\in D\}$.
    Then $f=\sup_n |g_n(\cdot)-g_n(\cdot\cdot)|$  is a unique  $\square$-supremum  $\square\text{-}\sup \;\mathcal D$, where  the uniqueness holds up to $\square$-equivalence.
    In particular, if $\X$ is Polish, $\square\text{-}\sup \;\mathcal D$ always exists for every $D \subset \L^0(\X)$.
\end{corollary}
\begin{proof} 
    By \Cref{lemma: sup}, it is enough to show that $f=\sup_n |g_n(\cdot)-g_n(\cdot\cdot)|$ is a $\square$-upper bound for $\mathcal D$.
    Let $g\in D$ and let $g_{n_k}$ be a subsequence of $g_n$ such that $g_{n_k}\to g$ pointwise on~$E\subset \X$ with $\mm(E)=1$.
    Then $f(\cdot, \cdot\cdot)\geq \limsup_k|g_{n_k}(\cdot)-g_{n_k}(\cdot\cdot)|=|g(\cdot)-g(\cdot\cdot)|$ on $E\times E$, proving that $f$ is a $\square$-upper bound. The latter statement follows by the general fact that $\L^0(\X)$ is separable when $\X$ is Polish. 
\end{proof}

\paragraph{Reconstruction of extended distances.}
Based on the box comparison, we can reconstruct extended metrics by a.e.~$1$-Lipschitz functions. 
\begin{proposition}[Reconstruction]\label{prop: reconstruction}
    Let $\X=(\X, \tau, \sfd, \mm)$ be an emm-space.  
    \begin{itemize}
        \item Set
    $$\mathcal D:=\{|f(\cdot)-f(\cdot\cdot)|\}_{f\in \Lip_1(\X,  \mm_\X)}.$$  Then $\sfd$ is the unique (up to $\square$-equivalence) $\square$-supremum of~$\mathcal D$, viz.,
    $$\sfd=\square\text{-}\sup\mathcal D.$$

    \item Suppose $\X=\varprojlim \X_n$, where $\X_n$ is an inverse system and $\p_n: \X \to \X_n$ is the projection attached to the inverse limit $($recall~\eqref{d:IVS}$)$. Then there exists a sequence $(f_n)_{n\in \N}$ with $f_n\in\Lip_1(\X_n)$ such that $$\sfd= \square\text{-}\sup_n \;|f_n\circ \p_n(\cdot)-f_n\circ \p_n(\cdot\cdot)|\}.$$
    Furthermore, $\sfd$ is the unique $\square$-supremum of the sequence $(\sfd_n)_{n\in \N}$ in the sense of 
    $$\sfd(\cdot,\cdot\cdot)= \square\text{-}\sup_n \;\sfd_n\bigl(\p_n(\cdot), \p_n(\cdot\cdot)\bigr).$$ 
    \end{itemize}
\end{proposition}
\begin{proof} 
    We start by noting that if $\varphi:\X\to \Y$ is an emm-isomorphism, then the pullback map $\varphi^*:\mathsf B(\Y\times \Y)\to \mathsf B(\X \times \X)$ is a bijection and order-preserving with respect to $\leq_\square$, i.e. 
    \begin{equation}
        f\leq_\square g  \quad \text{iff} \quad \varphi^* f\leq_\square \varphi^* g.
    \end{equation}
    Here $\mathsf B(\X \times \X)$ (respectively, $\mathsf B(\Y \times \Y)$) denotes the set of Borel measurable functions on $\X\times \X$ (respectively, $\Y\times \Y$), and $\varphi^*$ is defined as $\varphi^* f(x_1, x_2)=f(\varphi(x_1), \varphi(x_2))$ for all $f\in \mathsf B(\Y \times \Y)$.
    As an order-preserving bijection commutes with taking the supremum, we have that, if $\mathcal F\subset \mathsf B(\Y \times \Y)$, then 
    \begin{equation} \label{e:OPM}
    f=\square\text{-}\sup \mathcal F \iff \varphi^* f=\square\text{-}\sup \varphi^*(\mathcal F).
    \end{equation}
    Recalling that every emm-space~$\X$ is emm-isomorphic to $\varprojlim \X_n$ for some $(\X_n)_{n \in \N}$ due to \Cref{lemma: inverse system}, we may therefore assume without loss of generality that~$\X$ is the inverse limit of an inverse system $\X_n$ with the projection~$\p_n: \X \to \X_n$.
    \Cref{lemma: sup} and the definition of inverse limit yield $$\sfd(\cdot,\cdot\cdot)=\square\text{-}\sup \;\sfd_n(\p_n(\cdot),\p_n(\cdot\cdot)).$$
    Since the spaces $j\mapsto\X_j$ in the inverse system are mm-spaces, in particular, $\X$ is separable. Thus, for each $j$ there is a sequence $(f_{n, j})_{n, j\in \N}$ in $\Lip_1(\X_j, \sfd_j)$ such that $\sfd_j=\sup_n f_{n, j}$ for all $j\in \N$ (see e.g., \cite[Remark 2.1.5  and Lemma A.4]{Sav19}).
    Now, by a diagonal argument and composing the $1$-Lipschitz projection~$\p_n: \X \to \X_n$ to pull-back functions from $\X_n$ to $\X$, we can take a  sequence $(f_n)_{n\in \N}$ with $f_n\in\Lip_1(\X_n, \sfd_{\X_n})$ such that 
    \begin{equation}
        \sfd_j(\p_j(\cdot), \p_j(\cdot\cdot))\leq \square\text{-}\sup_n \;\Bigl|f_n(\p_n(\cdot))-f_n(\p_n(\cdot\cdot))\Bigr|.
    \end{equation}
    Moreover, since $\sfd_j$ increases in $j$ and $f_n\in \Lip_1(\X_n)$,
    \begin{equation}
        \sfd_j(\p_j(\cdot), \p_j(\cdot\cdot))\geq \square\text{-}\sup_{n\leq j} \; \Bigl|f_n(\p_n(\cdot))-f_n(\p_n(\cdot\cdot))\Bigr|.
    \end{equation}
    Putting this together, $$\square\text{-}\sup_j\; \sfd_j(\p_j(\cdot), \p_j(\cdot\cdot))=\square\text{-}\sup_n\; |f_n(\p_n(\cdot))-f_n(\p_n(\cdot\cdot))|,$$ which implies 
    \begin{equation}
        \sfd\geq_\square \square\text{-}\sup\; \mathcal D\geq_\square \square\text{-}\sup_n \;|f_n(\p_n(\cdot))-f_n(\p_n(\cdot\cdot))|=\square\text{-}\sup \;\sfd_n(\p_n(\cdot), \p_n(\cdot\cdot))=\sfd,
    \end{equation} 
    which completes the proof. 
\end{proof}
\begin{corollary}\label{cor: characterisation isomorphisms}
    Let $\X$ be an emm-space. 
    Let $\Phi:\X\to \X$ be an $\mm$-measurable map such that $\Phi_\#\mm=\mm$.
    Then $\Phi$ is an isomorphism if and only if $$\Phi^*\Lip_1(\X, \mm_\X)=\Lip_1(\X, \mm_\X).$$
\end{corollary}
\begin{proof}
    If $\Phi$ is an isomorphism, the thesis is immediate.
    Conversely,  assume that $\Phi$ preserves $\Lip_1(\X, \mm_\X)$. Let $\mathcal D=\{f(\cdot)-f(\cdot\cdot): f \in \Lip_1(\X, \mm_\X)\}$ and $\Phi^*\mathcal D=\{f(\cdot)-f(\cdot\cdot): f \in \Phi^*\Lip_1(\X, \mm_\X)\}$. 
   Using the first statement of \Cref{prop: reconstruction} and the hypothesis~$\Phi^*\Lip_1(\X, \mm_\X)=\Lip_1(\X, \mm_\X)$ yields
    $$\Phi^*\sfd_\X = \square\text{-}\sup\; \Phi^*\mathcal D=\square\text{-}\sup\; \mathcal D=\sfd_\X.$$ 
    Hence, $\Phi$ is a measure-preserving isometry defined almost everywhere, which completes the proof.  
\end{proof}

\paragraph{Kuratowski embedding.} 
Similar to the well-known fact that mm-spaces can be emebedded isometrically in $\ell^\infty$, emm-spaces can be embedded in the extended topological-metric space $(\mathbb R^\infty, \tau^\infty, \ell_\infty)$.
\begin{proposition}[{Embedding, cf.~\cite[Lemma 4.3]{AmbErbSav16} and \cite[Theorem 2.12]{SuYo25+}}]\label{prop: embedding}
    Let $\X$ be an emm-space. Then the following hold:
    \begin{itemize}

            \item there exists an emm-isomorphism~$\iota$ from~$\X$ to $(\mathbb R^\infty, \tau^\infty, \ell_\infty, \iota_\#\mm_\X)$.
           \item  If $\X$ is an extended topological-metric-measure space, the map $\iota$ can be chosen to be everywhere defined and a homeomorphism on the image.

            \item Finally, if $\X=\varprojlim_n \X_n$ is an inverse limit, one can choose $\iota$ with the following properties:
    \begin{enumerate}
        \item $\iota$ is a continuous and closed map defined on the whole $\X$;
        \item there is an increasing sequence of sets of indexes $(I_n)_{n\in \N}$, $I_n\nearrow\mathbb N$, and a sequence $\iota_n:\X_n\to \mathbb R^\infty$ of isometric embeddings such that
        $$\iota_n\circ \sfp_n=\proj_n\circ\iota,$$ where $\sfp_n:\X \to \X_n$ are the projections of the inverse system in~\eqref{d:IVS} and 
        $\proj_n$ is given by
        \begin{equation} \label{d:PJI}
            \Big(\proj_n((x_j)_{j\in \N})\Big)_k=\begin{cases}
                x_k\quad&\text{if }k\in I_n;\\
                0\quad&\text{otherwise.}
            \end{cases}
        \end{equation}
    \end{enumerate}
     \end{itemize}
\end{proposition}
\begin{proof}
    If we prove the statement under the assumption that $\X$ is an extended topological-metric-measure space, then \Cref{lemma: inverse system} and \Cref{remark: inverse system is regular} imply the result also for general emm-spaces.

    {\it The case of extended topological-metric-measure space.}  If $\X$ is an extended topological-metric-measure space, recall that by definition it holds that $\sfd_\X=\sup_{f\in \Lip_1(X, \tau_\X, \sfd_\X)}|f(\cdot)-f(\cdot\cdot)|$ and that the topology is generated by the same family $\Lip_1(X, \tau_\X, \sfd_\X)$.
    Since  the topology is Polish, in particular, separable,  there is a sequence $(f_k)_{k\in \N}\subset \Lip_1(X, \tau_\X, \sfd_\X)$ which generates $\tau_\X$.
    Moreover, as a consequence of the Lindel\"of property, which again follows from Polishness, we can reduce the formula to a sequence:
    \begin{equation}\label{eq: isometric embedding}
        \sfd_\X(\cdot, \cdot\cdot)=\sup_{f\in \Lip_1(X, \tau_\X, \sfd_\X)}|f(\cdot)-f(\cdot\cdot)|=\sup_{k\in \N}|f_k(\cdot)-f_k(\cdot\cdot)|.
    \end{equation}
    See \cite[Remark 2.1.5 and Lemma A.4]{Sav19}. We use the same notation $(f_k)_{k \in \N}$ for both the sequences generating $\tau_\X$ and generating $\sfd_\X$~in~\eqref{eq: isometric embedding} since we may combine these two sequences to obtain a single sequence, which generates $\tau_\X$ and $\sfd_\X$ simultaneously.   
    Now take~$\iota \coloneqq(f_1, f_2, \ldots)$, which is everywhere defined. 
    Since the sequence $k\mapsto f_k$ generates the topology, $\iota$ is a homeomorphism on the image.
    Finally, \eqref{eq: isometric embedding} implies that $\iota$ is an isometry.

    \smallskip
     {\it The case of inverse limit.} Let now $\X=\varprojlim \X_n$ be an inverse limit and let $p_{n, n+1}:\X^{n+1}\to \X^n$ be the bonding maps of the inverse system.
     For $m<n$ let also $p_{m, n}=p_{m, m+1}\circ\ldots\circ p_{n-1, n}:\X_n\to\X_m$ and $p_{n, n}=id$.
     
    By reasoning as above, we can choose sequences $(g_k^n)_{k\in \N}\subset \Lip_1(\X_n)$ such that for each $n\in \N$ the sequence $(g_k^n)_{k\in \N}$ generates the topology of $\X_n$ and $$\sfd_{\X_n}=\sup_{k\in \N}|g^n_k(\cdot)-g^n_k(\cdot\cdot)|.$$

    For $n, k\in \N$, let $f_k^n=g_k^n\circ \sfp_n$.
    By construction, the countable family $(f_k^n)_{n, k\in \N}$ generates the topology of $\X$ and satisfies $$\sfd_{\X}=\sup_{n, k\in \N}|f^n_k(\cdot)-f^n_k(\cdot\cdot)|.$$

    Let $j\mapsto h_j$ be an enumeration of the double sequence~$(f_k^n)_{n, k\in \N}$ and let $I_n=\{j\in\mathbb N: h_j=f_k^m \text{ for some } m\leq n \text{ and } k\}$. 
    Let $\iota\coloneqq (h_1, h_2, \dots)$ and let $\iota_n$ be defined by
    \begin{equation} \label{e:EH}
        (\iota_n(x))_j\coloneqq \begin{cases}
            h_j\circ p_{m, n} (x)\quad&\text{if } j\in I_n \text{ and } h_j=f_k^m,\\
            0\quad&\text{otherwise.}
        \end{cases}
    \end{equation}
    By the same argument as before, $\iota$ and the $\iota_n$ are emm-isomorphisms and continuous. Define the projection~${\rm proj}_n$ associated with $I_n$ as in~\eqref{d:PJI}. 
    Noting~$p_{m, n}\circ\sfp_n=\sfp_m$, the above construction gives~
    \begin{equation} \label{e:CIP}
    \iota_n\circ \sfp_n=\proj_n\circ\iota.
    \end{equation}
    
    Finally, we prove that $\iota$ is closed.
    It is clear that, for each $n$, $\iota_n$ is closed.
    Indeed, the topology on $\X_n$ is induced by the complete distance $\X_n$, thus for all closed $A\subset \X_n$, $\iota_n(A)$ is complete and thus closed.
    Let now $(x_k)_{k\in \N}\subset \X$ be such that $\iota(x_k)\to y$ in $\mathbb R^\infty$.
    By~\eqref{e:CIP}, we have
    \begin{equation}
        \iota_n(\sfp_n(x_k)) = \proj_n(\iota(x_k)) \to \proj_n(y).
    \end{equation}
    Since the $\iota_n$ are closed isometries, there exists $x^{(n)}\in \X_n$ such that $\sfp_n(x_k)\to x^{(n)}$ as $k\to\infty$.
    Moreover, by the fact that $I_n$ is increasing and thus $\proj_n\circ\proj_{n+1}=\proj_{n}$, it follows that $p_{n, n+1}(x^{(n+1)})=x^{(n)}$ for every $n\in \mathbb N$, where the maps $p_{n, n+1}: \X^{n+1} \to \X^n$ are the bonding maps of the inverse system.
    By definition of inverse limit, there exists $x\in\X$ such that $\sfp_n(x)=x^{(n)}$ for all $n\in \N$.
    By definition of the topology in the inverse limit, we conclude that $x_k\to x$ as $k\to\infty$.
\end{proof}

\begin{definition}[Kuratowski embedding] \label{def:KE} \ 
\begin{itemize}
\item     If $\X$ is an emm-space, we call \emph{Kuratowski embedding} a $\mm$-a.e.~defined Borel isometry $\iota$ from $\X$ with values in $(\mathbb R^\infty, \tau^\infty, \ell_\infty)$.
 \item   If $\X$ an extended topological-metric-measure space, we call \emph{continuous Kuratowski embedding} a Kuratowski embedding that is everywhere defined and continuous.

   \item If $\X=\varprojlim \X_n$ is an inverse limit, we call \emph{adapted Kuratowski embedding} an emm-isomorphism~$\iota$ given in the third bullet point in~\Cref{prop: embedding}. 
   \end{itemize}
\end{definition}
\begin{definition}[Cylinder functions]
    We say that a function $f:\mathbb R^\infty\to \mathbb R$ is a cylinder function if there is a finite $N\in \N$ and a function $\tilde f:\mathbb R^N\to \mathbb R$ such that
    \begin{equation}
        f(x_1, \ldots, x_N, \ldots)=\tilde f(x_1, \ldots, x_N).
    \end{equation}
\end{definition}
\begin{corollary}\label{corollary: cylinder functions}
    Let $\X$ be an extended topological-metric-measure space and let $\iota$ be a continuous Kuratowski embedding.
    Then the following property holds:
    \begin{itemize}
        \item Let $\mu_n$ be a sequence of positive Borel measures on $\X$ such that for some $C>0$ the estimate $\mu_n\leq C\mm_\X$ holds for all $n\in \N$.
        Then $\mu_n$ converges weakly if and only if $\iota_\#\mu_n$ converges weakly. 
        Moreover, in this case, 
        $$\lim_n \iota_\#\mu_n=\iota_\#\lim_n \mu_n.$$
    \end{itemize}
    If $\X=\varprojlim_n \X_n$ is an inverse limit and $\iota$ is an adapted Kuratowski embedding, then the following property holds:
    \begin{itemize}
        \item if $g\in \Lip_1(\mathbb R^\infty, \ell_\infty)$ is a cylinder function, then for $n$ large enough there exists $\tilde g_n\in \Lip_1(\X_n, \mm_{\X_n})$ such that
        $$g\circ\iota=\tilde g_n\circ \sfp_n \qquad  \mm_\X\text{-a.e.},$$
        where $\sfp_n:\X\to\X_n$ is the projection given in~\eqref{d:IVS}.
    \end{itemize}  
\end{corollary}
\begin{proof}
    Assume that $\X$ is an extended topological-metric-measure space and $\iota$ a continuous Kuratowski embedding.
    Let $n\mapsto\mu_n$ be a sequence of positive Borel measures on $\X$ and $C>0$ be such that $\mu_n\leq C\mm_\X$ holds for all $n\in \N$.
    If $\mu_n\rightharpoonup\mu$ for some Borel measure $\mu$ on $\X$, then, by continuity of $\iota$, $\iota_\#\mu_n\rightharpoonup\iota_\#\mu$.

    Conversely, assume that for some Borel measure $\nu$ on $\mathbb R^\infty$ it holds $\iota_\#\mu_n\rightharpoonup\nu$.
    By the uniform bound~$\mu_n \le C\mm_\X$, $n\mapsto \mu_n$ is tight: up to subsequences we might assume that $\mu_n\rightharpoonup \mu$.
    By the same argument above  $\nu=\iota_\#\mu$.
    By uniqueness of the limit point, the latter claim also follows.

    Assume now that $\X=\varprojlim_n \X_n$ is an inverse limit and that $\iota$ is an adapted Kuratowski embedding with $n\mapsto \iota_n$ and $n\mapsto \proj_n$ as in \Cref{prop: embedding}.
    Let $g\in \Lip_1(\mathbb R^\infty, \ell_\infty)$ be a cylinder function.
    
    By definition of adapted Kuratowski embedding, the sets $I_n$ invade $\mathbb N$: by definition of cylinder function, for $n$ large enough $g=g\circ\proj_n$.
    In particular, again by the definition of adapted Kuratowski embedding (\Cref{def:KE}),
    \begin{equation}
        g\circ\iota=g\circ\proj_n\circ\iota=g\circ\iota_n\circ \sfp_n.
    \end{equation}
    Since $\tilde g_n \coloneqq g\circ\iota_n$ is 1-Lipschitz, the second item is proved.
\end{proof}

\paragraph{Approximation results.}
We collect some approximation result that will be crucial for the representation of pyramids through emm-spaces.
\begin{lemma}[{\cite[Corollary 3.3]{PaTa25}, see also \cite{Mat00}}]\label{lemma: continuous extension}
    Let $\X$ be an extended topological-metric space. 
    Let $K\subset \X$ be a compact set and let $f:K\to \mathbb R$ be 1-Lipschitz and continuous. 
    Then there exists $\tilde{f}\in \Lip_1(X, \tau, \sfd)$ and such that $\tilde{f}_{\restr K}=f$.
\end{lemma}
By coupling \Cref{lemma: continuous extension} with Luzin's theorem, the following density result follows.
\begin{corollary}\label{cor: density of continuous functions}
    Let $\X$ be an extended topological-metric-measure space. 
    Then $\Lip_1(\X, \tau, \sfd)$ is dense in $\Lip_1(\X, \mm)$ with respect to $\L^0$ convergence.
\end{corollary}
Finally, in the case of inverse limits, general measurable 1-Lipschitz functions can be approximated by finite dimensional 1-Lipschitz functions.
\begin{proposition}[Lipschitz approximation]\label{proposition: LAP}
    Let $\X=\varprojlim \X_n$ be an inverse limit of mm-spaces~$\X_n$. Let $f\in \Lip_1(\X, \mm_\X)$.
    Then there exist $f_n\in\Lip_1(\X_n, \mm_{\X_n})$ such that $f_n\circ \sfp_n\to f$ in $\L^0(\X)$, where $n\mapsto \sfp_n$ is the sequence of projections attached to the inverse system given in~\eqref{d:IVS}.
\end{proposition}
\begin{proof}
By using an adapted Kuratowski embedding $\iota: \X \to \R^\infty$ in~\Cref{def:KE} and using the cylinder appoximation~in~\Cref{corollary: cylinder functions}, it is enough to prove that 1-Lipschitz cylinder functions are dense in $\Lip_1(\mathbb R^\infty, \mu)$ with respect to the~$\L^0(\mu)$-topology for every Radon measure~$\mu$. 
    By \Cref{cor: density of continuous functions}, it suffices to approximate all $g\in \Lip_1(\mathbb R^\infty, \tau^\infty, \ell_\infty)$ in $\L^0(\mu)$.

    Let $g$ be such a function.
    Let $g_n$ be defined by
    \begin{equation}
        g_n(x_1, \ldots)=g(x_1, \ldots, x_n, 0, 0, \ldots).
    \end{equation}
    Then $g_n$ is a sequence of 1-Lipschitz cylinder functions.
    By the $\tau^\infty$-continuity of $g$, $g_n\to g$ pointwise and thus in $\L^0(\mu)$. 
    This completes the proof.
\end{proof}

\subsection{Emm-representation of general pyramids and stability}
Thanks to \Cref{proposition: LAP}, we can improve on \Cref{proposition: piramidi - enmms} and obtain the full representation result.
\begin{theorem}\label{teo: piramidi - enmms}
    Let $\X=\varprojlim \X_n$ be an inverse limit with projection~$\p_n: \X \to \X_n$. Then the following hold:
    \begin{enumerate}[label=\textnormal{(\alph*)}]
        \item $\mathcal P_{\X_n} \xrightarrow{\tau_{\Pi}}\overline{\mathcal P_{\X}}^\square$.
        \item $\p_n^*\mathcal L_1(\X_n) \to \mathcal L_1(\X)$ in the Painlevé--Kuratowski topology in~$\L^0(\X)/\mathbb R$.
    \end{enumerate}
\end{theorem}
\begin{proof}
For (a), it suffices to show $$\overline{\bigcup_n \mathcal P_{\X_n}}^\square = \overline{\mathcal P_\X}^\square.$$
  By~\Cref{proposition: piramidi - enmms}, the inclusion $\subset$ has been shown.
  For the reverse inclusion, let $\Z \in \mathcal X$ be an mm-space  such that $\Z\prec \X$ and the goal is to  prove 
  $$\Z\in \overline{\bigcup_n \mathcal P_{\X_n}}^\square.$$ 
    Let $\varepsilon>0$. By \cite[Theorem 4.47]{Sh16}, there exists $N\in \N$ and  $\Z_1\prec Z$ with $\square(\Z, \Z_1)\leq \varepsilon$ such that  $\Z_1 \subset (\mathbb R^N, \ell_\infty)$.
    In particular, this yields $\Z_1\prec \X$.

    Let $f=(f_1, \ldots, f_N):\X\to \Z_1$ be a measure-preserving $1$-Lipschitz map.
    By \Cref{proposition: LAP}, there exist $n \in \N$ and $g_1, \ldots, g_N\in \Lip_1(\X_n, \mm_{\X_n})$ such that for $j=1, \ldots, N$ it holds 
    $$\sfd_{\L^0}(f_j, g_j\circ \p_n)\leq \frac{\varepsilon}N.$$
    Now define the map $g=(g_1, \ldots, g_N): \X_n\to (\mathbb R^N, \ell_\infty)$, set the space $\Z_2$ as the closure of $g(\X_n)$ and take the measure $\mm_{\Z_2}=g_\#\mm_{\X_n}$ regarded on the closure.
    By the 1-Lipschitz continuity of the maps $g_j$, it follows that $\Z_2\prec \X_n$.

    Noting~$\mm_{\Z_1}=f_\#\mm_\X$, $\mm_{\Z_2}=(g \circ \p_n)_\#\mm_\X$ and using~\eqref{in:LP}, we have
    \begin{equation}
        \sfd_{\sf P}(\mm_{\Z_1}, \mm_{\Z_2})\leq \sfd_{\L^0}(f, g\circ \p_n)\leq \varepsilon.
    \end{equation}
    By~\eqref{in:BP},  we obtain~$\square(\Z, \Z_2)\leq 2\varepsilon$. 
    As $\varepsilon$ can be arbitrarily small, we obtain
    $$\Z\in \overline{\bigcup_n \mathcal P_{\X_n}}^\square, $$
    which concludes the claim.

    For the latter statement, \Cref{proposition: LAP} immediately gives the one identification with the set of limit points in the Painlev\'e-Kuratowski convergence, and the fact that every subsequential limit in $\L^0$ of $1$-Lipschitz functions is again $1$-Lipschitz gives the other identification with the set of cluster points.  
\end{proof}
\begin{corollary} \label{c:IVI}
    Let $\mathcal P\in \Pi$ and let $\X_n$ be an inverse system asymptotic to $\mathcal P$ (recall \Cref{d:Asy}) and the inverse limit~$\X=\varprojlim \X_n$, then 
    $$\mathcal P=\overline{\mathcal P_{\X, c}}^\square=\overline{\mathcal P_\X}^\square.$$
\end{corollary}
        \begin{proof} 
        By definition, 
        $$ \overline{ \bigcup_{n \in \N} \mathcal P_{\X_n}}^{\square} \subset \overline{\mathcal P_{\X,c}}^\square \subset \overline{\mathcal P_\X}^\square.$$
        As $\X_n$ is an inverse system, 
        $$\tau_{\Pi}\text{-}\lim_{n \to \infty} \mathcal P_{\X_n} = \overline{ \bigcup_{n \in \N} \mathcal P_{\X_n}}^{\square}.$$
        Thus, by using~(a) in~\Cref{teo: piramidi - enmms} and the hypothesis that $\X_n$ is asymptotic to $\mathcal P$, 
        we conclude 
        $$\mathcal P=\tau_{\Pi}\text{-}\lim_{n \to \infty} \mathcal P_{\X_n} = \overline{ \bigcup_{n \in \N} \mathcal P_{\X_n}}^{\square} =\overline{\mathcal P_{\X,c}}^\square = \overline{\mathcal P_\X}^\square.$$
    \end{proof} 
  
\begin{theorem}[Pyramid representation theorem] \label{t:IPT}
    For every pyramid~$\mathcal P \in \Pi$, there exist an emm-space $\X$ such that $\mathcal P=\overline{\mathcal P_\X}^\square=\overline{\mathcal P_{\X, c}}^\square$.
\end{theorem}
\begin{proof}
Since every pyramid $\mathcal P$ has an inverse system~$\X_n$ asymptotic to $\mathcal P$ by \Cref{lemma: approximation of pyramids}, combined with \Cref{c:IVI} we obtain the pyramid representation thereom by the limit inverse limit emm-space~$\X=\varprojlim \X_n$.
\end{proof}
The representation given above is not, in general, unique.
\begin{example}[Non-uniquness of emm-spaces associated with pyramids] \label{e:CEG} 
    Let $\Gamma^\infty$ be the infinite-product of the $1$-dimensional centred Gaussian space with unit variance $(\R, |\cdot|, \gamma_1)$, and $X=(\mathbb R,|\cdot|,\gamma_{\sigma^2})$ be the $1$-dimensional centred Gaussian space with variance $\sigma^2=\frac{\sqrt{2}}{2}$.
    Then one has $\Gamma^\infty\times \X\succ\Gamma^\infty\succ \Gamma^\infty\times \X$ (see the proof below).
    In particular, 
    \begin{equation} \label{e:PCE}
        \mathcal {P}_{\Gamma^\infty\times \X}= \mathcal {P}_{\Gamma^\infty}.
    \end{equation}
    Nevertheless, the eigenvalues of the Laplacian associated with the Cheeger energies  of $\Gamma^\infty$ and $\Gamma^\infty\times \X$ are different. Indeed, the former is \(\{0, 1, \ldots\}\) while the latter is \(\{m+\sqrt{2}n: m, n \in \N_0\}\). So, the spaces \(\Gamma^\infty \times \X\) and \(\Gamma^\infty\) cannot be emm-isomorphic due to the invariance of Cheeger energies under emm-isomorphisms (\cite[Theorem~6.13]{SuYo25+}, which will be recalled in \Cref{t:ECE}). 
    Thus, the coincidence of pyramids does not necessarily imply the emm-isomorphism of two emm-spaces. We will see that this phenomenon cannot happen under for concentrated emm-spaces, see \Cref{theorem: isomorphism from pyramid}.

    The problem of characterising what pyramids have a unique representative up to emm-isomorphism is open.

\end{example}
\begin{proof}[Proof of \eqref{e:PCE}]
First, the projection
\(
\pi:\Gamma^\infty\times X\to \Gamma^\infty\) given by 
\(\pi(x,t)=x
\)
is clearly \(1\)-Lipschitz and measure-preserving. Hence
\[
 \Gamma^\infty \prec \Gamma^\infty\times X.
\]

Conversely, define
\(
F:\Gamma^\infty\to \Gamma^\infty\times X
\)
by
\[
F(x_1,x_2,x_3,\ldots)
=
\bigl((x_2,x_3,\ldots),\sigma x_1\bigr).
\]
Then, noting $\sigma^2=\frac{\sqrt{2}}{2}<1$, 
\[
\sfd_{\Gamma^\infty \times \X}\bigl(F(x),F(y)\bigr)^2
=
\sum_{i=2}^\infty|x_i-y_i|^2+\sigma^2|x_1-y_1|^2
\le
\sum_{i= 1}^\infty|x_i-y_i|^2
=
\ell_2(x,y)^2.
\]
Thus \(F\) is \(1\)-Lipschitz.

Moreover, 
\[
F_\#\gamma_1^{\otimes \infty}
=
\gamma_1^{\otimes\infty}\otimes\gamma_{\sigma^2},
\]
so \(F\) is measure-preserving. Therefore
\(
\Gamma^\infty\times X \prec \Gamma^\infty. 
\)
\end{proof}

\paragraph{The largest pyramid.}
    With our definition of pyramids, we can represent the largest pyramid $\mathcal X$ by 
$$\mathbb I_\infty=([0,1], \tau_{[0,1]}, \sfd_\infty, \mathsf{Leb}),$$
where $\tau_{[0,1]}$ and $\mathsf{Leb}$ are the standard Euclidean topology and the Lebesgue measure respectively, and 
\begin{equation}
    \sfd_\infty(x,y) =
    \begin{cases}
        0 & \qquad \text{if $x=y$}
        \\
        +\infty & \qquad \text{otherwise}.
    \end{cases}
\end{equation}
\begin{proposition}[Representation of the largest pyramid] \label{p:ME}
    $\mathbb I_\infty$ is an extended topological-metric-measure space and 
    $$\mathcal X = \mathcal P_{\mathbb I_\infty}$$
\end{proposition}
\begin{proof}
   To show that $\mathbb I_\infty$ is an extended topological-metric-measure space, let us notice that $\Lip_1(I_\infty, \tau_{[0, 1]}, \sfd_\infty)$ is identical to the set of all continuous functions.
   The fact that $\tau_{[0, 1]}$ is the initial topology of this family is tautological.
   For any couple of points $s<t\in [0, 1]$ and $M>0$, the function $f_M(r)=M\frac{r-s}{t-s}$ is such that $f_M\in \Lip_1(I_\infty, \tau_{[0, 1]}, \sfd_\infty)$ and $f_M(t)-f_M(s)=M$.
   In particular, $\sfd_\infty$ is $\tau_{[0, 1]}$-recovered.

   Take $\X \in \mathcal X$.
   Then, any parameter $F: [0,1) \to \X$ (i.e. any map such that $F_\#\mathsf{Leb} = \mm$) is $1$-Lipschitz with respect to $\sfd_\infty$.
   Thus, $X \in \mathcal P_{\mathbb I_\infty}$.
\end{proof}    

\paragraph{Stability under weak convergence of pyramids.} 
In~\cite{esakiInvariantsGromovsPyramids2024}, the log-Sobolev inequality, respectively the Poincar\'e inequality, for a general pyramid \(\mathcal P\) is defined by requiring every element of \(\mathcal P\) to satisfy the corresponding inequality, and the stability of these functional inequalities under the convergence of pyramids have been shown. Thanks to \Cref{t:IPT} and the stability of these functional inequalities under inverse limits~\cite[Theorem~1.2]{SuYo25+}, the stability theorem can be further improved as a stability at the level of  extended metric measure spaces associated with those pyramids. 
\begin{corollary}[Stability of log-Sobolev and Poincar\'e inequalities]\label{corollary: stability constants} 
Let $(\X_n)_{n \in \N}$ be a sequence of emm-spaces supporting  the log-Sobolev inequality~$\LS(C)$ with the log-Sobolev constant $C>0$ {\rm(}resp.~the Poincar\'e inequality $\P(C)$ with the Poincar\'e constant~$C>0${\rm)}. Suppose that 
$$\overline{\mathcal P_{X_n}}^\square \xrightarrow{\tau_{\Pi}} \mathcal P \in \Pi.$$
Then, there exists an emm-space $\X$ such that $\X$ supports~$\LS(C)$ {\rm (}resp.~$\P(C)${\rm )} and  $$\mathcal P=\overline{\mathcal P_{X}}^\square.$$
Furthermore, any emm-space $X$ representing $\mathcal P=\overline{\mathcal P_{X}}^\square$ supports $\LS(C)$ {\rm (}resp.~$\P(C)${\rm )}. 
\end{corollary}
\begin{proof}
Every mm-space $Y \in \mathcal P$ is a $\square$-limit of $Y_n \in \mathcal P_{\X_n}$ by the weak convergence of pyramid. $\Y_n$ supports $\LS(C)$ {\rm (}resp.~$\P(C)${\rm )} because $\Y_n \prec \X_n$ and $\LS(C)$ {\rm (}resp.~$\P(C)${\rm )} is stable  under a.e.~$1$-Lipschitz measure-preserving maps. 
Indeed, by the measure-preserving property, the LHS of $\LS(C)$ {\rm (}resp.~$\P(C)${\rm )} is invariant under push-forward. For the RHS, by the a.e.~$1$-Lipschitz property, we have the domination of the pullback Cheeger energies by \Cref{corollary: monotonicity Cheeger}, which concludes the stability.

Since the $\square$-convergence preserves $\LS(C)$ {\rm (}resp.~$\P(C)${\rm )} (see \cite[Theorems 4.11, 4.20]{esakiInvariantsGromovsPyramids2024}),  $Y$ satisfies $\LS(C)$ {\rm (}resp.~$\P(C)${\rm )}. Since we can construct $X$ as an inverse limit of $Y_k \in \mathcal P$ by \Cref{lemma: inverse system}, the stability of $\LS(C)$ {\rm (}resp.~$\P(C)${\rm )} under the invserse limit (see \cite[Theorem 1.2]{SuYo25+}) concludes the statement. 
\end{proof}

\subsection{Examples}
Here, we discuss pyramid representations of infinite-dimensional emm-spaces: infinite-product spaces,  abstract Wiener spaces and the configuration space. 
\subsubsection{Product spaces and Wiener space}
\begin{corollary}[product space] \label{c:PD}
    Let $\X^n=\prod_{i=1}^n \X_i$ and $\X^\infty = \prod_{i =1}^\infty \X_i$ be the Cartesian product  spaces of mm-spaces $\{\X_i\}_{i \in \N}$. 
    Then, 
    $$\overline{\mathcal P_{X^\infty}}^\square=\overline{\bigcup_n \mathcal P_{\X^n}}^\square.$$
\end{corollary}
\begin{proof}
    As the infinite-product space $\X^\infty$ is a particular case of inverse limit, \Cref{lemma: approximation of pyramids} and \Cref{teo: piramidi - enmms} yield the conclusion. 
\end{proof}

\begin{definition}[Abstract Wiener space] \label{d:WS}
Let \(H\) be a separable real Hilbert space and let \(W\) be a separable real Banach space. The triplet
\(\mathbb W=(W,H,\mu)\) is {\it an abstract Wiener space} if \(H\) is continuously and densely embedded into
\(W\), and \(\mu\) is the centred Gaussian measure on \(W\) whose Cameron--Martin space is
\(H\). This is equivalent to say that  every $\ell \in W^*$ is a centred Gaussian
random variable on \((W,\mu)\), and its variance is given by
\[
\int_W \ell(w)^2\,d\mu(w)
=
\|h_\ell\|_H^2,
\]
where \(h_\ell\in H\) is the unique element satisfying
\[
\ell(h)=\langle h,h_\ell\rangle_H
\qquad
\forall h\in H.
\]
\end{definition}

Let \(\mathbb W=(W,H,\mu)\) be an abstract Wiener space. We equip \(W\) with the Cameron--Martin
extended distance
\[
d_H(w_1,w_2):=
\begin{cases}
\|w_1-w_2\|_H, & w_1-w_2\in H,\\
+\infty, & \text{otherwise}.
\end{cases}
\]
The space~
\(
(W,\tau_W,d_H,\mu)
\)
is a Polish extended metric measure space, where \(\tau_W\) denotes the Banach topology on \(W\).

 Let $\Gamma^n=(\R^n, |\cdot|_{\ell_2}, \gamma^{\otimes n})$ be the $n$-dimensional centred Gaussian space with unit variance for $n \in \N$, and $\Gamma^\infty$ be the infinite-product of $(\R, |\cdot|, \gamma)$ endowed  with the product topology~$\tau^\infty$, viz., $\Gamma^\infty=(\R^\infty, \tau^\infty, \ell_2, \gamma^{\otimes \infty})$, which is an emm-space.  
\begin{corollary}[abstract Wiener space] \label{c:PWI}
    $$\overline{\bigcup_n \mathcal P_{\X^n}}^\square=\overline P_{\Gamma^\infty}^\square= \overline{\mathcal P_{\mathbb W}}^\square.$$
\end{corollary}
\begin{proof}
The first identity follows by \Cref{c:PD}.
By~\cite[Proposition 8.3]{SuYo25+}, any  infinite-dimensional abstract Wiener spaces is emm-isomorphic to \(\Gamma_\infty\).  Since pyramids are invariant under the emm-isomorphism, the second identity follows. 
\end{proof}

\subsubsection{Configuration spaces} \label{subsec:CF} 
The configuration space is the space of Radon point measures over a base space~$\X$. Unlike labelled particle spaces, the space of configurations is not naturally a Cartesian product. Its global structure can instead be described through restrictions to bounded windows, which form an inverse system. 

Let $\X \subset \R$ be a locally compact Polish subspace of $\R$. 
Let $\U(\X)$ be the space of locally finite discrete measures on $\X$, i.e., the totality of $\gamma=\sum_{i=1}^N \delta_{x_i}$ with $N \in \mathbb N_0 \cup \{+\infty\}$, $x_i \in \X$  and $\gamma(K)<+\infty$ for every compact $K \subset \X$. Regarding $\U(\X)$ as a subspace of the space of Radon measures over $\X$,  we endow $\U(\X)$ with the topology $\tau_{\sf vague}$ (called {\it vague topology}) by duality of compactly supported continuous functions in~$X$. Namely, $\gamma_n \to \gamma$ if and only if $\int_\X f \diff \gamma_n \to \int_\X f \diff \gamma$ for every compactly supported continuous function~$f$ in~$\X$.  We simply write $\U$ when $\X=\R$. 

Let $B_n\coloneqq(-n ,n)$, $\bar{B}_n\coloneqq[-n, n]$ and $\partial B_n=\bar B_n \setminus B_n$. Let $\pi^{(n)}$ be the Poisson measure with unit intensity on $\U(\bar B_n)$, that is, the unique probability measure on~$\U(\bar B_n)$ given by 
\begin{align} \label{d:PM}
\int_{\U(\bar B_n)} F \diff \pi^{(n)} = e^{-|\bar B_n|}\sum_{k=0}^\infty \frac{1}{k!}\int_{\bar B_n^{\times k}}F\Bigl(\sum_{i=1}^k \delta_{x_i}\Bigr) \diff x_1 \cdots \diff x_k
\end{align}
for every nonnegative Borel $F: \U(\bar B_n) \to \R$.
Let $\pi$ be the inverse limit measure of $\pi^{(n)}$ with bonding map $p_{m,n}: \U(\bar B_n) \to \U(\bar B_m)$ given by the restriction~$p_{m,n}(\gamma)=\gamma|_{\bar B_m}$ for $m \le n$. We call $\pi$ the {\it Poisson measure with unit intensity} in~$\U$. See, e.g., \cite[Section 2.1]{AlbKonRoe98} for further details. 
For $\gamma, \eta \in \U(\X)$, we define the following $\ell_2$-matching extended distance
$$\mssd_{\U(\X)}(\gamma, \eta)^2\coloneqq\inf_{\mssc \in \mathsf{Cpl(\gamma, \eta)}}\int_{\X \times \X} |x-y|^2 \diff \mssc (x, y), \qquad \inf \emptyset =+\infty $$
where $$\mathsf{Cpl}(\gamma, \eta)\coloneqq\{\mathsf c \in \U(\X \times \X): \mathsf c(A \times \X)=\gamma(A),\ \mathsf c(\X \times A)=\eta(A) \}.$$
Define 
$$\mssd_n(\gamma, \eta)=\inf_{\alpha, \beta \in \U(\partial B_n)} \mssd_{\U(\bar{B}_n)}(\gamma|_{B_n}+\alpha, \eta|_{B_n}+\beta) \qquad \gamma, \eta \in \U.$$
where $\gamma|_{B_n}$ and $\eta|_{B_n}$ are the restriction of $\gamma$ and $\eta$ on $B_n$.
The function~$\mssd_n$ is a semidistance, and $\mssd_n(\gamma, \eta)=0 \iff \gamma|_{B_n} = \eta|_{B_n}$, see~\cite[Proposition 3.2]{Suz25}. We write 
$$\gamma \overset{\mssd_n}\sim \eta \iff  \mssd_n(\gamma, \eta)=0.$$ Let $\U_n$ be the quotient metric space of $\U$ with respect to the equivalence relation $\overset{\mssd_n}\sim$ with the quotient map~$\mathsf q_n: \U \to \U_n$ and we denote by~$\tilde\mssd_n$  the quotient distance. 

We may regard $\pi^{(n)}$ as a measure on~$\U$ by regarding $\U(\bar B_n) \subset \U$.
Let $\pi_n=(\mathsf q_n)_\# \pi^{(n)}$ be the quotient measure of $\pi^{(n)}$ on $\U_n$, which is the same to define~$\pi_n=(\mathsf q_n)_\# \pi$.  We denote by $\bar \U_n=(\bar \U_n, \bar \sfd_n, \bar \pi_n)$  the completion of $(\U_n, \tilde \sfd_n)$, the canonical embedding $\iota_n: \U_n \to \bar \U_n$, and $\bar \pi_n=(\iota_n)_\# \pi_n$ is the push-forward measure to the completion. 
\begin{corollary}[Configuration space] \label{c:CFP}
Let $\U$ be the emm-space $(\U, \tau_{\sf vague}, \mssd_\U, \pi)$ and $\bar\U_n=(\bar\U_n, \bar\mssd_n, \bar\pi_n)$ be the mm-space as above. Then, 
$$\U \cong \varprojlim \bar\U_n.$$
In particular, 
$$\overline{\mathcal P_{\U}}^\square=\overline{\cup_{n \in \N} \mathcal P_{\bar\U_n}}^\square.$$
\end{corollary}
\begin{proof}
The bonding map~$p_{m,n}:\U(\bar B_n) \to \U(\bar B_m)$ descends through the quotient down to the map $\tilde p_{m, n}:\U_n \to \U_m$ given uniquely by the relation~$\tilde p_{m,n}\circ \mathsf q_n=\mathsf q_m$.  It is easy to see that $\tilde p_{m,n}: \U_n \to \U_m$ is a $1$-Lipschitz map, thus the map extends to the completion~$\bar p_{m, n}:\bar\U_n \to \bar\U_m$.
The spaces~$(\bar \U_n, \bar \sfd_n, \bar \pi_n)$ with bonding maps~$\bar p_{m,n}$ form an inverse system, which readily follows from the fact that $(\U(\bar B_n), \pi^{(n)})$ with $p_{m,n}$ forms an inverse system with the inverse limit 
$$(\U, \pi) = \varprojlim (\U(\bar B_n), \pi^{(n)})$$
(see, e.g., \cite[Section 2.1]{AlbKonRoe98}), and from the monotonicity~$\sfd_n \uparrow \sfd$ due to~\cite[Proposition 3.2]{Suz25}. 

Let $\bar{\mathsf{p}}_n: \U \to \bar \U_n$ be the map defined as $\bar{\mathsf{p}}_n=\iota_n\circ \mathsf q_n$. Then, $\bar{\mathsf p}_m = \bar p_{m,n} \circ \bar{\mathsf p}_n$, thus $(\bar{\mathsf p}_n(\gamma))_{n \in \N}$ takes value in~$\varprojlim \bar \U_n$. By the construction above, it is now standard to verify that  the map 
$$J: \U \to \varprojlim \bar \U_n \quad \text{defined as} \quad \gamma \mapsto J(\gamma)=(\bar{\mathsf p}_n(\gamma))_{n \in \N}$$ is an emm-isomorphism from $\U$ to $\varprojlim \bar \U_n$. Indeed, the isometry can be seen as 
$$\bar \sfd_\infty(J(\gamma), J(\eta))=\sup_n  \bar \sfd_{n}(\iota_n \circ \mathsf q_n(\gamma), \iota_n \circ \mathsf q_n(\eta))=\sup_n  \tilde \sfd_{n}( \mathsf q_n(\gamma),  \mathsf q_n(\eta)) = \sup_n  \sfd_n(\gamma, \eta) = \sfd_\U(\gamma, \eta),$$
where the quotient map~$\mathsf q_n$ is distance preserving by definition because $\tilde \sfd_n$ is the quotient distance. The measure-preserving property~$J_\# \pi = \bar \pi_\infty$ follows by observing that the $n$-th marginal of $J_\#\pi$ is~$\bar \pi_n$. By the uniqueness of the inverse limit measure, we conclude $J_\# \pi = \bar \pi_\infty$.
Thus, $\U \cong \varprojlim \bar \U_n$. 
    
    The second assertion is a direct application of the first assertion and \Cref{c:IVI}. 
\end{proof}
\section{Concentrated emm-spaces}\label{section: concentrated spaces}
For an emm-space~$\X$, the quotient space $\mathcal L_1(\X)=\Lip_1(\X, \mm_\X)/\mathbb R$ is closed in~$\L^0$ but not necessarily compact. We will see emm-spaces having non-compact $\mathcal L_1(\X)$ in \Cref{s:AP}, including abstract Wiener spaces and the infinite-particle configuration spaces.  This is a striking distinction from mm-spaces, where $\mathcal L_1(\X)$ is always compact. 
\begin{lemma} \label{p:LC}
    Let $\X$ be an emm-space. Then $\mathcal L_1(X)$ is closed in $\L^0$. 
    If $\X$ is an mm-space, then $\mathcal L_1(X)$ is compact. 
\end{lemma}
\begin{proof}
The latter statement has been proven in \cite[Proposition 4.46]{Sh16}.
Let $(f_n) \subset \mathcal L_1(X)$ be a $\sfd_{\L^0}$-Cauchy sequence converging to $f \in \L^0(\X)$.
By taking a subsequence, $f_n$ converges to $f$ $\mm$-a.e.. 
Thus, there exists $\Omega \subset \X$ with $\mm(\Omega)=1$ such that,  
\begin{equation}
    |f(x)-f(y)| = \lim_{n \to \infty}|f_n(x)-f_n(y)| \le \sfd_\X(x, y) \qquad x,y \in \Omega.
\end{equation}
The proof is therefore complete.
\end{proof}

A natural question would be, therefore,  exactly when an emm-space $\X$ has a compact~$\mathcal L_1(\X)$. In the following sections, we will prove that the compactness of $\mathcal L_1(\X)$  is equivalent to each of the following: 
\begin{itemize}
\item  a pyramid $\mathcal P=\overline{\mathcal P_\X}^\square$ is concentrated;
 \item $\X \in \mathfrak X_{\sf conc}$ can be identified to an element in the completion $\overline{\mathcal X}^{\sfd_{\sf conc}}$.
\end{itemize}

Having these equivalences in mind, we introduce a new notion of {\it concentrated emm-spaces} by imposing $\mathcal L_1(\X)$ to be compact.
\begin{definition}[Concentrated emm-space]
    We say that an emm-space $\X$ is a \emph{concentrated emm-space} if $\mathcal L_1(\X)$ is compact. We denote by $\mathfrak X_{\sf conc}$ the space of emm-equivalence classes of concentrated emm-spaces. 
\end{definition}
\begin{remark}
Notice that the concentration property is invariant under emm-isomorphism since any emm-isomorphism of emm-spaces induces an isometry for the associated $\mathcal L_1$.
\end{remark}
In the following sections, we study relations between concentrated emm-spaces and associated pyramids to address the first bullet point above. The second bullet point will be addressed in Sub\Cref{ss:OD}.
\paragraph{Essentially separable case.} 
If we impose that an emm-space~$(\X, \tau, \sfd, \mm)$ has a separable subspace~$(\Z, \tau_\sfd)$, then the compactness of $\mathcal L_1(\X)$ occurs if and only if $\X$ is emm-isomorphic to an mm-space. 
\begin{proposition}[Essentially separable case] \label{p:FDP}
Let \((\X,\tau,\mathsf d,\mathfrak m)\) be an emm-space.
Assume that there exists a Borel set \(\Z\subset \X\) with~$\mathfrak m(\Z)=1$
such that \((\Z,\tau_{\mathsf d})\) is separable. Then 
$$\mathcal L_1(\X) \  \text{is compact} \iff \X \cong \X' \quad \text{for some $\X' \in \mathcal X$}. $$
\end{proposition}
\begin{proof}
By \Cref{p:LC}, the RHS implies the LHS. We only prove the converse statement. 
Suppose that $\mathcal L_1(\X)$ is compact. Since $(\Z, \tau_{\sfd})$ is separable, $\Z$ intersects only
countably many finite-distance components. Here finite-distance components are given as follows:
let $\sim$ be the relation $x\sim y$ if and only if $\mathsf d(x,y)<\infty$.
For any $x\in \X$, its component $\Z_x=\{y \in \Z: \sfd(y,x)<\infty\}$  is called {\it finite-distance component of $x$} and we have the decomposition $\Z=\sqcup_{[x] \in \Z/\sim} \Z_x$.

These components form a countable
measurable partition of a full-measure set. If no component had full
$\mathfrak m$-measure, then some component $C$ would satisfy
\[
0<\mathfrak m(C)<1.
\]
For $n\in\mathbb N$, define
$
f_n:=n\,\mathbf 1_C.$
Each $f_n$ is $1$-Lipschitz because it is constant on each finite-distance
component, and between distinct components the distance is $+\infty$ by definition.
Let $n\ne m$, put $t:=|n-m|\ge1$, and let $c\in\mathbb R$. Then
\[
f_n-f_m-c=t\mathbf 1_C-c
\]
takes the values $t-c$ on $C$ and $-c$ on $X\setminus C$. If
$0<\varepsilon<1/2$, the two inequalities
$
|c|\le\varepsilon$ and 
$|t-c|\le\varepsilon$
cannot both hold. Hence
\[
\mathfrak m\bigl(|f_n-f_m-c|>\varepsilon\bigr)
\ge
\min\{\mathfrak m(C),1-\mathfrak m(C)\}.
\]
Choose
\[
0<\varepsilon<
\min\left\{
\frac12,\mathfrak m(C),1-\mathfrak m(C)
\right\}.
\]
Then, for every $c\in\mathbb R$,
\[
\mathfrak m\bigl(|f_n-f_m-c|>\varepsilon\bigr)>\varepsilon,
\]
and therefore
\[
\mathsf d_{\mathrm{KF}}([f_n],[f_m])\ge\varepsilon
\qquad(n\ne m),
\]
where $[f_n]$ denotes the representative of $f_n$ in $\mathcal L_1(X)=\Lip_1(\X, \mm)/\R$.
This contradicts the compactness of $\mathcal L_1(\X)$. Thus there is a
finite-distance component $C_0$ with
\[
\mathfrak m(C_0)=1.
\]
Since $C_0\cap \Z$ has full measure, $\mathsf d$ is finite on
$C_0\cap \Z$, and $C_0\cap \Z$ is $\mathsf d$-separable. Let
\[
\Y:=\overline{C_0\cap \Z}^{\,\sfd}
\]
be its metric completion, and let $\nu$ be the pushforward of
$\mm$ to $Y$. 

To conclude that $(\Y,\sfd,\nu)$ is an mm-space, we verify that $\mm$ is a $\sfd$-Borel measure on~$C:=C_0\cap \Z$, which readily implies that $\nu$ is $\sfd$-Borel in~the completion~$\Y$.  Take $D=\{a_i\}_{i \in \N} \subset C$ be a countable $\sfd$-dense set. Every open metric ball~$B^{\sfd}_r(a_i) \cap C$ is $\tau$-Borel because $\sfd$ is $\tau^{\times 2}$-l.s.c.. Since $D$ is countable and dense in $C$, the $\sfd$-Borel $\sigma$-algebra in $C$ is generated by $\{B^{\sfd}_r(a_i): r \in \Q_+, \ i \in \N\}$. Thus, we obtain the inclusion between the Borel $\sigma$-algebras 
$$\mathscr B_\sfd(C) \subset \mathscr B_\tau(C).$$ Since $\sfd$ metrises a finer topology than $\tau$ by hypothesis of emm-spaces, we also have $\mathscr B_\sfd(C) = \mathscr B_\tau(C)$.  As $\mm$ is $\tau$-Borel, $\mm$ is $\sfd$-Borel as well.

The canonical embedding map
\[
C_0\cap \Z\longrightarrow \Y
\]
is Borel, distance-preserving and measure-preserving. Hence $\X$ is emm-isomorphic to the mm-space $(\Y,\sfd,\nu)$.
 \end{proof}

 \begin{remark}
     As is clear from the proof above, the condition ``$\mathcal L_1(\X)$ is compact'' can be replaced by ``$\sfd|_{Z}<+\infty$''.
 \end{remark}
This particularly tells us that elements in~$\overline{\mathcal X}^{\sfd_{\sf conc}} \setminus \mathcal X$ consist of only essentially infinite-dimensional emm-spaces.
\begin{corollary}[Essentially finite-dimensional case] \label{c:FDP}
Let \((\X,\tau,\mathsf d,\mathfrak m)\) be an emm-space.
Assume that there exists a Borel set \(\Z\subset \X\) with~$\mathfrak m(\Z)=1$
such that \((\Z,\mathsf d)\) has finite Hausdorff dimension. Then 
$$\mathcal L_1(\X) \  \text{is compact} \iff \X \cong \X' \quad \text{for some $\X' \in \mathcal X$}. $$
\end{corollary}
\begin{proof}
    It suffices to show that the finite-dimensionality assumption implies
the~$\tau_{\mathsf d}$-separability on a full-measure set. Indeed, since
$\dim_{\mathcal H}(\Z,\mathsf d)<\infty$, choose $a<\infty$ with the $a$-dimensional Hausdorff measure~
\[
\mathcal H^a_{\mathsf d}(\Z)=0.
\]
For every $k\in \N$, there exists a countable cover
$
\Z\subset \bigcup_{i=1}^{\infty}U_{k,i}
$
with
$
\mathrm{diam}_{\mathsf d}(U_{k,i})<\frac1k.
$
Choosing one point from each $U_{k,i}$ having non-empty intersection with~$\Z$, the resulting countable set
is $\tau_\mathsf d$-dense in $\Z$. 
\end{proof}
 Without finite-dimensionality hypothesis, we have many examples of emm-spaces having compact $\mathcal L_1(\X)$, nonetheless not emm-isomorphic to any mm-space, which will be seen in \Cref{s:AP}. 
\subsection{Concentrated pyramids and concentrated emm-spaces}
As mentioned in the previous section, the concentration property of emm-spaces is closely linked to concentrated pyramids.
The pyramids associated with concentrated emm-spaces have stronger properties than general pyramids. 
The following compactness property plays a key role.
\begin{lemma}\label{lemma: compactness of 1-Lipschitz functions}
    Let $\X$ be a concentrated emm-space and let $(f_n)_{n\in \N}\subset \Lip_1(\X, \mm_\X)$ be such that $(f_n)_\#\mm$ is tight.
    Then $(f_n)_{n \in \N}$ is precompact in $\L^0$.
\end{lemma}
\begin{proof}
    Since $\X$ is concentrated, there exists a real sequence $(c_n)_{n\in \N}$ such that $n\mapsto f_n+c_n$ is compact in $\L^0$. 
    This means that, up to extracting a nonrelabeled subsequence, $n\mapsto f_n+c_n$ converges in $\L^0$, which in turns implies that $n\mapsto (f_n+c_n)_\#\mm$ is tight.

    Since both $n\mapsto (f_n+c_n)_\#\mm$ and $n\mapsto (f_n)_\#\mm$ are tight, this implies that (the nonrelabeled subsequence of) $n\mapsto c_n$ is bounded.
    Indeed, for instance, if $K$ is a compact set such that $ (f_n)_\#\mm(K)\geq \frac 34$ for all $n\in \N$ and $H$ is another  compact set such that $ (f_n+c_n)_\#\mm(H)\geq \frac 34$ for all $n\in \N$, then $H\cap (K+c_n)\neq \emptyset$ for all $n\in \N$.

    The boundedness of $c_n$ implies that $n\mapsto f_n=f_n+c_n-c_n$ is compact in~$\L^0$.
\end{proof}
The result can be extended to maps with values in $\ell^\infty$.
We start with an auxiliary lemma.
\begin{lemma}\label{lemma: L^0 compactness of maps}
    Let $(\X, \mm)$ be a probability space and let $\Y$ be a Polish space.
    Let $(f_k)_{k\in \N}$ be a sequence of continuous functions generating the topology of $\Y$.
    Let $(T_n)_{n\in \N}$ be a sequence of measurable maps from $\X$ to $\Y$ such that $(f_k\circ T_n)_{n\in \N}$ is precompact in $\L^0$ for all $k$ and that $(T_n)_{\#}\mm$ is tight.
    Then $T_n$ is precompact in $\L^0(\X, \Y)$. 
\end{lemma}
\begin{proof}
    Let $F:\Y\to (\mathbb R^\infty, \tau^\infty)$ defined by $F(x)=(f_1(x), \ldots, f_n(x), \ldots)$ and notice that $F$ is a homeomorphism on the image.
    By assumptions, the sequence $n\mapsto F\circ T_n$ is precompact in $\L^0(\X; \mathbb R^\infty)$.
    Moreover, if $\tilde T$ is a limit point with $\tilde T=\L^0-\lim_k F\circ T_{n_k}$, then $\tilde T_\#\mm$ is a weak limit point of $F_\# (T_{n_k})_{\#}\mm$.
    This implies that $\tilde T_\#\mm$ is concentrated on the image of $F$, thus for $\mm$-a.e.~$x$, $\tilde T(x)$ is in the image of $F$.
    This implies that the map $T:=F^{-1} \circ \tilde T$  is well defined.
    Finally, by continuity of $F^{-1}$ on the image of~$F$, we conclude that $T\in L^0(\X; \Y)$ is the limit of $T_{n_k}$, concluding the proof.
\end{proof}
\begin{corollary}\label{corollary: compactness of 1-Lipschitz functions}
    Let $\X$ be a concentrated emm-space. Suppose that $(f_n)_{n\in \N}\subset \Lip_1(\X, \mm; \ell^\infty)$ satisfies that $(f_n)_\#\mm$ is tight. 
    Then $(f_n)_{n \in \N}$ is precompact in $\L^0$. 
\end{corollary}
\begin{proof}
    Consider $$\Z=\overline{\bigcup_n\supp\big((f_n)_\#\mm\big)}^{\ell_\infty}$$
    endowed with the restriction of the ambient distance~$\ell_\infty$.
    Taking (non-relabelled) separably valued Borel representatives~$f_n$, the subset~$f_n(\X)$ is separable in $\ell^\infty$. Thus,  $\overline{f_n(\X)}^{\ell_\infty}$ is separable as well.  Since
    $$\supp\big((f_n)_\#\mm)\big) \subset \overline{f_n(\X)}^{\ell_\infty},$$ the union $\cup_n\supp\big((f_n)_\#\mm\big)$ is separable, hence, so is the closure~$\Z$. As $\Z$ is closed, $\Z$ is a Polish subspace of $\ell^\infty$.

    Let $k\mapsto g_k$ be a sequence of bounded 1-Lipschitz functions on $\Z$ generating its topology.
    By extending functions to the entire space by McShane extension~(\Cref{lemma: McShane}), we may regard $g_k\in \Lip_1(\ell^\infty)$ for all $k\in \N$.
    For every fixed $k\in \N$, the sequence $n \mapsto (g_k\circ f_n)_\#\mm$ is tight because $(f_n)_\#\mm$ is tight by hypothesis and $g_k$ is bounded. Thus, by \Cref{lemma: compactness of 1-Lipschitz functions}, the sequence $n\mapsto g_k\circ f_n$ is precompact in $\L^0$.
    By \Cref{lemma: L^0 compactness of maps}, the sequence $n\mapsto f_n$ is therefore precompact in $\L^0$. 
\end{proof}
As a consequence, we obtain a simpler form for the pyramids associated to concentrated emm-spaces.
\begin{proposition}\label{prop: the pyramid is closed}
    Let $\X$ be a concentrated emm-space. Then $\mathcal P_\X=\overline{\mathcal P_\X}^\square.$
\end{proposition}
\begin{proof}
    Let $\Y\in \overline{\mathcal P_\X}^\square$ and let $(\Y_n)_{n\in \N}\subset \mathcal P_\X$ be a sequence of mm-spaces such that $\Y_n\to \Y$ in the $\square$-topology. 
    Since the convergence in the $\square$-topology is equivalent to the convergence with respect to the Gromov-Prohorov distance, there exist $i, i_n$, $n\in \N$, isometric embeddings into $\ell^\infty$ such that $(i_n)_\#\mm_{\Y_n}\rightharpoonup i_\# \mm_\Y$.
    Furthermore, as~$(\Y_n)_{n\in \N}\subset \mathcal P_\X$,  there exist $(p_n)_{n\in \N}$, with $p_n:\X\to \ell^\infty$ $\mm_\X$-measurable and 1-Lipschitz, such that $(p_n)_\#\mm_\X=(i_n)_\#\mm_{\Y_n}$. 
    By \Cref{corollary: compactness of 1-Lipschitz functions}, up to subsequences there exists an $\mm_\X$-measurable, 1-Lipschitz map $p$ such that $p_n\to p$ $\mm_\X$-a.e.. In particular, this implies $p_\# \mm_\X=i_\#\mm_\Y$. Taking $f=i^{-1}\circ p: \X \to \Y$, which is $1$-Lipschitz and $f_\#\mm_\X=\mm_Y$,  we conclude $\Y\in \mathcal P_\X$.
\end{proof}

In order to discuss the relation between concentrated emm-spaces and concentrated pyramids we need the following auxiliary statement.
\begin{lemma}\label{lemma: hausdorff convergence}

Let $\X=\varprojlim\X_n$ with projections $\sfp_n:\X\to \X_n$ and assume that $(\mathcal L_1(\X_n))_{n \in \N}$ is a $\sfd_{\sf GH}$-Cauchy sequence. Then $\p_n^* \mathcal L_1(\X_n)$ converges to $\mathcal L_1(\X)$ in the Hausdorff sense. In particular, $\mathcal L_1(\X)$ is compact. 
\end{lemma}
A statement similar to the one above is claimed without proof in \cite[page~108]{Sh16}. We give a proof here for completeness.
\begin{proof}
    By \Cref{teo: piramidi - enmms}, the sequence $(\mathcal L_1(\X_n))_{n \in \N}$ embeds in $\L^0(\X)/\mathbb R$ via pullback, $\p_n^*\mathcal L_1(\X_n)\subset \p_{n+1}^*\mathcal L_1(\X_{n+1})$ and $\overline{\bigcup \p_n^*\mathcal L_1(\X_n)}^{\sfd_{\L^0}}=\mathcal L_1(\X)$.

    For $\varepsilon>0$ and $\Z$ a complete metric space, denote by
    \begin{equation}
        Cap(\varepsilon, \Z)\coloneqq \sup \{m: \Z \text{ contains }m\text{ disjoint } \varepsilon/2 \text {-balls}\}
    \end{equation}
    the $\varepsilon$-capacity and by 
    \begin{equation}
        Cov(\varepsilon, \Z)\coloneqq \inf\{m: \Z \text{ can be covered by }m\text{ disjoint }\varepsilon\text{-balls}\}
    \end{equation}    
    the $\varepsilon$-covering number. 
    Notice that $Cov(\varepsilon, \Z)\leq Cap(\varepsilon, \Z)$ and $Cap(\varepsilon, \Z)\leq Cov(\varepsilon/2, \Z)$ and that $\Z$ is compact iff $Cap(\varepsilon, \Z)<+\infty$ for all $\varepsilon>0$ iff $Cov(\varepsilon, \Z)<+\infty$ for all $\varepsilon>0$ (for more details, \cite[Section 3.1-3.2]{Sh16}).

    In order to prove the statement, it is enough to prove that $\mathcal L_1(\X)$ is compact, since $\p_n^*\mathcal L_1(\X_n)\subset \mathcal L_1(\X)$ for all $n$. 
    By contradiction, assume that there exists $\varepsilon$ such that $Cap(\varepsilon, \mathcal L_1(\X))=+\infty$. 
    For $M\in \N$, let $(f_j)_{j=1, \ldots, M}\subset \mathcal L_1(\X)$ be such that $\sfd_{\L^0/\mathbb R}(f_k, f_l)\geq \varepsilon$ for $l\neq k$ and $l, k\in\{1, \ldots, M\}$. 
    Now, by $\overline{\bigcup \p_n^*\mathcal L_1(\X_n)}^{\sfd_{\L^0/\mathbb R}}$, there exists $n\in \N$ and $(g_j)_{j=1, \ldots, M}\subset \mathcal L_1(\X_n)$ such that $\sfd_{\L^0/\mathbb R}(f_l, g_l)\leq \varepsilon/4$. 
    In particular, $\sfd_{\L^0/\mathbb R}(g_k, g_l)\geq \varepsilon/2$ for $l\neq k$ and $l, k\in\{1, \ldots, M\}$. Thus $Cap(\varepsilon/2, \X_n)\geq M$.

    Summing up, $\limsup_n Cap(\varepsilon/2, \X_n)=+\infty$. This is a contradiction: by Gromov-Hausdorff compactness, the spaces $\X_n$ are equi-uniformly bounded \cite[Section 7.4]{BurBurIva01} (see also \cite[Lemma 3.12]{Sh16}.
\end{proof}

As a consequence, we get this description of concentrated pyramids and concentrated spaces.
\begin{theorem}\label{thm: concentrated iff concentrated}
    Let $\mathcal P$ be a pyramid and $\X$ an emm-space such that $\overline{\mathcal P_\X}^\square=\mathcal P$.
    Then $\mathcal P$ is concentrated if and only if $\X$ is a concentrated emm-space.
\end{theorem}
\begin{proof}
    Assume that $\X$ is concentrated and let us prove that $\mathcal P$ is concentrated. Recalling \Cref{d:Asy}, we need to prove that  the family~$(\mathcal L_1(\Y))_{Y \in \mathcal P}$ is Gromov--Hausdorff-precompact. 
    
    By \Cref{prop: the pyramid is closed}, $\mathcal P=\overline{\mathcal P_\X}^\square=\mathcal P_\X$.  
    Since $\X\succ \Y$ for all $\Y \in \mathcal P=\mathcal P_\X$, the spaces $\mathcal L_1(\Y)$ can be embedded isometrically into $\mathcal L_1(\X)$ as a family of compact subsets. 
    By the compactness of $\mathcal L_1(\X)$, the family~of the compact spaces~$(\mathcal L_1(\Y))_{Y \in \mathcal P}$ embedded into $\mathcal L_1(\X)$ is Gromov--Hausdorff-precompact, thus $\mathcal P$ is concentrated.

    Conversely, assume that $\mathcal P=\overline{\mathcal P_\X}^\square$ is concentrated, and we prove that $\X$ is concentrated, i.e., $\mathcal L_1(\X)$ is compact in $\L^0(\X)$. By \Cref{lemma: inverse system}, we can take an inverse system~$\X_n$  such that $\X \cong \varprojlim\X_n$.
    In particular, $\X\succ \X_n\in \mathcal P$ for all $n\in \N$.    
    As $\mathcal P$ is concentrated, $(\mathcal L_1(\X_n))_{n\in \N}$ is Gromov--Hausdorff-precompact. By \Cref{lemma: hausdorff convergence}, $$\mathcal L_1(\X)=\text{GH-}\lim_n \mathcal L_1(\X_n)$$ and the limit~$\mathcal L_1(\X)$ is compact.
\end{proof}
Finally, for concentrated emm-spaces the Lipschitz order can be read through pyramids and it becomes a partial order. 
\begin{theorem}\label{theorem: isomorphism from pyramid}
    Let $\X$ be a concentrated emm-space and let $\Y$ be an emm-space such that $\mathcal P_\Y\subset \mathcal P_\X$. 
    Then $\Y$ is concentrated and $\Y\prec \X$. 
    In particular, if $\mathcal P_\X=\mathcal P_\Y$, then $\X\cong\Y$.
\end{theorem}
\begin{proof}
    Thanks to \Cref{prop: embedding} and up to emm-isomorphism, we can assume that $\Y$ is embedded in $\mathbb R^\infty$. 
    Let $\pi_n:\mathbb R^\infty\to \mathbb R^\infty$ be such that
    \begin{equation}
        (\pi_n(x))_j=\begin{cases}
            x_j\quad&\text{if }j\leq n,\\
            0&\text{otherwise}.
        \end{cases}
    \end{equation}
    Then, if $E_n=\pi_n(\mathbb R^\infty)$, the space~$(E_n, \tau^\infty,  \ell_\infty, (\pi_n)_\#\mm_\Y )$ is not only an emm-space but an mm-space, since the restriction topology coincides with the one induced by the metric.
    Moreover, we have $(E_n, \ell_\infty, (\pi_n)_\#\mm_\Y )\prec \Y$. 
    Thus, since $E_n\subset \mathbb R^\infty$, there exists a sequence $n\mapsto p_n:\X \to \mathbb R^\infty$ of Borel maps which are 1-Lipschitz on a full measure set and such that 
    \begin{equation} \label{e:pmpm}
    (p_n)_\#\mm_\X =(\pi_n)_\#\mm_\Y.
    \end{equation}

    Let ${\rm pr}_{j}: \R^\infty \to \R$ be the projection to the $j$-th coordinate~${\rm pr}_j(x)=x_j$ for $x=(x_n)_{n \in \N}$. The family $\{{\rm pr}_j\}_{j \in \N}$ generates the product topology in~$\R^\infty$. 
    Recalling that $p_n$ is $1$-Lipschitz measure-preserving as seen above and using the equality~\eqref{e:pmpm}, for each $j \in \N$  the sequence $n\mapsto {\rm pr}_j\circ p_n$ is in~$\Lip_1(\X, \mm_\X)$ and satisfies 
    \begin{equation}
        ({\rm pr}_j\circ p_n)_\#\mm_\X=({\rm pr}_j)_\#(\pi_n)_\#\mm_\Y.
    \end{equation}
    By continuity of ${\rm pr}_j$ for all~$j\in \N$, the RHS converges weakly to $({\rm pr}_j)_\# \mm_\X$, so it is tight.
    In particular, by \Cref{lemma: compactness of 1-Lipschitz functions}, for all $j\in \N$ the sequence $n\mapsto {\rm pr}_j\circ p_n$ is precompact in $\L^0(\X)$.  
    By \Cref{lemma: L^0 compactness of maps} we obtain the existence of the limit map~$p:\X\to \mathbb R^\infty$ such that, up to extracting a nonrelabeled subsequence, $p_n\to p$ in $\L^0$.   
    In particular, $p$ is 1-Lipschitz on a full measure set and $p_\#\mm_\X=\mm_\Y$, yielding 
    $$\Y\prec \X.$$
    Recall that $\mathcal L_1(\X)$ is compact due to the hypothesis that $\X$ is concentrated. 
    Isometrically embedding the complete metric space~$\mathcal L_1(\Y)$ into the compact metric space $\mathcal L_1(\X)$ via pull-back with $p$, we conclude that $\mathcal L_1(\Y)$ is compact. Thus, $\Y$ is concentrated.

    Finally, assume that $\mathcal P_\Y=\mathcal P_\X$. 
    Then there is a measure preserving map $p':\Y\to \X$ which is 1-Lipschitz on a full measure set. 
    Consider $q=p'\circ p:\X\to \X$, which is measure preserving and 1-Lipschitz on a full measure set. 
    Notice that the pullback (quotient) map~$[q^*]: \mathcal L_1(\X) \to \mathcal L_1(\X)$ is an isometric embedding endowed with the quotient Ky Fan distance. By compactness it is surjective. This implies that the pullback $q^*:\Lip_1(\X, \mm_\X)\to\Lip_1(\X, \mm_\X)$ is also surjective.    
    Since 
    \begin{equation}
        q^*\sfd_\X \geq_\square |q^*f(\cdot)-q^*f(\cdot\cdot)|
    \end{equation}
    for all $f\in \Lip_1(\X, \mm_\X)$, the surjectivity implies that 
    \begin{equation}
        q^*\sfd_\X \geq_\square |f(\cdot)-f(\cdot\cdot)|
    \end{equation}
    for all $f\in\Lip_1(\X, \mm_\X)$. 
    Thus, by \Cref{prop: reconstruction}, $q^*\sfd_\X\geq_\square \sfd_\X$. But, by the 1-Lipschitz property of $q$, $q^*\sfd_\X\leq_\square \sfd_\X$. Thus, $q$ is an emm-isomorphism. 
    This directly implies that $p$ and $p'$ are isomorphisms, thus concluding the proof.
\end{proof}
\begin{corollary}
    Let $\X$ be a concentrated emm-space.
    Let $\X_n$ be an inverse system such that $\mathcal P_{\X_n}\to \mathcal P_\X$. 
    Then $\X\cong \varprojlim\X_n$.
\end{corollary}

As a corollary, we have the equivalence between concentrated emm-spaces and concentrated pyramids. 
\begin{corollary}[Concentrated emm-spaces and pyramids] \label{c:ECP} The following are equivalent for an emm-space~$\X$$:$
\begin{itemize}
    \item $\X$ is concentrated, i.e., $\mathcal L_1(\X)$ is compact;
    \item the pyramid $\mathcal P=\overline{\mathcal P_\X}^{\square}$ is concentrated.
\end{itemize}
\end{corollary}
\subsection{A characterisation of concentrated metric measure spaces}
The goal of this section is to give a characterisation of extended metric measure spaces to be concentrated. 
\begin{definition}[Mm-factor approximation]\label{def:obs}
An emm-space $(\X,\tau,\mssd,\mfm)$ satisfies \emph{mm-factor approximation} if for every $\varepsilon>0$
there exist a complete separable metric space $(K_\varepsilon,\rho_\varepsilon)$ and a Borel
$1$-Lipschitz map $\pi_\varepsilon:(\X,\mssd)\to(K_\varepsilon,\rho_\varepsilon)$ such that
for every $f\in\mathrm{Lip}_1(\X, \mm_\X)$ there exist $g\in\mathrm{Lip}_1(K_\varepsilon)$
such that 
\begin{equation}\label{eq:Aeps}
\mfm\big(|f-(g\circ\pi_\varepsilon)|>\varepsilon\big)\le \varepsilon.
\tag{$A_\varepsilon$}
\end{equation}
\end{definition}

\begin{proposition}\label{thm:equiv}
Let $(\X,\tau,\mssd,\mfm)$ be an emm-space. The following are equivalent:
\begin{enumerate}[label=(\roman*)]
\item \label{equiv:1} $\mcL_1(\X)$ is compact (i.e., $\X$ is concentrated).
\item \label{equiv:2} $(\X,\tau,\mssd,\mfm)$ satisfies the mm-factor approximation~\eqref{eq:Aeps}.
\end{enumerate}
\end{proposition}
\begin{proof} {\it \ref{equiv:2} $\implies$ \ref{equiv:1}.}
Assume \ref{equiv:2}. Let $([f_n])_{n\in \N}$ be any sequence in $\mcL_1(\X)$, with representatives $f_n\in\mathrm{Lip}_1(\X, \mm_\X)$, and fix $\varepsilon_j\coloneqq2^{-j}$.
For each $j$, apply \ref{equiv:2} to obtain an mm-factor $\pi_j:\X\to K_j$, 
a function $g_{n,j}\in \mathrm{Lip}_1(K_j)$  such that
\[
\mfm\bigl(|f_n-(g_{n,j}\circ\pi_j)|>\varepsilon_j\bigr)\le \varepsilon_j.
\]
Equivalently, 
\begin{equation}\label{eq:approxKF}
\sfd_{\L^0}([f_n],\pi_j^*[g_{n,j}])\le \varepsilon_j .
\end{equation}
Note that, by \Cref{p:LC}, the space~$\mathcal L_1(K_j)$ is compact in $\L^0(K_j, \nu)$ for every Borel probability~$\nu$ on $K_j$, in particular, we take the pushforward measure $(\pi_j)_\#\mm$.
For each fixed $j$, the compactness of~$\mcL_1(K_j)$  implies
that the sequence $[g_{n,j}]$ has a convergent subsequence. Extract inductively a subsequence of indices
$n\mapsto n^{(j)}$ such that $[g_{n^{(j)},j}]$ converges in $\mcL_1(K_j)$, and make $n^{(j+1)}$  a subsequence of $n^{(j)}$.
By the diagonalisation argument, we obtain indices $n_k\coloneqq n^{(k)}_k$ such that, for every fixed $j$,
\begin{equation}\label{eq:conv_g}
[g_{n_k,j}]\quad\text{converges in }\mcL_1(K_j)\text{ as }k\to\infty.
\end{equation}

We claim $[f_{n_k}]$ is Cauchy in $\mcL_1(\X)$. Fix $\eta>0$ and choose $j$ with $2\varepsilon_j<\eta/3$.
By \eqref{eq:conv_g}, there exists $L$ such that for $k,\ell\ge L$,
\[
\sfd_{\L^0}([g_{n_k,j}],[g_{n_\ell,j}])<\frac\eta3.
\]
Applying the fact that the pull-back by measure-preserving map preserves $\sfd_{\L^0}$, we have 
\[
\sfd^{\mfm}_{\L^0}(\pi_j^*[g_{n_k,j}],\pi_j^*[g_{n_\ell,j}])=\sfd^{(\pi_j)_\# \mfm}_{\L^0}([g_{n_k,j}],[g_{n_\ell,j}])<\eta/3,
\]
where we use the notation~$\sfd^{\mfm}_{\L^0}$ and~$\sfd^{(\pi_j)_\# \mfm}_{\L^0}$ to distinguish the Ky Fan distance with the reference measure~$\mfm$ from the one with~$(\pi_j)_\# \mfm$ (but we drop it in the rest of arguments as it is clear from the context).
Now, by the triangle inequality and \eqref{eq:approxKF}, we have
\begin{align}
&d_{\L^0}([f_{n_k}],[f_{n_\ell}])
\\
&\le d_{\L^0}([f_{n_k}],\pi_j^*[g_{n_k,j}])
 + d_{\L^0}(\pi_j^*[g_{n_k,j}],\pi_j^*[g_{n_\ell,j}])
 + d_{\L^0}(\pi_j^*[g_{n_\ell,j}],[f_{n_\ell}])
 \\
&< \varepsilon_j + \eta/3 + \varepsilon_j
\\
&<\eta.
\end{align}
Thus $[f_{n_k}]$ is Cauchy. Since $\mathcal L_1(\X)$ is closed in $\L^0$ (\Cref{p:LC}), the sequence~$[f_{n_k}]$ converges in $\mcL_1(\X)$. Hence every sequence has a convergent subsequence, i.e.\ $\mcL_1(\X)$ is compact.

\bigskip
{\it \ref{equiv:1} $\implies$ \ref{equiv:2}.} Assume~\ref{equiv:1}. Fix $\varepsilon>0$.
Since $\mcL_1(\X)$ is compact, it is totally bounded, so there exist finitely many classes
$[f_1],\dots,[f_N]\in \mcL_1(\X)$ with $N=N(\e)$ such that for every $[f]\in \mcL_1(\X)$ there exists $1 \le i\le N$ satisfying
\begin{equation}\label{eq:net}
\sfd_{\L^0}([f],[f_i])\le \varepsilon.
\end{equation}
Now take $f\in\mathrm{Lip}_1(\X, \mm_\X)$. By \eqref{eq:net}, we can choose $i$ such that $\sfd_{\L^0}([f],[f_i])\le\varepsilon$.
By definition of $\sfd_{\L^0}$, there exists $c\in\mathbb{R}$ such that
\begin{equation}\label{eq:close_fi}
\mfm\bigl(\{|f-f_i-c|>\delta\}\bigr)\le \delta.
\end{equation}
Let $K_\e:=\R^N$ and $\rho_\varepsilon(x,y)=\|x-y\|_\infty$. 
Define the Borel map
\[
\pi_\varepsilon:\X\to K_\varepsilon,\qquad \pi_\varepsilon(x)\coloneqq( f_1(x),\dots,f_N(x)).
\]
Take $g_i:K_\varepsilon\to\mathbb{R}$ be the projection to the $i$-th coordinate $g_i(z)=z_i$. 
Then $g_i\in\mathrm{Lip}_1(K_\varepsilon)$. Taking $\tilde{g}_i:=g_i-c$ and 
using \eqref{eq:close_fi}, we have
\begin{align}
\mfm\bigl(\{|f-(\tilde g_i\circ\pi_\varepsilon)|>\varepsilon\}\bigr) = \mfm\bigl(\{|f-f_i-c|>\varepsilon\}\bigr) < \varepsilon.
\end{align}
This is  \eqref{eq:Aeps}.
\end{proof}

\begin{remark}
The mm-factor condition~\eqref{eq:Aeps} a posteriori coincides with  $(K_\varepsilon, (\pi_\varepsilon)_\#\mm)\prec \X$ and $\sfd_{\sf conc}(\X, K_\varepsilon)\leq \varepsilon$, which rephrases the density of $\mathcal X$ in $(\mathfrak X_{\sf conc}, \sfd_{\sf conc})$.
\end{remark}

\section{Observable distance for concentrated emm-spaces} \label{s:ODC}
\subsection{Observable distance} \label{ss:OD}
We now extend the observable distance to  concentrated emm-spaces. 
\begin{definition}[Observable distance] \label{d:Obs}
Let $\X, \Y$ be concentrated emm-spaces. The \emph{observable distance} between $\X$ and $\Y$ is defined as 
\begin{equation}\label{eq: conc emm}
    \sfd_{\sf conc}(\X, \Y)=\inf_{\iota_\X, \iota_\Y} \sfd^{\L^0}_{\sf H} (\iota_X^*\Lip_1(\X, \mm_\X), \iota_\Y^*\Lip_1(\Y, \mm_\Y)),
\end{equation}
where $\iota_\X, \iota_\Y$ range among the parameters for $\X$ and $\Y$.
\end{definition}

\begin{remark}\
\begin{itemize}
    \item 
    It is easy to see that this definition is invariant with respect to emm-isomorphisms.
    Indeed given $\Phi:\X\to\X'$ an emm-isomorphism, $\iota_\X$ is a parameter for $\X$ if and only if $\Phi\circ\iota_\X$ is a parameter for $\X'$ and 
    \begin{equation}
        (\Phi\circ\iota_\X)^* \Lip_1(\X', \mm_{\X'})=\iota_\X^* \Phi^*\Lip_1(\X', \mm_{\X'})=\iota_\X^*\Lip_1(\X, \mm_\X).
    \end{equation}
    The same holds for the inverse emm-isomorphism.

      \item It is not immediate from the definition that $\sfd_{\sf conc}$ is a distance in  emm-isomorphism classes of concentrated emm-spaces $\mathfrak X_{\sf conc}$ since one needs to prove that $\sfd_{\sf conc}(\X, \Y)=0$ implies that $\X$ is emm-isomorphic to $\Y$. 
    This will follow by identifying $(\mathfrak X_{\sf conc}, \sfd_{\sf conc})$ and the completion $(\overline{\mathcal X}^{\sfd_{\sf conc}}, \sfd_{\sf conc})$, which will be addressed in the proof of~\Cref{t:2} below.
    \end{itemize}
\end{remark}
Similarly to the case of mm-spaces, an equivalent definition can be given in terms of couplings.
Given $\ppi\in\adm(\mm_\X, \mm_\Y)$, we write
\begin{equation}
    \sfd^\sppi_{\sf conc}(\X, \Y)=\sfd_{\L^0(\sppi)}^{\sf H}(\proj_\X^*\Lip_1{(\X, \mm_\X)}, \proj_\Y^*\Lip_1{(\Y, \mm_\Y)}),
\end{equation}
where $\proj_\X, \proj_\Y$ are the projections from the product space~$\X \times \Y$ to each factors.
\begin{lemma}
    Given $\X, \Y$ concentrated emm-spaces, 
    \begin{equation}
        \sfd_{\sf conc}(\X, \Y)=\inf_{\sppi\in\adm(\mm_\X, \mm_\Y)}\sfd^\sppi_{\sf conc}(\X, \Y).
    \end{equation}
\end{lemma}
The proof follows the same lines as the one for mm-spaces in \cite{nakajimaBoxDistanceObservable2022}, we show it here for completeness.
\begin{proof}
    Given a pair of parameters $\iota_\X, \iota_\Y$, it is immediate by definition that the coupling $\ppi= (\iota_\X, \iota_\Y)_\#\mathsf{Leb}$ satisfies $$\sfd^\sppi_{\sf conc}(\X, \Y)= \sfd^{\L^0}_{\sf H} (\iota_X^*\Lip_1(\X, \mm_\X), \iota_\Y^*\Lip_1(\Y, \mm_\Y)).$$
    This implies $\sfd_{\sf conc}(\X, \Y)\geq \inf_{\sppi\in\adm(\mm_\X, \mm_\Y)}\sfd^\sppi_{\sf conc}(\X, \Y)$.
    
    Viceversa, given a coupling $\ppi$ on $\X\times \Y$, let $\iota_{\X\times\Y}=(\iota_\X, \iota_\Y)$ be a parameter for $(\X\times\Y, \ppi)$.
    Then $\iota_\X, \iota_\Y$ are parameters for $\X$ and $\Y$ and a direct computation gives $\sfd^\sppi_{\sf conc}(\X, \Y)= \sfd^{\L^0}_{\sf H} (\iota_X^*\Lip_1(\X, \mm_\X), \iota_\Y^*\Lip_1(\Y, \mm_\Y))$, completing the proof.
\end{proof}
The extension to concentrated spaces of the the observable distance given above actually coincides with the completion of the observable distance as defined on mm-spaces.
More precisely, each element of $\overline{\mathcal X}^{\sfd_{\sf conc}}$ can be identified biunivocally with one in $\mathfrak X_{\sf conc}$ and the observable distance on $\mathfrak X_{\sf conc}$ coincides with the one on $\overline{\mathcal X}^{\sfd_{\sf conc}}$.

Recall that a completion of a metric space~$(\Z, \sfd_\Z)$ is a pair $(\widehat{\Z}, \iota)$, where $\widehat \Z$ is a complete metric space, $\iota: \Z \to \widehat \Z$ is an isometric embedding (called {\it canonical isometric embedding}), and the image~$\iota(\Z)$ is dense in $\widehat{\Z}$.

 \begin{theorem}\label{t:Vt2} 
 The following hold:
 \begin{itemize}
\item $(\mathfrak X_{\sf conc}, \sfd_{\sf conc})$ is complete and separable.
    \item $(\mathfrak X_{\sf conc}, \sfd_{\sf conc})$ is a completion of $(\mathcal X, \sfd_{\sf conc})$. The canonical isometric embedding of $\mathcal X$ into~$\mathfrak X_{\sf conc}$ for the completion is the trivial map given by identifying each mm-isomorphism class in~$\mathcal X$ with the associated emm-isomorphism class~in~$\mathfrak X_{\sf conc}$.
\item the map $\mathcal P_\bullet: (\mathfrak X_{\sf conc}, \tau_{\sfd_{\sf conc}}) \to (\Pi, \tau_{\Pi})$ is homeomorphic onto its image.
\end{itemize}
In particular, the following are equivalent for $(\X_n)_{n \in \N}\subset \mathfrak X_{\sf conc}$ and $\X \in \mathfrak X_{\sf conc}$$:$
\begin{enumerate}
\item $\X_n \xrightarrow{\sfd_{\sf conc}} \X$;
\item $\mathcal P_{\X_n} \xrightarrow{\tau_{\Pi}} \mathcal P_\X$.
\end{enumerate}
\end{theorem}
\begin{remark}
    Notice that we have abused the notation $\mathcal P_\bullet$ above to indicate two maps: (a) the map sending concentrated emm-spaces to the associated pyramids via the $1$-Lipschitz  order (\Cref{d:PRM}); (b) the map sending elements in $(\overline{\mathcal X}^{\sfd_{\sf conc}}, \sfd_{\sf conc})$ to pyramids via the unique extension property of the 1-Lipschitz map~\eqref{d:EPY}. A posteriori, this  notation is justified since 
    these two maps coincide by the proof below of the identification of~$\mathfrak X_{\sf conc}$ with the completion of $\mathcal X$.  
\end{remark}
\begin{proof}
    It is enough to show that there is an isometry $\varphi$ between $(\mathfrak X_{\sf conc}, \sfd_{\sf conc})$ and $(\overline {\mathcal X}^{\sfd_{\sf conc}}, \sfd_{\sf conc})$ that leaves $\mathcal P_\bullet$ invariant.
    Then the other statements follow  since $\mathcal P_\bullet:\overline {\mathcal X}^{\sfd_{\sf conc}}\to \Pi$ is an homeomorphism on the image (see~(\cite[Section 3$\frac{1}{2}$]{Gro06}\cite[Theorem 7.27]{Sh16}), and  $(\overline {\mathcal X}^{\sfd_{\sf conc}}, \sfd_{\sf conc})$ is a complete and separable metric space, where the separability follows by the inequality~$\sfd_{\sf conc} \le \square$ in~\eqref{e:BXD} and the separability of $(\mathcal X, \square)$. 

    By \Cref{t:IPT,thm: concentrated iff concentrated,theorem: isomorphism from pyramid}, the map $\mathcal P_\bullet: \mathfrak X_{\sf conc}\to \Pi$ is invertible.
    By composing it with the inverse of $\mathcal P_\bullet:\overline {\mathcal X}^{\sfd_{\sf conc}}\to \Pi$, we get a bijection $\varphi: \mathfrak X_{\sf conc}\to \overline {\mathcal X}^{\sfd_{\sf conc}}$.
    By construction, $\mathcal P_\bullet$ is left invariant by $\varphi$ and $\varphi$ is the identity when restricted to $\mathcal X$.

    Let us prove that $\varphi$ is an isometry. 
    Let us start by proving that whenever $\X=\varprojlim\X_n$ is a concentrated emm-space, then it holds $\sfd_{\sf conc}(\X, \X_n)\to 0$.
    Indeed, if $\iota_\X$ is any parameter for $\X$, $\iota_{\X_n}=\sfp_n\circ\iota_\X$ is a parameter for $\X_n$, where $\sfp_n$ are the projections in the inverse limit.
    Moreover
    \[
    \begin{split}
    \label{al:numero}
        \sfd_{\sf conc}(\X, \X_n)&\leq \sfd^{\L^0}_{\sf H} (\iota_X^*\Lip_1(\X, \mm_\X), \iota_{\X_n}^*\Lip_1(\X_n, \mm_{\X_n}))\\
        &=\sfd^{\L^0/\mathbb R}_{\sf H} (\iota_X^*{\mathcal L}_1(\X), \iota_{\X_n}^*{\mathcal L}_1(\X_n))\\
        &=\sfd^{\L^0/\mathbb R}_{\sf H} ({\mathcal L}_1(\X), (\sfp_n)^*{\mathcal L}_1(\X_n)).
    \end{split}
    \]
    Now, by \Cref{thm: concentrated iff concentrated} and noting $(\X_n)_{n\in \N}\subset \mathcal P_\X$, the sequence $n\mapsto \mathcal L_1(\X_n)$ is GH-precompact.
     Let $\Z$ be  a metric space and $(\X_{n_k})_{k\in \N}$ be a subsequence such that $\mathcal L_1(\X_n)\to \Z$ in the GH-topology. Note that  $\X$ is isomorphic to the subsequential limit $\varprojlim \X_{n_k}$ because $\X=\varprojlim \X_{n}$ for the whole sequence.
    By \Cref{lemma: hausdorff convergence}, $Z$ is thus isomorphic to $\mathcal L_1(\varprojlim \X_{n_k})$ and thus to $\mathcal L_1(\X)$.
    This  implies that the last term in \eqref{al:numero} goes to 0, yielding $\sfd_{\sf conc}(\X, \X_n)\to 0$. 

    Notice also that the observable distance between concentrated emm-spaces satisfies a triangle inequality.
    Indeed, given $\X, \Y, \Z\in \mathfrak X_{\sf conc}$, $\ppi_{\X, \Y}\in \adm(\mm_\X, \mm_\Y)$ and $\ppi_{\Y, \Z}\in \adm(\mm_\Y, \mm_\Z)$, then the composition $\ppi_{\X, \Z}$ is in $\adm(\mm_\X, \mm_\Z)$ and we have
    \begin{equation}
        \sfd_{\sf conc}^{\sppi_{\X, \Z}}(\X, \Z)\leq \sfd_{\sf conc}^{\sppi_{\X, \Y}}(\X, \Y)+\sfd_{\sf conc}^{\sppi_{\Y, \Z}}(\Y, \Z).
    \end{equation}
    Taking the infimum over $\ppi_{\X, \Y}$ and $\ppi_{\Y, \Z}$ yields the triangle inequality.

    Finally, let $\X, \Y\in\mathfrak X_{\sf conc}$.
    Up to emm-isomorphism, we can assume that $\X=\varprojlim\X_n$ and $\Y=\varprojlim\Y_n$ are inverse limits.
    In particular, by the triangle inequality, $\X_n=\varphi(\X_n)$ and $\Y_n=\varphi(\Y_n)$ are $\sfd_{\sf conc}$-Cauchy in $\overline {\mathcal X}^{\sfd_{\sf conc}}$, thus recalling that $\varphi$ is the identity map on~$\mathcal X$, 
    \begin{equation}
        \sfd_{\sf conc}(\varphi(\X), \varphi(\Y))=\lim_n {\sfd}_{\sf conc}(\varphi(\X_n), \varphi(\Y_n))=\lim_n {\sfd}_{\sf conc}(\X_n, \Y_n).
    \end{equation}
    Finally, on one hand
    \begin{align}
        \sfd_{\sf conc}(\varphi(\X), \varphi(\Y))&=\lim_n {\sfd}_{\sf conc}(\X_n, \Y_n)\\
        &\leq \lim_n {\sfd}_{\sf conc}(\X_n, \X)+\sfd_{\sf conc}(\X, \Y)+\lim_n{\sfd}_{\sf conc}(\Y, \Y_n)\\
        &=\sfd_{\sf conc}(\X, \Y);
    \end{align}
    on the other hand, 
    \begin{align}
        \sfd_{\sf conc}(\X, \Y)&\leq \lim_n {\sfd}_{\sf conc}(\X, \X_n)+\lim_n\sfd_{\sf conc}(\X_n, \Y_n)+\lim_n{\sfd}_{\sf conc}(\Y_n, \Y)\\
        &=\sfd_{\sf conc}(\varphi(\X), \varphi(\Y)). \qedhere
    \end{align}
\end{proof}
\begin{corollary}[Equivalent characterisations of concentrated emm-spaces] \label{c:ECE} The following are equivalent for an emm-space~$\X$$:$
\begin{itemize}
    \item $\X$ is concentrated, i.e., $\mathcal L_1(\X)$ is compact$;$
    \item $\X \in \overline{\mathcal X}^{\sfd_{\sf conc}}$ $;$ 
    \item $\mathcal P=\overline{\mathcal P_\X}^\square$ is concentrated,
\end{itemize}
where the second bullet point should be read in the sense of the identification of $\mathfrak X_{\sf conc}$ with~$\overline{\mathcal X}^{\sfd_{\sf conc}}$ in the second statement in~\Cref{t:Vt2}.
\end{corollary}
\begin{proof} 
    This is an immediate consequence of \Cref{c:ECP}, the second statement in~\Cref{t:Vt2} and the first statement in~\Cref{t:EPM}. 
\end{proof}

\subsection{Fibrations}\label{subsection: fibration}
In the setting of mm-spaces, the frameworks of fibration, introduced in \cite{gigliStabilityHeatFlow2024} (see also \cite{vinciniVariationalProblemsNonsmooth2025}), gives a simple but robust way to interpret the convergence of spaces in an extrinsic fashion.
For instance, both measure Gromov-Hausdorff convergence and convergence in concentration have a simple characterization in this language.
At the same time, it allows to formulate associated convergences of measures and functions defined on varying spaces, giving flexible tools to tackle stability problems. In this subsection, we extend this approach to emm-spaces.

\begin{definition}[Fibration]
    Let $(\X_n, \mm_n)$ and $(\X, \mm)$ be Polish spaces endowed with Borel probability measures.
    We say that the sequence $n\mapsto \X_n$ fibrates over $\X$ with projections $n\mapsto p_n:\X_n\to\X$ if the maps $p_n$ are Borel and the weak convergence $(p_n)_\#\mm_n\rightharpoonup \mm$ holds.
\end{definition}
Notice that the measures $(p_n)_\#\mm_n$ do not need to be absolutely continuous with respect to $\mm$.
In particular, the image of $p_n$ need not be contained in $\supp(\mm)$.
Fibrating sequences allow to define natural notions of convergence for measures and functions along varying spaces.
Here we recall the definitions and relevant properties.
\begin{definition}[Strong/weak convergences along fibration]
    Let $n\mapsto (\X_n, \mm_n)$ be a sequence fibrating over $(\X, \mm)$ with projections $p_n$.
    Let $n\mapsto \mu_n$ be a sequence of finite Borel measures on $\X_n$ and $\mu$ a finite Borel measure on $\X$.
    Let also $n\mapsto f_n\in \L^p(\X_n)$ and $f\in \L^p(\X)$ for some $p\in(1, +\infty)$.
    Finally, let $\Z$ be a Polish space, $n\mapsto T_n\in \L^0(\X_n; \Z)$  and $T\in \L^0(\X; \Z)$.
    We define weak convergences indicated by $\rightharpoonup$, and strong convergences indicated by $\to$ as follows:
    \begin{itemize}
        \item $\mu_n\rightharpoonup\mu$ if $(p_n)_\#\mu_n\rightharpoonup\mu$;
        \item $f_n\rightharpoonup f$ in $\L^p$ if $\sup_n\|f_n\|_{\L^p}<+\infty$ and $f_n\mm_n\rightharpoonup f\mm$;
        \item $f_n\to f$ in $\L^p$ if $f_n\rightharpoonup f$ in $\L^p$ and $\|f_n\|_{\L^p}\to \|f\|_{\L^p}$.
        \item $T_n\to T$ in $\L^0$ if for all $\Phi\in C^0_{\sf b}(\Z)$, $(\Phi\circ T_n) \mm_n \rightharpoonup (\Phi\circ T) \mm$.
    \end{itemize}
\end{definition}
\paragraph{Good coupling.}
In order to work with the strong convergences,  it is useful to introduce the notion of good couplings, which allow to ``metrise'' these convergences.
In the following,
we use the notation $\proj_i: \prod_{i=1}^n A_i \to A_i$ for the projection to the $i$-th factor.

\begin{definition}[Good coupling]
    Let $n\mapsto (\X_n, \mm_n)$ be a sequence fibrating over $(\X, \mm)$.
    A sequence of Borel measures $n\mapsto \aalpha_n\in\adm(\mm_n, \mm)$ is said to be a sequence of {\it good couplings} if $(p_n\times id)_\#\aalpha_n\rightharpoonup (id, id)_\#\mm$.
\end{definition}
\begin{remark}
    It is easy to see that any fibration admits the existence of good couplings.
    Indeed, if $n\mapsto \X_n$ fibrates over $\X$ with projections $n\mapsto p_n$, taking $$\aalpha_n\coloneqq (id, p_n)_\#\mm_{\X_n}, \quad  \text{any} \ \  \bbeta_n\in \adm\Big((p_n)_\# \mm_{\X_n}, \mm_\X\Big) \quad \text{s.t.} \quad \bbeta_n\rightharpoonup (id, id)_\#\mm_{\X}$$ (e.g., you can always take $\bbeta_n$ to be an optimal coupling between its marginals with respect to the cost ${\mathfrak c}=\sfd\wedge 1$, where $\sfd$ is a distance metrising the topology on $\X$), then for any gluing $\ggamma_n$ of $\aalpha_n$ and $\bbeta_n$, i.e.~any probability measure $\ggamma_n\in\pr(\X_n\times \X\times\X)$ such that $(\proj_{1, 2})_\#\ggamma_n=\aalpha_n$, $(\proj_{2, 3})_\#\ggamma_n=\bbeta_n$, $(\proj_1, \proj_3)_\#\ggamma_n$ is a good coupling. 
\end{remark}
\begin{definition}
    Let $n\mapsto (\X_n, \mm_n)$ be a sequence fibrating over $(\X, \mm)$ and let $\Z$ be a Polish space.
    Let $\sfd_\Z$ be a distance on $\Z$ metrising  the topology of~$\Z$ and $n\mapsto \aalpha_n$ be a sequence of good couplings.
    For $n\in \N$, $T_n\in \L^0(\X_n;\Z)$ and $T\in \L^0(\X;\Z)$, we put
    \begin{equation}
        \sfd^{\saalpha_n}_{\L^0}(T_n, T)\coloneqq \inf\bigg\{\varepsilon>0: \aalpha_n\Big(\big\{\sfd_\Z(T_n\circ\proj_1, T\circ\proj_2)>\varepsilon\big\}\Big)<\varepsilon\bigg\}. 
    \end{equation}
\end{definition}
\begin{remark}\label{remark: triangle inequality}
    Notice that, while $\sfd^{\saalpha_n}_{\L^0}(\cdot,\cdot\cdot)$ is obviously not a distance in itself since the first argument is in $\L^0(\X_n)$ while the second argument is in~$\L^0(\X)$.
    However, it satisfies a sort of triangle inequality:
    indeed, if $n\in \N$, $S_n, T_n\in \L^0(\X_n; \Z)$ and $S, T\in \L^0(\X; \Z)$, then it holds
    \begin{align}
        \sfd^{\saalpha_n}_{\L^0}(T_n, T)&\leq \sfd_{\L^0}(T_n, S_n)+ \sfd^{\saalpha_n}_{\L^0}(S_n, S)+\sfd_{\L^0}(S, T),\\
        \sfd_{\L^0}(T_n, S_n)&\leq \sfd^{\saalpha_n}_{\L^0}(T_n, T)+\sfd_{\L^0}(T, S)+ \sfd^{\saalpha_n}_{\L^0}(S_n, S),\\
        \sfd_{\L^0}(S, T)&\leq \sfd^{\saalpha_n}_{\L^0}(T_n, T)+\sfd_{\L^0}(T_n, S_n)+ \sfd^{\saalpha_n}_{\L^0}(S_n, S).
    \end{align} 
\end{remark}
\begin{remark}
    The quantity $\sfd^{\saalpha_n}_{\L^0}$ depends on the distance $\sfd_\Z$ chosen to metrise the topology on $\Z$.
    As we will see later, any choice of such a distance can be utilised to ``metrise'' (in the sense of the forthcoming \Cref{proposition: L^0 properties}) the exact same notion of convergence in $\L^0$ along the fibration.
    Since in the following there will never be competing choices for $\sfd_\Z$, no confusion will be caused by leaving the choice implicit.
\end{remark}
In order to characterize the convergence in $\L^0$  we need a further definition.
\begin{definition}
    Let $n\mapsto (\X_n, \mm_n)$ be a sequence fibrating over $(\X, \mm)$.
    Let $(\mu_n)_{n\in \N}$ be a sequence of measures, where $\mu_n$ is a finite Borel measure on $\X_n$, and let $\mu$ be a finite Borel measure on $\X$.
    We say that the sequence $n\mapsto \mu_n$ is equi-integrable if whenever $n\mapsto E_n\subset \X_n$ satisfies $\mm_{\X_n}(E_n)\to 0$, then $\mu_n(E_n)\to 0$.
\end{definition}
This also allows to speak about a finer convergence for measures.
\begin{definition}[$\L^1$-weak convergence of measures] \label{d:L1W}
    Let $n\mapsto (\X_n, \mm_n)$ be a sequence fibrating over $(\X, \mm)$.
    Let $(\mu_n)_{n\in \N}$ be a sequence of measures, where $\mu_n$ is a finite Borel measure on $\X_n$, and let $\mu$ be a finite Borel measure on $\X$.
    We say that {\it $\mu_n$ converges to $\mu$ weakly in $\L^1$} if the sequence $n\mapsto \mu_n$ is equi-integrable and $\mu_n\rightharpoonup \mu$. 
\end{definition}
\begin{remark}
    Suppose that $\X_n\equiv \X$ and the projections (for fibration) are given by the identity map. Then, the $\L^1$-weak convergence~$\mu_n\overset{\L^1}\rightharpoonup \mu$ holds if and only if the following two conditions hold:
    \begin{itemize}
        \item 
the total variation of the singular part of $\mu_n$ with respect to $\mm_n$ converges to $0$;
\item the Radon--Nikodym~density of the regular part of $\mu_n$ with respect to $\mm_n$ converges weakly in $\L^1$ to the Radon--Nikodym~density of $\mu \ll \mm$, which exists by the equi-integrability hypothesis. 
    \end{itemize}
\end{remark}

\paragraph{Characterisations of strong convergences.}
All of the results that are presented here without a proof have already been proven in \cite[Section 3]{gigliStabilityHeatFlow2024} and \cite[Section 2.2]{vinciniVariationalProblemsNonsmooth2025}.
\begin{proposition}\label{proposition: L^0 properties}
    Let $n\mapsto (\X_n, \mm_n)$ be a sequence fibrating over $(\X, \mm)$ and let $\Z$ be a Polish space.
    Let moreover $n\mapsto T_n\in \L^0(\X_n;\Z)$ be a sequence and $T\in \L^0(\X;\Z)$.
    Then the following three conditions are equivalent:
    \begin{itemize}
        \item $T_n\to T$ in $\L^0$;
        \item for a sequence of good couplings $n\mapsto\aalpha_n$, \begin{equation}
            \sfd^{\saalpha_n}_{\L^0}(T_n, T)\to 0;
        \end{equation}
        \item for all sequences of good couplings $n\mapsto\aalpha_n$, \begin{equation}
            \sfd^{\saalpha_n}_{\L^0}(T_n, T)\to 0;
        \end{equation}
        \item for any equi-integrable sequence $n\mapsto \mu_n$ converging weakly to $\mu$, $(T_n)_\#\mu_n\rightharpoonup T_\#\mu$.
    \end{itemize}
    Moreover, if $\Z'$ be another Polish space and $G:\Z\to\Z'$ is a continuous function, $T_n\to T$ in $\L^0$ implies $G\circ T_n\to G\circ T$ in $\L^0$.
\end{proposition}
\begin{proof}
The equivalence between the first item, the second and the third is the content of \cite[Corollary 3.11]{gigliStabilityHeatFlow2024}. 
The equivalence between the first item and the fourth is the content of \cite[Proposition 3.27]{gigliStabilityHeatFlow2024}.
The last part of the statement follows from the continuity with respect to weak convergence of the pushforward map $G_\#:\pr (\Z)\to \pr(\Z')$ when $G$ is continuous.
\end{proof}
\begin{corollary}\label{cor: approximation via projections}
    Let $n\mapsto (\X_n, \mm_n)$ be a sequence fibrating over $(\X, \mm)$ with projections $n\mapsto p_n$.
    Then $p_n\to id$ in $\L^0$.
    Moreover, if $\Z$ is a Polish spaces and $g:\X\to \Z$ is a continuous function, then $g\circ p_n\to g$ in $\L^0$.
\end{corollary}
\begin{proof}
    Let us start by proving that $p_n\to id$ in $\L^0$.
    This follows at once from the third characterisation in \Cref{proposition: L^0 properties}: by definition of weak convergence, if $\mu_n\rightharpoonup\mu$, then $(p_n)_\#\mu_n\rightharpoonup\mu=(id)_\#\mu$. 
    Finally, if $g$ is as above, by \Cref{proposition: L^0 properties} $g\circ p_n\to g\circ id=g$ in $\L^0$.
\end{proof}
A similar characterization holds for the strong $\L^p$ convergences.
\begin{proposition}[{\cite[Proposition 3.23]{gigliStabilityHeatFlow2024}}]
    Let $n\mapsto (\X_n, \mm_n)$ be a sequence fibrating over~$(\X, \mm)$, $n\mapsto f_n\in \L^p(\X_n)$ be a sequence of functions and $f\in \L^p(\X)$, where $p\in (1, +\infty)$.
    Then the following are equivalent:
    \begin{itemize}
        \item $f_n\to f$ in $\L^p$;
        \item for a sequence of good couplings $n\mapsto\aalpha_n$, 
        \begin{equation}
            \|f_n\circ\proj_1-f\circ\proj_2\|_{\L^p(\saalpha_n)}\to 0
            ;
        \end{equation}
        \item $f_n\to f$ in $\L^0$ and $|f_n|^p\mm_{\X_n}$ is equi-integrable.
    \end{itemize}
\end{proposition}
Finally, one can also define convergence in measure for functions defined on a fibrating sequence whose target is another fibrating sequence.
\begin{definition}
    Let $n\mapsto (\X_n, \mm_n)$ be a sequence fibrating over~$(\X, \mm)$ and let $n\mapsto (\Y_n, \mm_n)$ be a sequence fibrating over~$(\Y, \mm)$ with projections $n\mapsto p_n^\Y$.
    Let $n\mapsto T_n:\X_n\to \Y_n$ be a sequence of Borel maps such that $(T_n)_\#\mm_{\X_n}$ is equi-integrable.
    Finally, let $T:\X\to\Y$ be a Borel map.
    Then we say that $T_n$ converges to $T$ in $\L^0$ (or in measure) if $p_n^\Y\circ T_n \to T$ in $\L^0$.
\end{definition}
\begin{remark}
    Notice that the convergence in measure for functions on varying spaces does not encompass the one for functions defined on a single space as a particular case unless  the functions have equibounded compression.
\end{remark}
\begin{proposition}[{\cite[Proposition 3.29]{gigliStabilityHeatFlow2024}}]\label{proposition: L0 to fibration}
    Let $n\mapsto (\X_n, \mm_n)$ be a sequence fibrating over a space $(\X, \mm)$ and let $n\mapsto (\Y_n, \mm_n)$ be another sequence fibrating over another space $(\Y, \mm)$ with projections $n\mapsto p_n^\Y$.
    Let $n\mapsto T_n:\X_n\to \Y_n$ be a sequence of Borel maps such that $(T_n)_\#\mm_{\X_n}$ is equi-integrable.
    Finally, let $T:\X\to\Y$ be a Borel map.
    Then the following are equivalent:
    \begin{itemize}
        \item $T_n\to T$ in $\L^0$;
        \item for all equi-integrable sequences $n\mapsto \mu_n\in \pr(\X_n)$ with $\mu_n\rightharpoonup\mu$, it holds $(T_n)_\#\mu_n\rightharpoonup T_\#\mu$.
    \end{itemize}
\end{proposition}
As mentioned before, fibrations allow to express extrinsically multiple kind of convergences.
For instance, the following is a reformulation of the well known identity between $\square$-convergence and mGH-convergence.
\begin{definition}[{\cite[Definition 4.21]{Sh16}}] \label{d:eiso}
    Let $\X, \Y$ be mm-spaces, let $\varphi:\X\to\Y$ be a Borel map and $\varepsilon>0$.
    The map $\varphi$ is said to be an $\varepsilon$-mm-isomorphism if there exists $E\subset \X$ such that the following hold:
    \begin{itemize}
        \item $\mm_\X(E)\geq 1-\varepsilon$;
        \item $|\sfd_\X-\varphi^*\sfd_\Y|\leq \varepsilon$ on $E\times E$;
        \item $\sfd_{\sf P}(\varphi_\#\mm_\X, \mm_\Y)\leq \varepsilon$.
    \end{itemize}
\end{definition}
\begin{proposition}\label{prop: mGH fibration}
    Let $n\mapsto \X_n$ be a sequence of mm-spaces and $\X$ be another mm-space.
    Then the following are equivalent:
    \begin{itemize}
        \item $\X_n$ converges to $\X$ in the $\square$-topology;
        \item there exists $\X'$ mm-isomorphic to $\X$ such that $n\mapsto \X_n$ fibrates over $\X'$ with projections $n\mapsto p_n$ and the projections $p_n$ are isometries;
        \item the sequence $n\mapsto \X_n$ fibrates over $\X$ with projections $n\mapsto q_n$ and there is a vanishing sequence $n\mapsto\varepsilon_n>0$ such that $q_n$ is a $\varepsilon_n$-mm-isomorphism for $n\in \N$.
    \end{itemize}
\end{proposition}
\begin{proof}
    It is a well-known fact (see \cite[Theorem 3.15]{GigMonSav15} for a far more general statement) that a sequence $(\X_n)_{n\in \N}$ of spaces converges in the $\square$-topology to a limit $\X$ if and only if there exists a larger complete and separable metric space $\X'$, a sequence of isometric embeddings $p_n:\X_n\to \X$ and a limit isometric embedding $\varphi:\X\to \X'$ such that $(p_n)_\#\mm_{\X_n}\rightharpoonup \varphi_\# \mm_{\X}$.
    Notice that the space endowed with the measure $\varphi_\# \mm_{\X}$ is mm-isomorphic to $\X$.
    This proves the equivalence between the first and second item.

    By \cite[Lemma 4.22]{Sh16}, a sequence $(\X_n)_{n\in\N}$ converges in the $\square$-topology to an mm-space $\X$ if and only if there exists a sequence of $\varepsilon_n$-mm-isomorphism $q_n\X_n\to \X$ for a vanishing sequence $n\mapsto\varepsilon_n>0$.
    By the definition of $\varepsilon_n$-mm-isomorphism, such a sequence $n\mapsto q_n$ satisfies $(q_n)_\#\mm_{\X_n}\rightharpoonup \mm_\X$.
    This proves the equivalence between the first and the third item.
\end{proof}
Similarly, fibrating sequences allow to read convergence in concentration as a condition on the convergence in measure of 1-Lipschitz functions.    
\begin{proposition}[{\cite[Proposition 2.2.40]{vinciniVariationalProblemsNonsmooth2025}}, see also \cite{Gro06} and \cite{Sh16}]\label{proposition: definition via fibration}
    Let $n\mapsto \X_n$ be a sequence of mm-spaces and $\X$ be another mm-space.
    Then $\X_n\to \X$ in concentration if and only if it fibrates on $\X$ with projections $(p_n)_{n\in \N}$ satisfying the following two properties:
    \begin{itemize}
        \item for each $f\in \Lip_1(\X, \mm_\X)$ there exists a sequence $n\mapsto f_n\in \Lip_1(\X_n, \mm_{\X_n})$ such that $f_n\to f$ in $\L^0$;
        \item if $n\mapsto f_n\in\Lip_1(\X_n, \mm_{\X_n})$ is a sequence with the Lévy mean of $f_n$ constantly equal to 0, then every subsequence admits a further subsequence converging in $\L^0$ to some $f\in \Lip_1(\X, \mm_\X)$ -- possibly dependent on the subsequence.
    \end{itemize}
\end{proposition}
\begin{remark}\label{remark: box implies conc}
    It is a well known fact that convergence in concentration is weaker than $\square$-convergence.
    This can be read in the fact that the sequence $n\mapsto q_n$ in \Cref{prop: mGH fibration} satisfy the properties in \Cref{proposition: definition via fibration}.
\end{remark}
Finally, while in general the convergences defined above depend strongly on the choice of projections $p_n$, this is not the case when the spaces converge in concentration -- up to the action of the group of automorphisms.
Thanks to \Cref{proposition: L0 to fibration}, that can be read through the convergence of the projections.
\begin{definition}
    Let $n\mapsto \X_n$ be a sequence of mm-spaces fibrating over an mm-space $\X$ with projections $n\mapsto p_n$ and fibrating over an mm-space $\Y$ with projections $n\mapsto q_n$.
    We say that $p_n$ and $q_n$ are equivalent up to isomorphisms if there exists a sequence $n\mapsto \Phi_n\in {\rm Iso}(\X;\Y)$ and good couplings $n\mapsto \aalpha_n$ with respect to $p_n$ such that $\sfd_{\L^0}^{\saalpha_n}(q_n, \Phi_n)\to 0$.
\end{definition}

\begin{proposition}[{\cite[Theorem 4.3]{gigliStabilityHeatFlow2024}}]
    Let $n\mapsto \X_n$ be a sequence of mm-spaces converging in concentration to a limit mm-space $\X$.
    Let $n\mapsto p_n, q_n:\X_n\to\X$ be two sequences of projections inducing fibrations of $\X_n$ over $\X$ and satisfying the assumptions of \Cref{proposition: definition via fibration}.
    Then $p_n$ and $q_n$ are equivalent up to isomorphisms.
\end{proposition}

\subsection{The extrinsic perspective}\label{subsection: extrinsic convergence}
The aim of this subsection is to use the notion of fibration described in \Cref{subsection: fibration} in order to give an extrinsic formulation of  the convergence of pyramids as well as the convergence in concentration for emm-spaces.
The definition of fibration relies on the fact that the pushforward measures converge weakly, hence, the topology of the target space will play a non-trivial role in the construction.
As clear from the construction in the proof of \Cref{proposition: definition via fibration concentrated}, however, this choice of topology on the limit space is equivalent to the selection of a suitable representative in the emm-isomorphism class. 

Let us start by defining the framework which we will use to characterise the convergence in concentration of concentrated emm-spaces.
\begin{definition} \label{d:DOC}
    Let $(\X_n)_{n\in \N}$ be a sequence of emm-spaces and let $\Y$ be an extended topological-metric-measure space.
    Assume that $n\mapsto (\X_n, \tau_{\X_n}, \mm_{\X_n})$ fibrates over $(\Y, \tau_\Y, \mm_\Y)$ with projections $n\mapsto p_n$.
    We say that the projections $(p_n)_{n\in \N}$
    \begin{itemize}
        \item induce a {\it domination} if for each $f\in \Lip_1(\Y, \mm_\Y)$ there exists a sequence $n\mapsto f_n\in \Lip_1(\X_n, \mm_{\X_n})$ such that $f_n\to f$ in $\L^0$. 
        \item induce a {\it concentration} if they induce a domination and whenever $n\mapsto f_n\in \Lip_1(\X_n, \mm_{\X_n})$ is such that Lévy mean of $f_n$ is constantly 0, every subsequence of $(f_n)_{n\in \N}$ admits a further subsequence converging in $\L^0$ to some $f\in \Lip_1(\Y, \mm_\Y)$ -- possibly dependent on the subsequence.
    \end{itemize}
\end{definition}
Projections that induce a domination gain automatically some regularity.
\begin{lemma}\label{lemma: regularity projections}
    Let $n\mapsto \X_n$ be a sequence of emm-spaces which fibrates over an extended topological-metric-measure space $\X$ via projections $n\mapsto p_n$ inducing a domination.
    Let $\varphi_1, \ldots, \varphi_k\in \Lip_1(\X, \tau_\X, \sfd_\X)$ be bounded and let $\sfd_k(x, y)=\sup_{j=1, \ldots, k} |\varphi_j(x)-\varphi_j(y)|$.
    Then there exists a sequence $\varepsilon_n\to 0$ such that the following holds: for all $n\in \N$ there exists $E_n\subset \X_n$ with $\mm(E_n)\geq 1-\varepsilon_n$ such that $\sfd_k(p_n(x), p_n(y))\leq \sfd_{\X_n}(x, y)+\varepsilon_n$ for all $x, y\in E_n$.
\end{lemma}
\begin{proof}
    Consider the 1-Lipschitz and continuous map $\Phi:\X\to (\mathbb R^k, \ell_\infty)$ defined by $\Phi(x)=(\varphi_1(x), \ldots, \varphi_k(x))$.
    By continuity of $\Phi$ and \Cref{cor: approximation via projections}, $\Phi\circ p_n\to \Phi$ in $\L^0$. 
    
    Since $\Phi$ is 1-Lipschitz and the projections induce a domination, there is a sequence of 1-Lipschitz maps $\Phi_n:\X_n\to (\mathbb R^k, \ell_\infty)$ satisfying $\Phi_n\to \Phi$ in $\L^0$.
    Indeed, by the definition of domination, there exist $(\varphi_{n,1})_{n\in \N}, \ldots, (\varphi_{n,k})_{n\in \N}$ such that $\varphi_{n,j}\in \Lip_1(\X_n, \mm_{\X_n})$ for all $n$ and $j$ and moreover $\varphi_{n,j}\to\varphi_j$ in $\L^0$.
    By defining $\Phi_n=(\varphi_{n, 1}, \ldots, \varphi_{n, j})$, we get that $\Phi_n$ is 1-Lipschitz and $\Phi_n\to \Phi$ in~$\L^0$. 

    Let now $n\mapsto\aalpha_n$ be a good coupling with respect to the fibration.
    By \Cref{remark: triangle inequality},
    \begin{align}
       \frac 12 \varepsilon_n &\eqqcolon\sfd_{\L^0(\mm_{\X_n})}(\Phi\circ p_n, \Phi_n)\\ &\leq \sfd_{\L^0}^{\saalpha_n}(\Phi\circ p_n, \Phi)+\sfd_{\L^0}^{\saalpha_n}(\Phi_n, \Phi)\to 0,
    \end{align} 
    where the first term tends to zero by \Cref{proposition: L^0 properties}, together with the fact that $\Phi\circ p_n \to \Phi$ in $\L^0$, which follows from the definition of the fibration $p_n$. The second term tends to zero by \Cref{proposition: L^0 properties}, together with the convergence $\Phi_n \to \Phi$ in $\L^0$ established above.
    
    By the definition of the $\L^0$-distance $\sfd_{\L^0(\mm_{\X_n})}$, there exists a sequence of sets $E_n\subset \X_n$ such that $\mm_{\X_n}(E_n)\geq 1-\frac 12 \varepsilon_n$ and
    \begin{equation}
        \ell_\infty(\Phi(x),\Phi_n(x))\leq \frac 12 \varepsilon_n, \qquad x \in E_n.
    \end{equation}
    In particular, if $x, y\in E_n$, 
    \begin{equation}
        \sfd_k(x, y)=\ell_\infty(\Phi(x), \Phi(y))\leq \varepsilon+\ell_\infty(\Phi_n(x), \Phi_n(y))\leq \varepsilon +\sfd_{\X_n}(x, y),
    \end{equation}
    which is what we wanted to prove.
\end{proof}
Our first goal is to prove the following two results.
\begin{theorem}\label{proposition: definition via fibration concentrated}
    Let $n\mapsto \X_n$ be a sequence of concentrated emm-spaces. Then $n\mapsto \X_n$ converges in concentration to a limit emm-space $\Y$ if and only if there exists an extended topological-metric-measure space $\tilde{\Y}$ isomorphic to $\Y$ such that $n\mapsto \X_n$ fibrates over $\tilde\Y$ with projections $n\mapsto p_n$ inducing concentration.
\end{theorem}
In the following theorem, $K-\liminf_n \overline{\mathcal P_{\X_n}}^\square$ denotes the inferior limit in the sense of Painlev\'e--Kuratowski, viz.
    $$ K-\liminf_n \overline{\mathcal P_{\X_n}}^\square:=\{\Z\in\mathcal \X: \exists n\mapsto \Z_n\in \overline{\mathcal P_{\X_n}}^\square, \Z_n\xrightarrow{\square}\Z\}.$$ 
\begin{theorem}\label{proposition: general projections}
    Let $n\mapsto \X_n$ be a sequence of emm-spaces and let $\Y$ be an emm-space.
    Then $$\overline{\mathcal P_\X}^\square\subset K-\liminf_n \overline{\mathcal P_{\X_n}}^\square$$ if and only if there exists an extended topological-metric-measure space $\tilde{\X}$ isomorphic to $\X$ such that $\X_n$ fibrates over $\tilde{\X}$ with projections $p_n$ inducing a domination. 
\end{theorem}
\begin{remark} \ 
\begin{itemize}
    \item 
    The set $K-\liminf_n \overline{\mathcal P_{\X_n}}^\square$  can be alternatively given as the intersection of all of the limit pyramids of the sequence $n\mapsto \overline{\mathcal P_{\X_n}}^\square$.
    \item Similarly, one can prove via a diagonal argument that $\overline{\mathcal P_\X}^\square\subset K-\limsup_n \overline{\mathcal P_{\X_n}}^\square$ if and only if there is a subsequence $k\mapsto \X_{n_k}$ such that $\overline{\mathcal P_{\X_{n_k}}}^\square\to \mathcal P'$ and $\mathcal P\subset \mathcal P'$.
    In particular, with a rather similar proof, one can characterise the weaker inclusion 
    $$\overline{\mathcal P_\X}^\square\subset K-\limsup_n \overline{\mathcal P_{\X_n}}^\square$$ by requiring that some subsequence $k\mapsto \X_{n_k}$ fibrates over $\tilde\X$ with projections $p_n$ inducing a domination.
    \end{itemize}
\end{remark}

The idea of the proofs of both~\Cref{proposition: definition via fibration concentrated} and~\Cref{proposition: general projections}  rests on a double approximation.
On one hand, for each emm-space $\Y$ we have an inverse limit $\Y'=\varprojlim \Y_k$ which is isomorphic to $\Y$: this comes with a sequence of mm-spaces $\Y_k\in\mathcal P_\Y$.
On the other hand, in both of the settings of \Cref{proposition: definition via fibration concentrated} and \Cref{proposition: general projections}, there exists sequences $(\X_{n, k})_{n\in \N}$ such that $\X_{n, k}\prec \X_n$ and $\X_{n, k}\to \Y_k$ in the $\square$-topology.
What we will do is to choose suitable projections from $\X_n$ to $\X_{n, k}$, from $\X_{n, k}$ to $\Y_k$ and from $\Y_k$ to $\Y$ which we will then concatenate.

\begin{lemma}\label{lemma: fibration inverse limit}
    Let $\Y=\varprojlim\Y_k$ be an emm-space and let $\iota$ be an adapted Kuratowski embedding as in \Cref{def:KE}.
    Then $\Y_k$ fibrates over $\tilde Y=(\mathbb R^\infty, \tau_\infty, \ell_\infty, \iota_\#\mm_\Y)$ with continuous projections inducing a domination. 
    If $\Y$ is concentrated, the projections can be chosen to induce concentration.
\end{lemma}
\begin{proof}
    Let us follow the notation of \Cref{prop: embedding}, namely let $k\mapsto \sfp_k$ be the projections granted by the inverse system structure, $k\mapsto I_k\subset\mathbb N$ be a sequence of index sets granted by \Cref{prop: embedding}, $k\mapsto \proj_k:\mathbb R^\infty\to \mathbb R^\infty$ be the projections map on the coordinate hyperplanes defined by $I_k$, i.e., the sets whose nonzero entries are only the ones in $I_k$. Let also $\iota: \Y\to \mathbb R^\infty$ be the adapted Kuratowski embedding and $k\mapsto \iota_k:\Y_k\to \mathbb R^\infty$ be the sequence of maps granted by \Cref{prop: embedding}.
    
    First, $(\iota_k)_\#\mm_{\Y_k}=(\iota_k\circ \sfp_k)_\#\mm_{\Y}=(\proj_k\circ\iota)_\#\mm_\Y$.
    Now, since $I_k\nearrow\mathbb N$ in \Cref{prop: embedding}, $\proj_k\to id$ pointwise in $\tau_\Y$ and $\iota$ is a $\tau_Y$-continuous map, 
    $(\iota_k)_\# \mm_{\Y_k} \rightharpoonup \iota_\# \mm_\Y$.
    This implies that the sequence $k\mapsto \Y_k$ fibrates over $\Y$ with projections $k\mapsto \iota_k$.
    
    For the same reasons, notice that $\aalpha_k=(\sfp_k, \iota)_\#\mm_Y$ is a sequence of good couplings with respect to the fibration introduced above.
    Indeed, $k\mapsto (\sfp_k, \iota)$ converges pointwise to $(\iota, \iota)$ and $(\iota, \iota)_\#\mm_\Y=(id,id)_\#\iota_\# \mm_\Y$.

    Let now $f\in \Lip_1(\tilde Y, \iota_\#\mm_\Y)$.
    By the same proof as in \Cref{proposition: LAP}, there exists a sequence $k\mapsto g_k$ of 1-Lipschitz cylinder functions such that $g_k\to f$ in $\L^0$.
    But by \Cref{prop: embedding}, up to appropriately relabeling the sequence, we can consider the sequence $k\mapsto f_k=g_k\circ\iota_k\in \Lip_1(\Y_k, \mm_{\Y_k})$.
    For this sequence, 
    \begin{align}
        \sfd_{\L^0}^{\saalpha_k}(f_k, f)&=\sfd_{\L^0(\mm_\Y)}(g_k\circ(\iota_k\circ \sfp_k), f\circ\iota)\\
        &= \sfd_{\L^0(\mm_\Y)}(g_k\circ\iota, f\circ\iota)\\
        &=\sfd_{\L^0(\iota_\#\mm\Y)}(g_k, f)\to 0.
    \end{align}
    In particular, the projections $k\mapsto \iota_k$ induce a domination.

    Finally, assume that $\Y$ is concentrated and let $k\mapsto f_k\in\Lip_1(\Y_k, \mm_{\Y_k})$ be a sequence of functions with Lévy mean equal to 0. 
    Consider the sequence $k\mapsto \tilde g_k=f_k\circ \iota_k^{-1}$ on the domain of definition of $\iota_k^{-1}$.
    It is easy to see that there exist 1-Lipschitz cylinder functions $g_k\in \Lip_1(\tilde \Y, \iota_\#\mm_\Y)$ extending $\tilde g_k$ for $k\in \N$.
    Moreover, since $g_k\circ \proj_k=g_k$ and $(\iota_k)_\#\mm_{\Y_k}=(\proj_k)_\#\iota_\#\mm_\Y$, also the functions $g_k$ have Lévy mean equal to 0.
    
    By the fact that $\Y$ is concentrated, the sequence $k\mapsto g_k\circ\iota$ is pre-compact in $\L^0(\Y)$ with limit points in $\Lip_1(\Y, \mm_\Y)$.
    In particular, $k\mapsto g_k$ is pre-compact in $\L^0(\iota_\#\mm_\Y)$ with limit points in $\Lip_1(\tilde Y, \iota_\#\mm_\Y)$.
    This proves that the projections $k\mapsto \iota_k$ induce concentration. 
\end{proof}

\begin{corollary}\label{cor: intermediate maps}
    Let $\X=\varprojlim\Y_k$ be an emm-space.
    For all $k\in \N$, let $(\X_{n,k})_{n\in \N}$ be a sequence of mm-spaces such that $\X_{n, k}\to \Y_k$ in the $\square$-topology as $n\to \infty$.
    Let $\iota$ be an adapted Kuratowski embedding for $\X=\varprojlim \Y_k$.
    Then there exists a non-decreasing diverging sequence $k_n$ such that $\X_{n, k_n}$ fibrates over $\tilde\Y=(\mathbb R^\infty, \tau_\infty, \iota_\#\mm_\X)$ with projections which induce a domination.
    If $\X$ is concentrated, the projections can be chosen to induce concentration.
\end{corollary}
\begin{proof}
    By \Cref{prop: mGH fibration}, for each $k\in \N$ the sequence $n\mapsto \X_{n, k}$ fibrates over $\Y_k$ with projections $n\mapsto \varphi_{n, k}$ that are $\varepsilon_{n, k}$-mm-isomorphisms (recall \Cref{d:eiso}) for some vanishing sequence $n\mapsto \varepsilon_{n, k}$.
    For all $k\in \N$, let $n\mapsto \bbeta_{n, k}\in \adm(\mm_{\X_{n, k}}, \mm_{\Y_k})$ be good couplings with respect to the fibration induced by $n\mapsto \varphi_{n, k}$.

    By \Cref{lemma: fibration inverse limit}, $\Y_k$ fibrates over $\tilde Y=(\mathbb R^\infty, \tau^\infty, \ell_\infty, \iota_\#\mm_\Y)$ with continuous projections $k\mapsto p_k$ inducing a domination (concentration if $\X$ is concentrated). 
    For simplicity of notation, let us write $\mm_{\tilde\Y}=\iota_\#\mm_\Y$. 
    Let $k\mapsto \aalpha_k\in \adm(\mm_{\Y_k}, \mm_{\tilde\Y})$ be good couplings with respect to the fibration induced by $k\mapsto p_k$.

    Let us now define the projections $q_{n, k}: \X_{n, k}\to \tilde\Y$ as $q_{n, k}=p_k\circ \varphi_{n, k}$ for $n, k\in \N$.
    For $n, k\in \N$, let also $\tilde \ggamma_{n, k}\in \pr(\X_{n, k}\times \Y_k\times \tilde\Y)$ be a gluing of $\bbeta_{n, k}$ and $\aalpha_{k}$, meaning that $({\sf proj}_{1, 2})_\# \tilde \ggamma_{n, k}= \bbeta_{n, k}$ and $({\sf proj}_{2,3})_\# \tilde \ggamma_{n, k}= \alpha_k$ (see e.g.~\cite[p.11]{Vil09}).
    Finally, let $\ggamma_{n, k}=({\sf proj}_{1, 3})_\#\tilde \ggamma_{n, k}$.

    We want to prove that a suitable sequence $n\mapsto \X_{n, k_n}$ fibrates over $\tilde\Y$ via the projections $n\mapsto q_{n, k_{n}}$, that $n\mapsto \ggamma_{n, k_n}$ are good couplings and that the projections $n\mapsto q_{n, k_{n}}$ induce a domination (concentration if $\X$ is concentrated).

    Let us start by listing some properties of the objects that we just defined.
    \begin{itemize}
        \item Given $k\in \N$, by the continuity of $p_k$, we get, in the limit as $n\to \infty$,
        \begin{equation}
            (q_{n, k})_{\#}\mm_{\X_{n, k}}= (p_k)_\# (\varphi_{n, k})_\#\mm_{\Y_k}\rightharpoonup (p_k)_\#\mm_{\Y_k};
        \end{equation}
        \item For all $n, k\in \N$, 
        \begin{align}
            (q_{n, k}\times id)_{\#}\ggamma_{n, k}&=({\sf proj}_{1, 3})_\#(q_{n, k}\times p_k\times id)_{\#}\tilde \ggamma_{n, k}\\
            &=({\sf proj}_{1, 3})_\#(p_k\times p_k\times id)_{\#}(\varphi_{n, k}\times id\times id)_\#\tilde \ggamma_{n, k}.
        \end{align}
        Let $k\in \N$ and consider any limit point $\ppi^3$ as $n\to\infty$ of the sequence $(p_k\times p_k\times id)_{\#}(\varphi_{n, k}\times id\times id)_\#\tilde \ggamma_{n, k}$ and notice that, by the continuity of $p_k$, it satisfies 
        \begin{equation}
            ({\sf proj}_{1, 2})_\# \ppi^3=(p_k\times p_k)_\#(id, id)_\# \mm_{\Y_k}=(p_k, p_k)_\#\Y_k.
        \end{equation}
        This in particular implies that
        \begin{equation}
            ({\sf proj}_{1, 3})_\# \ppi^3=({\sf proj}_{2, 3})_\#\ppi^3=(p_k\times id)_\#\alpha_k.
        \end{equation}
        By the computations above, $(q_{n, k}\times id)_{\#}\ggamma_{n, k}\rightharpoonup (p_k\times id)_\#\aalpha_k$ as $n\to \infty$.
        \item By \Cref{remark: box implies conc}, for all $k\in \N$ the projections $n\mapsto\varphi_{n, k}$ induce concentration.
        In particular, this implies that the quantity 
        \begin{align}
            \sfd_{\L^0}^{\sbbeta_{n, k}, {\sf H}}&(\Lip_1(\X_{n, k}, \mm_{\X_{n, k}}), \Lip_1(\Y_k, \mm_{\Y_k}))\\
            &=\max\{\max_{f\in \Lip_1(\X_{n, k})} \min_{g\in \Lip_1(\Y_k)}\sfd_{\L^0}^{\sbbeta_{n, k}}(f, g), \max_{ g\in \Lip_1(\Y_k)} \min_{f\in \Lip_1(\X_{n, k})}\sfd_{\L^0}^{\sbbeta_{n, k}}(f, g)   \}
        \end{align}
        satisfies $\sfd_{\L^0}^{\sbbeta_{n, k}, {\sf H}}(\Lip_1(\X_{n, k}, \mm_{\X_{n, k}}), \Lip_1(\Y_k, \mm_{\Y_k}))\to 0$ as $n \to \infty$;
        \item In a similar fashion, since the projections $k\mapsto p_k$ induce a domination, we have that
        \begin{equation}
            \max_{g\in \Lip_1(\tilde\Y, \iota_\#\mm_\Y)}\min_{f\in \Lip_1(\Y_k, \mm_{\Y_k})} \sfd^{\saalpha_k}_{\L^0}(f, g)\to 0
        \end{equation}
        as $k\to +\infty$.
        \item By the previous two bullet points and the definition of $\sggamma_{n, k}$, 
        \begin{align}
            \max_{g\in \Lip_1(\tilde\Y, \iota_\#\mm_\Y)}&\min_{f\in \Lip_1(\X_{n, k}, \mm_{\X_{n, k}})}\sfd_{\L^0}^{\sggamma_{n, k}}(f, g)\\
            &\leq \sfd_{\L^0}^{\sbbeta_{n, k}, {\sf H}}(\Lip_1(\X_{n, k}, \mm_{\X_{n, k}}), \Lip_1(\Y_k, \mm_{\Y_k}))\\
            &\qquad+ \max_{g\in \Lip_1(\tilde\Y, \iota_\#\mm_\Y)}\min_{f\in \Lip_1(\Y_k, \mm_{\Y_k})} \sfd^{\saalpha_k}_{\L^0}(f, g).
        \end{align}
    \end{itemize}
    Now, by a diagonal argument we can find a diverging and non-decreasing sequence $k_n$ such that the following points hold.
    \begin{enumerate}
        \item \label{it: fibration} The sequence of spaces $n\mapsto \X_{n, k_n}$ fibrates over $\tilde \Y$ with projections $q_n$, i.e.~$(q_{n, k_{n}})_{\#}\mm_{\X_{n, k}}\rightharpoonup \mm_{\tilde\Y}$ as $n\to \infty$;
        \item \label{it: good coupling} The sequence $n\mapsto \ggamma_{n, k_n}$ is a sequence of good couplings, meaning $(q_{n, k_{n}}\times  id)_{\#}\ggamma_{n, k}\rightharpoonup (id, id)_\#\mm_{\Y}$ as $n\to \infty$;
        \item \label{it: induce a domination} The projections $n\mapsto q_{n, k_n}$ induce a domination, thanks to 
        \begin{equation}
            \max_{g\in \Lip_1(\tilde\Y, \iota_\#\mm_\Y)}\min_{f\in \Lip_1(\X_{n, k}, \mm_{\X_{n, k}})}\sfd_{\L^0}^{\sggamma_{n, k}}(f, g)\to 0
        \end{equation}
        as $n\to \infty$.
    \end{enumerate}
    By \cref{it: fibration,it: good coupling,it: induce a domination}, the statement is proved.

    If $\X$ is also concentrated, the last two bullet points above can be strengthened to
    \begin{itemize}
        \item Since the projections $k\mapsto p_k$ induce concentration, we have $$\sfd_{\L^0}^{\saalpha_{k}, {\sf H}}(\Lip_1(\Y_k, \mm_{\Y_k}), \Lip_1(\tilde\Y, \iota_\#\mm_\Y))\to 0 \quad \text{as} \quad  k \to \infty \ ;$$
        \item By the previous bullet points and the definition of $\sggamma_{n, k}$,
        \begin{align}
            \sfd_{\L^0}^{\sggamma_{n, k}, {\sf H}}&(\Lip_1(\X_{n, k}, \mm_{\X_{n, k}}), \Lip_1(\tilde\Y, \iota_\#\mm_\Y))\\
            &\leq \sfd_{\L^0}^{\sbbeta_{n, k}, {\sf H}}(\Lip_1(\X_{n, k}, \mm_{\X_{n, k}}), \Lip_1(\Y_k, \mm_{\Y_k}))+ \sfd_{\L^0}^{\saalpha_{k}, {\sf H}}(\Lip_1(\Y_{k}, \mm_{\Y_k}), \Lip_1(\tilde\Y, \iota_\#\mm_\Y)).
        \end{align}
    \end{itemize}
    Now, by a diagonal argument we can find again a diverging and non-decreasing sequence $k_n$ such that \cref{it: fibration,it: good coupling} hold.
    Moreover, \cref{it: induce a domination} can be strengthened to
    \begin{enumerate}
    \setcounter{enumi}{3}
        \item\label{it: induce convergence in concentration} The projections $n\mapsto q_{n, k_n}$ induce concentration, thanks to 
        \begin{equation}
            \sfd_{\L^0}^{\sggamma_{n, k_n}}(\Lip_1(\X_{n, k_n}, \mm_{\X_{n, k_n}}), \Lip_1(\Y, \mm_\Y))\to 0
        \end{equation}
        as $n\to \infty$.
    \end{enumerate}
    By \cref{it: fibration,it: good coupling,it: induce convergence in concentration}, the statement is proved.
\end{proof}
In order to prove \Cref{proposition: definition via fibration concentrated}, we need the following simple statement, linking Gromov convergence and Gromov--Hausdorff convergence.
\begin{lemma}\label{lemma: Hausdorff convergence via GH}
    Let $n\mapsto \X_n$ be a sequence of metric spaces, $n\mapsto K_n, K_n'\subset \X_n$ be two sequences of compact subsets.
    Assume that the following two assumptions hold:
    \begin{itemize}
        \item there exists a vanishing sequence $n\mapsto \varepsilon_n$ such that $K_n'\subset B(K_n, \varepsilon_n)$ for all $n\in \N$;
        \item there exists a compact metric space $K$ such that both $n\mapsto K_n$ and $n\mapsto K_n'$ converge in the Gromov-Hausdorff sense to $K$.
    \end{itemize}
    Then $\sfd_{\sf H}(K_n, K_n')\to 0$ as $n\to\infty$.
\end{lemma}
\begin{proof}
    By the properties of Gromov--Hausdorff convergence, there is a complete metric space $\Z$, a sequence of isometric embeddings $n\mapsto \varphi_n:K_n\to Z$ and a limit embedding $\varphi:K\to \Z$ such that $\varphi(K_n)$ Hausdorff converges to $\varphi(K)$.
    By a gluing argument, up to enlarging the target space $\Z$, we can also assume that $\varphi_n$ is an isometric embedding defined on the whole of $\X_n$ for all $n\in \N$.
    In particular, $\varphi(K_n')\subset B^\Z(\varphi_n(K_n), \varepsilon_n)$ for all $n\in \N$.
    This in turn implies that $n\mapsto \varphi(K_n')$ is Hausdorff precompact.
    Up to subsequences, we can assume that $n\mapsto \varphi(K_n')$ converges in the Hausdorff topology to some $H\subset \varphi(K)$.
    But now this implies that $K_n'\to H$ in the Gromov-Hausdorff topology, thus $H$ is isometric to $K$ and thus to $\varphi(K)$.
    Finally, since $\varphi(K)$ is compact, it coincides with its only compact subset to which it is isometric.
    This implies
    \begin{equation}
        \sfd_{\X_n}^{\sf H}(K_n, K_n')=\sfd_{\Z}^{\sf H}(\varphi_n(K_n), \varphi_n(K_n'))\to 0.
    \end{equation}
\end{proof}

\begin{proof}[Proof of \Cref{proposition: definition via fibration concentrated}]
    Assume first that there exists $\tilde{\Y}$ and projections $n\mapsto p_n$ inducing concentration.
    Let us prove that $\X_n\to \tilde{\Y}$ (and thus to $\Y$) in concentration.
    Let $n\mapsto \aalpha_n$ be a sequence of good couplings and let us prove that $\sfd_{\sf conc}^{\saalpha_n}(\X_n, \Y)\to 0$ as $n\to\infty$.

    Assume by contradiction that there exists $\varepsilon>0$ and a non-relabeled sequence of $(\X_n)_{n\in \N}$ such that $\sfd_{\sf conc}^{\saalpha_n}(\X_n, \Y)>\varepsilon$ for all $n\in \N$.
    Then there exists a further non-relabeled subsequence such that either of the following two properties hold:
    \begin{enumerate}
        \item there exist $n\mapsto f_n\in \Lip_1(\X_n, \mm_{\X_n})$ such that for all sequences $n\mapsto g_n\in \Lip_1(\tilde{\Y}, \mm_{\tilde\Y})$ it holds $\sfd^{\saalpha_n}_{\L^0}(f_n, g_n)>\varepsilon$;\label{it: contradiction case 1}
        \item there exist $n\mapsto g_n\in \Lip_1(\tilde{\Y}, \mm_{\tilde \Y})$ such that for all sequences $n\mapsto f_n\in \Lip_1(\X_n, \mm_{\X_n})$ it holds $\sfd^{\saalpha_n}_{\L^0}(f_n, g_n)>\varepsilon$.\label{it: contradiction case 2}
    \end{enumerate}
    In the first case there exists a sequence $(c_n)_{n\in \N}$ of real numbers such that the Lévy mean of $f_n-c_n$ is 0 for all $n\in \N$.
    By assumption, there exists $g\in \Lip_1(\tilde\Y, \mm_{\tilde Y})$ such that 
    \begin{equation}
        \liminf_n \sfd^{\saalpha_n}_{\L^0}(f_n, g+c_n)=\liminf_n \sfd^{\saalpha_n}_{\L^0}(f_n-c_n, g)=0.
    \end{equation}
    This contradicts \cref{it: contradiction case 1}.
    
    In the second case, since $\tilde{\Y}$ is concentrated, there exists a sequence $(c_n)_{n\in \N}$ of real numbers and $g\in \Lip_1(\tilde{\Y}, \mm_{\tilde Y})$ such that $\liminf_n\sfd_{\L^0(\mm_{\tilde{\Y}})}(g_n-c_n, g)=0$.
    Moreover, by assumption, there exist $n\mapsto f_n\in \Lip_1(\X_n, \mm_{\X_n})$ such that $f_n\to g$ in $\L^0$, thus 
    \begin{align}
        \sfd^{\saalpha_n}_{\L^0}(f_n+c_n, g_n)& =\sfd^{\saalpha_n}_{\L^0}(f_n, g_n-c_n)\\
        &\leq \sfd^{\saalpha_n}_{\L^0}(f_n, g)+\sfd_{\L^0(\mm_{\tilde{\X}})}(g_n-c_n, g).
    \end{align}
    In particular, $\liminf_n \sfd^{\saalpha_n}_{\L^0}(f_n+c_n, g_n)=0$, contradicting \cref{it: contradiction case 2}.

    Viceversa, assume $\X_n\to \Y$ in concentration and up to emm-isomorphism, assume that $\Y=\varprojlim \Y_k$ for some inverse system $k\mapsto \Y_k$.
    Since convergence in concentration implies the weak convergence of pyramids, for all $k\in \N$ there exist $n\mapsto \X_{n, k}\prec\X_n$ such that $\X_{n, k}\to\Y_k$ in the box metric as $n\to\infty$.
    By \Cref{cor: intermediate maps}, there exists $\tilde\Y$ isomorphic to $\Y$ and a subsequence $(k_n)_{n\in \N}$ and projections $q_{n, k_n}$ such that $(\X_{n, k_n})_{n\in \N}$ fibrates over $\tilde\Y$ with projections $n\mapsto q_{n, k_n}$ with the desired properties.
    
    Let now, for each $n\in \N$, $\psi_n:\X_n\to \X_{n, k_n}$ be the 1-Lipschitz, measure preserving maps granted by the definition of Lipschitz order.
    Define $n\mapsto p_n:\X_n\to \tilde\Y$ as $q_{n, k_n}\circ\psi_n$ and notice that $n\mapsto p_n$ is a sequence of projections inducing a fibration of $n\mapsto \X_n$ over $\tilde\Y$.
    Indeed, 
    \begin{equation}
        (p_n)_\#\mm_{\X_n}= (q_{n, k_n})_\#(\psi_n)_\# \mm_{\X_n}=(q_{n, k_n})_\#\mm_{\X_{n, k_n}}\rightharpoonup\mm_{\tilde\Y}.
    \end{equation}
    
    Let $n\mapsto \aalpha_n$ be a sequence of good couplings for the fibration of $n\mapsto \X_{n, k_n}$ over $\tilde\Y$ and  for each $n\in \N$ let $n\mapsto \tilde\bbeta_n$ be a gluing of $ [(id, \psi_n)_\# \mm_{\X_n}]$ and $\aalpha_n$. 
    Finally, for each $n\in \N$, let $\bbeta_n=({\sf proj}_{1, 3})_\#\tilde\bbeta_n$.
    In particular, it is immediate that $n\mapsto \bbeta_n$ is a sequence of good couplings for the fibration of $n\mapsto \X_n$ over $\tilde\Y$.
    Notice also that, by construction, for all $f\in \L^0(\X_{n, k_n})$ and $g\in \L^0(\tilde\Y)$, $\sfd_{\L^0}^{\saalpha_n}(f, g)=\sfd_{\L^0}^{\sbbeta_n}(f\circ \psi_n, g)$.
    
    Finally, \Cref{lemma: Hausdorff convergence via GH} implies that $$\sfd_{\L^0(\mm_n)}^{\sf H}(\Lip_1(\X_n, \mm_{\X_n}), \psi_n^*\Lip_1(\X_{n, k_n}, \mm_{\X_{n, k_n}}))\to 0.$$
    This yields
    \begin{align}
        &\sfd_{\L^0}^{\saalpha_n, {\sf H}}(\Lip_1(\X_n, \mm_{\X_n}), \Lip_1(\tilde\Y, \iota_\#\mm_\Y))
        \\
        &\leq \sfd_{\L^0}^{\saalpha_n, {\sf H}}(\Lip_1(\X_n, \mm_{\X_n}), \psi_n^*\Lip_1(\X_{n, k_n}, \mm_{\X_{n, k_n}}))\\&\quad +\sfd_{\L^0}^{\sbbeta_n, {\sf H}}(\Lip_1(\X_{n, k_n}, \mm_{\X_{n, k_n}}), \Lip_1(\tilde\Y, \iota_\#\mm_\Y))\\
        &=
        \sfd^{\L^0(\mm_n)}_{\sf H}(\Lip_1(\X_n, \mm_{\X_n}), \psi_n^*\Lip_1(\X_{n, k_n}, \mm_{\X_{n, k_n}}))\\
        &\quad+\sfd_{\L^0}^{\sbbeta_n, {\sf H}}(\Lip_1(\X_{n, k_n}, \mm_{\X_{n, k_n}}), \Lip_1(\tilde\Y, \iota_\#\mm_\Y))\to 0,
    \end{align}
    where $\sfd_{\L^0}^{\saalpha_n, {\sf H}}$ and $\sfd_{\L^0}^{\sbbeta_n, {\sf H}}$ are as in the proof of \Cref{cor: intermediate maps}.
    This concludes the proof.
\end{proof}
The proof of \Cref{proposition: general projections} is similar.

\begin{proof}[Proof of \Cref{proposition: general projections}]
    Let us start by assuming that $n\mapsto \X_n$ fibrates over $\tilde\X$ with a sequence of projections $n\mapsto p_n$ which induces a domination.
    Let us prove that there is a dense subset of $\mathcal P_\X$ which is contained in $K-\liminf_n \overline{\mathcal P_{\X_n}}^\square$.

    Let $N\in \N$ and let $\Y=(\mathbb R^N, \ell_\infty, \mm_Y)\prec\X$ be an mm-space. 
    Let $\Phi=(\varphi^1, \ldots, \varphi^N):\X\to \Y$ be a 1-Lipschitz, measure preserving map.
    Since the projections induce a domination, for $k=1, \ldots, N$ there exists a sequence $n\mapsto \varphi^k_n\in \Lip_1(\X_n, \mm_{\X_n})$ such that $\varphi^k_n\to \varphi^k$ in $\L^0$ as $n\to\infty$.
    In particular, for $n\in \N$ the mm-spaces $\Y_n=(\mathbb R^N, \ell_\infty, (\Phi_n)_\# \mm_{\X_n})$, where $\Phi_n=(\varphi^1_n, \ldots, \varphi^N_n)$ satisfy $\Y_n\prec \X_n$.
    Moreover, $\Phi_n\to \Phi$ in $\L^0$, so $(\Phi_n)_\# \mm_{\X_n}\rightharpoonup \Phi_\#\mm_\X$.
    This implies $\Y_n\to \Y$ in the $\square$-topology and $\Y\in K-\liminf_n \overline{\mathcal P_{\X_n}}^\square $.
    Since the mm-spaces of the form above are dense in $\mathcal \X$, $\mathcal P_\X\subset K-\liminf_n \overline{\mathcal P_{\X_n}}^\square$.

    Viceversa, assume that $\mathcal P_\X\subset K-\liminf_n \overline{\mathcal P_{\X_n}}^\square$.
    By following the same lines as in the proof of \Cref{proposition: definition via fibration concentrated}, we get that $n\mapsto X_n$ fibrates over $\tilde\X$ with projections $n\mapsto p_n$ and good couplings $n\mapsto \bbeta_n$.
    We also get a sequence $n\mapsto\X_{n, k_n}\prec\X_n$ which fibrates over $\tilde \X$ with projections which induce concentration, good couplings $n\mapsto \aalpha_n\in \adm(\mm_{\X_{m, k_n}}, \mm_\X)$ and a sequence of 1-Lipschitz measure preserving maps $n\mapsto \psi_n:\X_n\to\X_{n, k_n}$.
    Moreover, we have that for all $f\in \L^0(\X_{n, k_n})$ and $g\in \L^0(\X)$, $\sfd_{\L^0}^{\saalpha_n}(f, g)=\sfd_{\L^0}^{\sbbeta_n}(f\circ \psi_n, g)$.
    
    Let now $f\in\Lip_1(\X, \mm_{\X})$ and let $n\mapsto g_n\in\Lip_1(\X_{n, k_n}, \mm_{\X_{n, k_n}})$ be such that $g_n\to f$ in $\L^0$.
    Consider $n\mapsto f_n=g_n\circ\psi_n\in \Lip_1(\X_n, \mm_{\X_n})$ and notice that $\sfd_{\L^0}^{\sbbeta_n}(f_n, f)=\sfd_{\L^0}^{\saalpha_n}(g_n, f)\to 0$.
    This concludes the proof.
\end{proof}

While we have proved that the convergence in concentration of concentrated emm-spaces can be realised extrinsically via the use of fibrations, the choice of the fibrations themselves is by no mean unique.
For example, assume that there exists a bicontinuous isometry $\Psi:\tilde\X\to\tilde\X$ which preserves the measure of the limit.
Then, if $n\mapsto p_n$ is a sequence of projections inducing concentration, also $n\mapsto \Psi\circ p_n$ is a sequence of projections inducing concentration.

Notice also that the limits of measures and functions along the fibrations are not canonical: in the example above, if $f_n\to f$ in $\L^0$ (or $\L^p$) with respect to the fibration induced by $n\mapsto p_n$, then $f_n\to f\circ\Psi^{-1}$ with respect to $n\mapsto \Psi_n\circ p_n$.
Similarly, if $n\mapsto \mu_n\in \pr(\X_n)$ converges weakly with respect to the projections $n\mapsto p_n$ to $\mu\in \pr(\tilde \X)$ with e.g.~bounded entropy along the sequence, then it converges weakly to $\Psi_\#\mu$ with respect to $n\mapsto \Psi\circ p_n$.

Our goal for the last part of this subsection is to prove that the non-uniqueness in the choice of the projections (and, thus, in the structure of the convergence of measures and functions) is only up to the action of the automorphisms of the target space.
In order to do so, let us start by studying the isomorphisms of concentrated emm-spaces.
In particular, we can prove that as soon as two extended topological-metric-measure spaces are concentrated, the set of their isomorphisms is compact in the $\L^0(\X;\Y)$ topology, where on the target we consider the $\tau_\Y$ topology.

\begin{proposition}\label{pro: compactness isomorphisms}
    Let $\X, \Y$ be extended topological-metric-measure spaces and concentrated emm-spaces. 
    Then the space ${\rm Iso}(\X; \Y)$ of the isomorphisms from $\X$ to $\Y$ is compact in the $\L^0(\X;\Y)$ topology.
\end{proposition}
\begin{proof}
    Let us start by proving that any sequence $(\Phi_n)\subset {\rm Iso}(\X; \Y)$ admits an $\L^0$ convergent subsequence.
    Let $f\in \Lip_1(\Y, \sfd_\Y)$ and notice that the sequence $(f\circ\Phi_n)_{n\in \N}$ is contained in $\Lip_1(\Y, \mm_{\X_n})$ and such that the common pushforward measure $\mu=f_\#\mm_{\X}=(f\circ\Phi_n)_\#\mm_{\X}$ is tight.
    Since $\X$ is concentrated, $(f\circ\Phi_n)_{n\in \N}$ is precompact in $\L^0$.
    The definition of isomorphism and \Cref{lemma: L^0 compactness of maps} imply that $(\Phi_n)_{n\in \N}$ is precompact in $\L^0$.

    Let us now conclude by proving that ${\rm Iso}(\X; \Y)$ is closed.
    Let then $(\Phi_n)\subset {\rm Iso}(\X; \Y)$ be a converging sequence and let $\Phi$ be its limit.
    Clearly $\Phi$ is measure preserving and 1-Lipschitz.
    By \Cref{cor: characterisation isomorphisms}, it remains to prove that $\Phi^*\Lip_1(\Y, \mm_\Y)$ is dense in $\Lip_1(\X, \mm_{\X})$.
    If $f\in \Lip_1(\X, \mm_{\X})$ is bounded, the fact $\Phi_n\Lip_1(\Y, \mm_{\Y})=\Lip_1(\X, \mm_\X)$ implies that there exist equibounded $g_n\in \Lip_1(\Y, \mm_\Y)$ such that $f=g_n\circ \Phi_n$. 
    Since $\Y$ is concentrated, up to subsequence, $g_n\to g\in \Lip_1(\Y, \mm_\Y)$ in $\L^0$.
    Since $\Phi_n$ are automorphisms, they have bounded compression, thus $f=\lim g_n\circ \Phi_n= g\circ \Phi$.
\end{proof}

We are now ready to discuss the relation between different sequences of projections inducing concentration.
\begin{definition}
    Let $\X$ be an extended topological-metric-measure space and let $\X_n$ be a sequence of emm-spaces.
    Let $p_n$ and $q_n$ be sequences of projections inducing two fibrations of $\X_n$ over $\X$.
    We say that $p_n$ and $q_n$ are equivalent up to isomorphisms if there exist sequences $\Phi_n, \Psi_n\in {\rm Aut}(\X)$ and good couplings $\aalpha_n$ and $\bbeta_n$ with respect to $p_n$ and $q_n$ such that $\sfd_{\L^0}^{\saalpha_n}(q_n, \Phi_n)\to 0$ and $\sfd_{\L^0}^{\sbbeta_n}(p_n, \Psi_n)\to 0$.
\end{definition}
As in the setting of mm-spaces \cite{gigliStabilityHeatFlow2024}, it turns out that all sequences of projections are equivalent up to the action of the isomorphisms.
\begin{proposition}\label{prop: equivalence of projections}
    Let $\X_n$ be a sequence of concentrated emm-spaces converging in concentration to a limit concentrated emm-space $\X$.
    Let $p_n, q_n:\X_n\to\X$ be two sequences of projections inducing fibrations of $\X_n$ over $\X$ and satisfying the assumptions of \Cref{proposition: definition via fibration concentrated}.
    Then $p_n$ and $q_n$ are equivalent up to isomorphisms.
\end{proposition}
\begin{remark}
    The above \Cref{prop: equivalence of projections} provides also a different proof of \Cref{theorem: isomorphism from pyramid}.
    Indeed, the existence of a sequence of projections to a concentrated space (isomorphic to) $\X$ inducing concentration due to \Cref{proposition: definition via fibration concentrated} depends only on the fact that $\mathcal P_\X=\lim_n \mathcal P_{\X_n}$.
    The statement above then implies that any couple of concentrated spaces with the same pyramid must be isomorphic.
\end{remark}
The proof of \Cref{prop: equivalence of projections} rests on a compactness result resembling \Cref{pro: compactness isomorphisms}. 
\begin{lemma}\label{lemma: L^0 compactness of maps fibration}
    Let $\X_n$ be a sequence of emm-spaces fibrating on a Polish space $\X$.
    Let $(f_k)_{k\in \N}$ be a sequence of continuous maps generating the topology of $\Y$.
    Let $(T_n)_{n\in \N}$ be a sequence of measurable maps from $\X_n$ to $\Y$ such that $(f_k\circ T_n)_{n\in \N}$ is precompact in $\L^0$ for all $k$ and that $(T_n)_{\#}\mm_{\X_n}$ is tight.
    Then $T_n$ is precompact in $\L^0$. 
\end{lemma}
The proof is identical to the one of \Cref{lemma: L^0 compactness of maps}, just replacing the convergence in $\L^0(\X;\Y)$ with the convergence in $\L^0$ along the fibration.

\begin{corollary}\label{cor: compactness projections}
    Let $\X$ and $\Y$ be concentrated extended topological-metric-measure spaces and let $\X_n$ be a sequence of emm-spaces.
    Assume that $\X_n$ fibrates over $\X$ with projections $p_n$ inducing concentration and over $\Y$ with projections $q_n$ also inducing concentration.
    Then the sequence $(q_n)_{n\in \N}$ is precompact with respect to the convergence in $\L^0$ induced by $p_n$ and its limit points are 1-Lipschitz measure preserving maps from $\X$ to $\Y$.
\end{corollary}
\begin{proof}
    Consider $f\in \Lip_1(\Y, \tau_\Y, \sfd_\Y)$ bounded.
    Notice that $q_n\to id$ in the $\L^0$ convergence induced by $q_n$: since $f$ is continuous, this implies $f\circ q_n\to f$ with respect to $q_n$.
    
    Since the projections $q_n$ induce concentration, there exists a sequence $n\mapsto f_n\in\Lip_1(\X_n, \mm_{\X_n})$ such that $f_n\to f$ in $\L^0$ along the fibration determined by $q_n$.
    In particular, the pushforward measures are tight and $\sfd_{\L^0}(f\circ q_n, f_n)\to 0$.
    
    Since the projections $p_n$ induce concentration, the sequence $(f_n)_{n\in \N}$ is precompact in $\L^0$ with respect to the fibration determined by $p_n$ and its limit points are in $\Lip_1(\Y, \mm_\Y)$.
    The estimate $\sfd_{\L^0}(f\circ q_n, f_n)\to 0$ implies that also $f\circ q_n$ is precompact in $\L^0$ with respect to the fibration determined by $p_n$.
    
    Finally, let us choose a sequence $(f^k)_{k\in \N}\subset \Lip_1(\Y, \tau_\Y, \sfd_\Y)$ of bounded functions generating the topology of $\Y$. 
    The thesis follows by applying \Cref{lemma: L^0 compactness of maps fibration}.
\end{proof}
We are now ready to prove the equivalence up to isomorphisms of any two sequences of projections.
\begin{proof}[Proof of \Cref{prop: equivalence of projections}]
    By \Cref{cor: compactness projections}, it is enough to prove that any limit point of $n\mapsto q_n$ with respect to the fibration determined by $n\mapsto p_n$ is in ${\rm Iso}(\X;\Y)$.
    
    In order to prove that, let $\Phi\in \L^0(\X;\Y)$ be such that, up to subsequences, $q_n\to \Phi$ in $\L^0$ with respect to the fibration determined by $n\mapsto p_n$ and let us prove that $\Phi^*\Lip_1(\Y, \mm_\Y)=\Lip_1(\X, \mm_\X)$.
    
    By \Cref{cor: compactness projections}, $\Phi$ is 1-Lipschitz and thus $\Phi^*\Lip_1(\Y, \mm_\Y)\subset \Lip_1(\X, \mm_\X)$.
    Conversely, by the fact that $q_n$ induce concentration, for any $f\in \Lip_1(\X, \mm_\X)$ there exists $n\mapsto f_n\in \Lip_1(\X_n, \mm_{\X_n})$ such that $f_n\to f$ in $\L^0$ along to the fibration induced by $q_n$.
    In particular, $n\mapsto (f_n)_\#\mm_{\X_n}$ is a tight sequence.
    
    Since also $n\mapsto p_n$ induce concentration, up to a further non-relabeled subsequence, $f_n\to g\in \Lip_1(\X, \mm_\X)$ along the fibration induced by $n\mapsto p_n$.
    Finally, since $q_n\to \Phi$ in $\L^0$ along the fibration induced by $n\mapsto p_n$, it is easy to see that $f=g\circ \Phi$, concluding the proof of $\Phi^*\Lip_1(\Y, \mm_\Y)\subset \Lip_1(\X, \mm_\X)$.
    By \Cref{cor: characterisation isomorphisms}, $\Phi$ is an isomorphism.
\end{proof}

\subsection{Stability}
The notion of fibration and the extrinsic characterisation of convergence in concentration for concentrated emm-spaces allows to study the stability of functional analytic objects such as Cheeger energies already explored by the first author in \cite{Gig10} in the context of the measured Gromov--Hausdorff convergence (see also \cite{KuwShi03}), was recently applied to convergence in concentration in \cite{ozawaStabilityRCDCondition2019a} and \cite{gigliStabilityHeatFlow2024}.

\paragraph{Cheeger energy I: test plan.}
We briefly recall Cheeger energies in emm-spaces based on test plans.  
For a topological space~$(\X, \tau)$, let~$\cu{(\X, \tau)}$ denote the space of continuous maps from $[0,1]$ to~$(\X, \tau)$. We endow~$\cu{(\X, \tau)}$ with the compact-open topology. For an extended metric space~$(\X, \sfd)$,  we say that a map $[0,1] \ni t \mapsto \gamma_t \in (\X, \sfd)$  belongs to $\acq{2}{(\X, \sfd)}$ (the space of $2$-absolutely continuous curves) if there exists $f\in \L^2(0, 1)$ such that 
\begin{equation}
    \sfd(\gamma_s, \gamma_t)\leq \int_s^t f(r)\,\d r
\end{equation}
for all $s<t$ in $(0, 1)$. Let~$\pr(\cu{(\X, \tau)})$ denote the space of Borel probability measures on~$\cu{(\X, \tau)}$.
Following~\cite{AmbGigSav14}, given an emm-space~$(\X, \tau, \sfd, \mm)$, 
we define {\it a test plan} as $\pi\in \pr(\cu{(\X, \tau)})$ such that:
\begin{itemize}
    \item $\pi$ is concentrated on $\acq{2}{(\X, \sfd)}$;
    \item there exists $C>0$ such that  $(e_t)_\#\pi\leq C\mm$ for every $t \ge 0$.
\end{itemize}

We write $\mathscr T(\X)$ for the set of test plans on $\X$.
We  say that a set in $\pr(\cu{(\X, \tau)})$ is $\mathscr T(\X)$-negligible if it is $\pi$-negligible for all $\pi \in \mathscr T(\X)$. 
Similarly, we say that a property holds $\mathscr T(\X)$-a.e.~if the property holds on a set whose  complement is $\mathscr T(\X)$-negligible.
We  recall that, if $\gamma\in \acq{2}{(\X, \sfd)}$, then its metric derivative $|\dot{\gamma}|$ can be defined as the a.e.~limit
\begin{equation}
    |\dot{\gamma}_t|=\lim_{h\to 0}\frac{\sfd_\X(\gamma_{t}, \gamma_{t+h})}{|h|}.
\end{equation}

\begin{definition}[Weak upper gradients and Cheeger energy] \label{d:Che}
    Let $\X$ be an emm-space. 
    A function $G\in \L^2(\X)$ is called a \emph{weak upper gradient} of a function $u\in \L^0(\X)$ if for $\mathscr T(\X)$-a.e.~curve $\gamma$ it holds
    \begin{equation}
        |u\circ\gamma_1-u\circ\gamma_0|\leq \int_0^1 (G\circ\gamma_t)|\dot{\gamma}_t|\,\d t.
    \end{equation}
    The \emph{minimal weak upper gradient} $|Du|_{w, \X}$ of $u$ is the pointwise a.e.~minimum (which exists, see \cite{AmbGigSav14}) among the weak upper gradients of $u$.
    The \emph{Cheeger energy}  $\ch_\X: L^2(\X) \to [0,+\infty]$ is defined as 
    \begin{equation}
        \ch_\X(u)=\frac 12 \int_\X |D u|_{w, \X}^2\,\d\mm_\X.
    \end{equation}
    Conventionally, we define $\ch_\X(u)\coloneqq+\infty$ if there is no weak upper gradient of~$u$. We simply write $\ch$ instead of $\ch_\X$ unless any confusion could occur. We write 
    $$\W^{1,2}(\X):=\{u \in \L^2(\X): \ch_\X(u)<+\infty\}.$$
    The space~$\W^{1,2}(\X)$ is a Banach space endowed with the norm:
    $$\|u\|_{\W^{1,2}(\X)}:=\sqrt{\ch_\X(u)+\|u\|^2_{\L^2(\X)}}.$$
\end{definition}
In \cite{SuYo25+}, the following results have been shown:
\begin{lemma}[{\cite[Section~6]{SuYo25+}}]
     Let $\X, \Y$ be two emm-spaces and let $\varphi:\X\to \Y$ be an emm-isomorphism.
     Then there exists a unique $\mathscr{T}(\X)$-a.e.~defined map $\varphi_*:\cu{(\X, \tau_{\sfd_\X})}\to \cu{(\Y, \tau_{\sfd_\Y})}$ such that for $\mathscr T(\X)$-a.e.~curve $\gamma$ the equality $(\varphi_* \gamma)_t=\varphi(\gamma_t)$ holds for all $t\in [0, 1]$, and 
     the following hold:
     \begin{itemize}
         \item $\varphi_\# (e_t)_\#\pi=(e_t)_\#(\varphi_*)_\#\pi$ for all $t\in[0, 1]$;
         \item the map $(\varphi_*)_\#$ is a bijection between $\mathscr T(\X)$ and $\mathscr T(\Y)$;
         \item for $\mathscr T(\X)$-a.e.~$\gamma$ and $\mathsf{Leb}$-a.e.~$t\in [0, 1]$ it holds
          \begin{equation} |\dot{\gamma}_t|=|\dot{(\varphi_*\gamma)}_t|.
         \end{equation}
     \end{itemize}
\end{lemma}
From these, they conclude the invariance of minimal upper gradients and Cheeger energy up to isomorphism.
\begin{theorem}[{\cite[Theorem~6.13]{SuYo25+}}] \label{t:ECE}
    Let $\X, \Y$ be two emm-spaces and let $\varphi:\X\to\Y$ be an emm-isomorphism.
    Then, for all $u\in \L^2(\Y)$, $$|D(u\circ\varphi)|_{w, \X}=|D u|_{w, \Y}\circ\varphi \qquad \text{$\mm_\X$-a.e.},$$
    and $\varphi$ induces an isometric isomorphism between Banach spaces~$\W^{1,2}(\X)$ and $\W^{1,2}(\Y)$$:$ 
$$\|u\|_{\W^{1,2}(\Y)}=\|u\circ \varphi\|_{\W^{1,2}(\X)}\qquad u \in \W^{1,2}(\Y).$$
\end{theorem}
Following the same arguments in \cite[Section 6]{SuYo25+}, we obtain the following.
\begin{lemma}
     Let $\X, \Y$ be two emm-spaces and let $\varphi\in \Lip_1(\X, \mm_\X;\Y)$ be measure preserving.
     Then there exists a unique $\mathscr{T}(\X)$-a.e.~defined map $\varphi_*:\cu{(\X, \tau_{\sfd_\X})}\to \cu{(\Y, \tau_{\sfd_\Y})}$ such that for $\mathscr T(\X)$-a.e.~curve $\gamma$ the equality $(\varphi_* \gamma)_t=\varphi(\gamma_t)$ holds for all $t\in [0, 1]$. 
     Moreover, the following hold:
     \begin{itemize}
         \item $\varphi_\# (e_t)_\#\pi=(e_t)_\#(\varphi_*)_\#\pi$ for all $t\in[0, 1]$;
         \item $(\varphi_*)_\#$ maps $\mathscr T(\X)$ into $\mathscr T(\Y)$;
         \item for $\mathscr T(\X)$-a.e.~$\gamma$ and $\mathsf{Leb}$-a.e.~$t\in [0, 1]$ it holds
          \begin{equation} |\dot{\gamma}_t|\geq|\dot{(\varphi_*\gamma)}_t|.
         \end{equation}
     \end{itemize}
\end{lemma}
In particular, the following corollary holds.
\begin{corollary}\label{corollary: monotonicity Cheeger}
    Let $\X, \Y$ be two emm-spaces and let $\varphi\in \Lip_1(\X, \mm_\X;\Y)$ be measure preserving.
    Then, for all $u\in \L^0(\Y)$, $|D(u\circ\varphi)|_{w, \X}\leq |D u|_{w, \Y}\circ\varphi$ $\mm_\X$-a.e.,
    in particular, 
    $$\ch_\Y(u)\geq \ch_\X(u\circ\varphi) \qquad u \in \L^2(\Y).$$ 
\end{corollary}

\paragraph{Cheeger energy II: relaxation.} The Cheeger energy given above coincides with the one constructed through the $L^2$-relaxation in extended topological-metric measure spaces. As we need both characterisations for the following sections, we briefly present the relaxation approach as well below.

Given a pre-extended topological-metric space $\X$ and $u\in\Lip(\X, \tau, \sfd_\X)$, the \emph{asymptotic Lipschitz constant} $\lip_a(u)$ of $u$ is defined as
\begin{equation}
    \lip_a(u)(x)=\limsup_{y, z\to x}\frac{|u(y)-u(z)|}{\sfd(z, y)},
\end{equation}
where the $\limsup$ is considered with respect to the topology $\tau$ and is assumed to be 0 if taken over an empty set of sequences -- i.e.~if $x$ is isolated.
Notice that $\lip_a(u)$ is a $\tau$-upper semicontinuous function of $x$.
For a pre-extended topological-metric-measure space $\X$ and $u\in\L^0(\X)$, the \emph{relaxed Cheeger energy} $\ch_*(u)$ of $u$ is defined as
\begin{equation}
    \ch_*(u)=\inf\bigg\{\liminf_n\frac 12 \int_\X \lip_a^2 (u_n)\,\d\mm: (u_n)_{n\in \N}\subset \Lip(\X, \tau, \sfd_\X), \; u_n\to u\text{ in }\L^2(\X)\bigg\}.
\end{equation}

\begin{theorem}[\cite{Sav19}]
    Let $\X$ be an extended topological-metric-measure space.
    Then $\ch_*=\ch$.
\end{theorem}

\paragraph{Cheeger energies under inverse limits.} 
The following stability result, the full Mosco convergence for the Cheeger energy for the inverse limit, has been proved by the second author and Yokota in \cite[Theorem 4.6]{SuYo25+} (see also \cite[Theorem 9.1]{AmbErbSav16}).
\begin{proposition}[\cite{AmbErbSav16,SuYo25+}]\label{proposition: mosco inverse limits}
    Let $\X=\varprojlim \X_n$ be an inverse limit with projections $n\mapsto p_n$.
    If we identify $\ch_{\X_n}$ with the following functional
    \begin{equation}
        \L^0(\X)\ni u\mapsto\begin{cases}
            \ch_{\X_n}(v\circ p_n)\quad &\text{if }u=v\circ p_n\text{ for some }v\in \L^0(\X_n),\\
            +\infty&\text{otherwise},
        \end{cases}
    \end{equation}
    then $\ch_\X$ is the largest $\L^2$-lower-semicontinuous functional satisfying $\ch_\X\leq \inf_n \ch_{\X_n}$.
\end{proposition}
\begin{remark}
    \Cref{proposition: mosco inverse limits} can be regarded as a Mosco convergence result. 
    Indeed, since $n\mapsto \ch_{\X_n}$ is non increasing, \cite[Proposition 5.7]{dalmasoIntroductionGConvergence1993} implies that $\ch_\X$ is the strong $\L^2$ $\Gamma$-limit. 
    By its convexity and standard functional analytic results, $\ch_\X$ also coincides with the weak $\L^2$ lower semicontinuous envelope, hence applying again \cite[Proposition 5.7]{dalmasoIntroductionGConvergence1993} in the weak $\L^2$ topology proves the  Mosco convergence.  
\end{remark}

\paragraph{$\Gamma$-convergence.}
We recall the definition of $\Gamma$-convergence.
\begin{definition}\label{def: Gamma-conv}
    Given a sequence of emm-spaces fibrating over a limit topological-metric-measure space $\X$, the $\Gamma-\limsup$ of the functionals $\ch_{\X_n}$ in $\L^2$ is defined as 
    \begin{equation}
        \Gamma-\limsup_n \ch_{\X_n}(u)=\sup \{\liminf_n \ch_{\X_n}(u_n): \L^0(\X_n)\ni u_n\to u \text{ in }\L^2\}.
    \end{equation}
    Similarly, 
     the $\Gamma-\liminf$ of the functionals $\ch_{\X_n}$ in $\L^2$ is defined as 
    \begin{equation}
        \Gamma-\liminf_n \ch_{\X_n}(u)=\inf \{\liminf_n \ch_{\X_n}(u_n): \L^0(\X_n)\ni u_n\to u \text{ in }\L^2\}.
    \end{equation}
    While it always holds $\Gamma-\limsup_n \ch_{\X_n}\leq \Gamma-\limsup_n \ch_{\X_n}$, in case equality holds we call the limit functional the $\Gamma-\lim$ of the sequence.
\end{definition}
\begin{remark}
    The convergence in $\L^2$ along the fibration can be metrised (see \cite{gigliStabilityHeatFlow2024}): in this case, what is defined above actually coincides with the classically defined $\Gamma-\limsup$ of functionals over a separable metric space. 
\end{remark}

Our first goal is to prove the $\Gamma$-$\limsup$ inequality for the Cheeger energies along fibrations inducing a domination.
\begin{theorem}\label{theorem: gamma limsup}
    Let $(\X_n)_{n\in \N}$ be a sequence of emm-spaces and assume that $\X_n$ fibrates over an extended topological-metric-measure space $\X$ with projections $n\mapsto p_n$ inducing a domination.
    Then \begin{equation}
        \Gamma-\limsup_n \ch_{\X_n}\leq \ch_\X.
    \end{equation}
\end{theorem}

A key for the proof of \Cref{theorem: gamma limsup} is the fact that, similarly to how \Cref{lemma: L^0 compactness of maps fibration} extends the compactness of 1-Lipschitz maps with a Polish target, also the approximation result can be generalised.
\begin{lemma}\label{lemma: good subspaces}
    Let $(\Omega, \mu)$ be a measure space.
    Let $x_1, \ldots, x_n\in \L^\infty(\mu)$ and let $\varepsilon>0$.
    Then there exist $y_1, \ldots, y_n\in \L^\infty(\mu)$ and  a finite dimensional subspace $V\subset \L^\infty(\mu)$ containing $y_1, \ldots, y_n$ such that the following two hold:
    \begin{itemize}
        \item $\sfd_{\L^\infty(\mu)}(x_j, y_j)\leq \varepsilon$ for $j=1, \ldots, n$;
        \item $V$ endowed with the restriction metric is linearly isometric to $(\mathbb R^k, \ell_\infty)$ for some $k\in \N$.
    \end{itemize}
    In particular, the statement above holds for $\ell^\infty=\L^\infty(\mathbb N, \sum_{n=1}^\infty\delta_n)$.
\end{lemma}
\begin{proof}
    By the density of simple functions in $\L^\infty(\mu)$, there exists a partition $\Omega=\bigsqcup_{1\leq l\leq M} E_l$ of $\Omega$ in measurable sets $E_1, \ldots, E_M$ and simple functions $y_j=\sum_{1\leq l\leq M} a_{j, l}\One_{E_l}$ such that $\sfd_{\L^\infty(\mu)}(x_j, y_j)\leq \varepsilon$ for $j=1, \ldots, n$.

    The space $V={\sf span}\{\One_{E_1}, \ldots, \One_{E_M}\}$ is then linearly isometric to $(\mathbb R^k, \ell_\infty)$, where $k$ is the number of sets in $E_1, \ldots, E_M$ with positive measure.
\end{proof}
\begin{lemma}\label{lemma: polish approximation fibration}
    Let $\X_n$ be a sequence of emm-spaces fibrating on a Polish space $\X$ via projections $n\mapsto p_n$ inducing a domination.
    Let $T\in \Lip_1(\X, \mm_\X;\ell^\infty)$.
    Then there exists a sequence $n\mapsto T_n\in \Lip_1(\X_n, \mm_{\X_n};\ell^\infty)$ such that $T_n\to T$ in $\L^0$.
\end{lemma}
\begin{proof}
    Let us start by proving the statement under the assumption that there exists a finite dimensional space $V$ such that $T_\#\mm_\X(V)=1$ and such that, endowed with the restriction metric, it is linearly isometric to $(\mathbb R^k, \ell_\infty)$ for some $k\in \N$.
    Let $\Phi=(\varphi_1, \ldots, \varphi_k):V\to\mathbb R^k$ be an isometry and let $T_j=\varphi_j\circ T$ for $j=1, \ldots, k$.
    By the domination, since $T_1, \ldots, T_k$ are 1-Lipschitz, there exist $(T_{1, n})_{n\in \N}, \ldots, (T_{k, n})_{n\in \N}$ such that $T_{j, n}\to T_j$ in $\L^0$ as $n\to \infty$ for $j=1, \ldots, k$.
    In particular, by inspecting the distance, $\tilde T_n\to \Phi\circ T$ in $\L^0$, where $\tilde T_n=(T_{1, n}, \ldots, T_{k, n})$.
    Since $\Phi$ is an isometry, $T_n=\Phi^{-1}\circ T_n$ converges to $T$ in $\L^0$.

    Let us now consider a general $T\in \Lip_1(\X, \mm_\X;\ell^\infty)$ and let us start by noticing that $T_\#\mm_\X$ is tight.
    In particular, given $\varepsilon>0$ there exist $x_1, \ldots, x_N\in \ell^\infty$ such that 
    \begin{equation}
        T_\#\mm_\X\Bigg(\bigcup_{j=1}^N B^{\ell^\infty}(x_j, \varepsilon/4)\Bigg)\geq 1-\varepsilon/2.
    \end{equation}
    By \Cref{lemma: good subspaces}, we can assume that there are $y_1, \ldots, y_N$ such that $\sfd_{\ell^\infty}(x_j, y_j)\leq \varepsilon/4$ for $j=1, \ldots, N$ and a finite dimensional subspace $V$ which contains $y_1, \ldots, y_N$ such that, endowed with the restriction metric, it is linearly isometric to $(\mathbb R^k, \ell_\infty)$ for some $k\in \N$ via an isomorphism $\Phi:V\to \mathbb R^k$. 
    
    If $\pi:\ell^\infty\to \{y_1, \ldots, y_N\}$ is a Borel nearest point projection as in \cite[Lemma 3.4]{Sh16}, then $\tilde T \coloneqq\pi\circ T$ is 1-Lipschitz up to $\varepsilon/2$ (see \cite[Definition 4.38]{Sh16}) and satisfies $\sfd_{\L^0}(T, \tilde T)\leq \varepsilon/2$.

    Since $\Phi\circ \tilde T$ is 1-Lipschitz up to $\varepsilon/2$, by \cite[Lemma 5.4]{Sh16}, there exists a 1-Lipschitz function $\tilde S:\X\to V$ such that $\ell_\infty(\tilde T, \tilde S)\leq \varepsilon/2$.
    In particular, $S=\Phi^{-1}\circ \tilde S$ satisfies $\ell_\infty(S, T)\leq \varepsilon$ and $S_\#\mm_\X(V)=1$.
    By the first part, there exist $n\mapsto S_n$ such that $S_n\to S$ in $\L^0$.
    Sending $\varepsilon$ to 0, by a diagonal argument we conclude.
\end{proof}

\begin{proof}[Proof of \Cref{theorem: gamma limsup}]
    Let us start by recalling that by \Cref{lemma: inverse system} there exist an inverse limit $\Y=\varprojlim \Y_k$ and an isomorphism $\Phi:\X\to \Y$.
    
    By the Kuratowski embedding for mm-spaces, $\Y_k\simeq (\ell^\infty,  \mm_{\Y_k})$ for all $k\in \N$ and some measures $\mm_k$.
    In order to simplify the notation, and in view of the equivalence of $\ch$ up to isometries, we will identify $\Y_k$ with $(\ell^\infty, \mm_{\Y_k})$.

    For $k\in \N$, let $\tilde q_k:\X\to \Y_k$ be the 1-Lipschitz measure preserving projections granted by the inverse limit and let $q_k=\tilde q_k\circ\Phi$.
    Clearly, for each $k\in \N$, $q_k$ is a.e.~1-Lipschitz and measure preserving.
    By \Cref{lemma: polish approximation fibration}, there exist a.e.~1-Lipschitz maps $q_{n, k}:\X_n \to \ell^\infty$ such that $q_{n, k}\to q_k$ in $\L^0$.
    By \Cref{proposition: L^0 properties}, for each $f\in \Lip_b(\Y_k)$ it holds  $f\circ q_{n, k}\to f\circ q_k $ in $\L^0$.
    Let also $\X_{n, k}=(\ell^\infty, (q_{n, k})_\#\mm_{\X_n})$ for $n, k\in \N$ and notice that the $\X_{n, k}$ are mm-spaces and satisfy $\X_{n, k}\prec \X_n$ for all $n, k\in \N$ and $(q_{n, k})_\#\mm_{\X_n}\rightharpoonup \mu_{\Y_k}$, implying also $\X_{n, k}\to \Y_k$ as $n\to \infty$ in the $\square$-topology for all $k\in \N$.

    Let us now fix $k\in \N$.
    Consider $g\in \Lip_b(\Y_k)$.
    By upper semicontinuity of the asymptotic Lipschitz constant and the properties of weak convergence,
    \begin{equation}\label{eq: semi C^0 up}
        \limsup_n\int \lip_a^2(g)\d(q_{n, k})_\#\mm_{\X_n}\leq \int \lip^2_a(g)\d\mm_{\Y_k}.
    \end{equation}
    By the definition of Cheeger energy via relaxation, if $f$ is in $\Dom(\ch_{\Y_k})$ with, there exist $f_j\in \Lip_b(\Y_k)$ such that $f_j\to f$ in $\L^2$ and 
    \begin{equation}\label{eq: relax}
        \ch_{\Y_k}(f)=\lim_j \frac 12 \int \lip^2_a(f_j)\d\mm_{\Y_k}.
    \end{equation}
    
    Coupling \eqref{eq: semi C^0 up} and \eqref{eq: relax}, by a diagonal argument there exists a sequence $n\mapsto f_n\in \Lip_b(\ell^\infty)$, a not necessarily injective relabeling of $j\mapsto f_j$, such that  $f_n\circ q_{n, k}\to f\circ q_n$ and
    \begin{equation}
        \frac 12 \limsup_n\int \lip_a(f_n)\d(q_{n, k})_\#\mm_{\X_n}\leq \ch_{\Y_k}(f).
    \end{equation}
    Again, by the definition of relaxation, 
    \begin{equation}
        \limsup_n\ch_{\X_{n, k}}(f_n)\leq \frac 12 \limsup_n\int \lip_a(f_n)\d(q_{n, k})_\#\mm_{\X_n}\leq \ch_{\Y_k}(f)
    \end{equation}
    and, by \Cref{corollary: monotonicity Cheeger}, 
    \begin{equation}
        \limsup_n \ch_{\X_n}(f_n\circ q_{n, k})\leq \limsup_n\ch_{\X_{n, k}}(f_n)\leq \ch_{\Y_k}(f).
    \end{equation}

    By the density of bounded Lipschitz functions, then $\Gamma-\limsup \ch_{\X_n}\leq \Phi^*\ch_{\Y_k}$, where we identify $\ch_{\Y_k}$ with a functional on $\X$ as in \Cref{proposition: mosco inverse limits} and by $\Phi^*$ we mean $\Phi^*\ch_{\Y_k}(u)=\ch_{\Y_k}(u\circ \Phi^{-1})$.
    Finally, this implies $\Gamma-\limsup \ch_{\X_n}\leq \inf_k \Phi^*\ch_{\Y_k}=\Phi^*\inf_k \ch_{\Y_k}$: by \Cref{proposition: mosco inverse limits} it follows that $\Gamma-\limsup \ch_{\X_n}\leq \Phi^*\ch_{\Y}=\ch_\X$, where the last equality follows from \Cref{t:ECE}.
\end{proof}
This yields a different proof of \Cref{corollary: stability constants}, which we restate here in a slightly more general form.
\begin{corollary} \label{c:SLP}
    Let $n\mapsto \X_n$ be a sequence of emm-spaces  supporting the log-Sobolev inequality~$\LS(C)$ {\rm(}resp.~the Poincar\'e inequality $\P(C)${\rm)} and let $\X$ be an emm-space such that $\overline{\mathcal P_\X}^\square\subset K-\liminf_n \overline{\mathcal P_{\X_n}}^\square$.
    Then $\X$ supports~$\LS(C)$ {\rm (}resp.~$\P(C)${\rm )}.
\end{corollary}
\begin{proof}
    By \Cref{proposition: general projections} and up to isomorphism, we can assume that $\X$ is an extended topological-metric-measure space and that $n\mapsto\X_n$ fibreates over $\X$ with projections $n\mapsto p_n$ inducing a domination.
    
    Let us prove the stability of the log-Sobolev inequality, for the Poincaré inequality the proof is analogous.
    Let $u\in \L^0(\X)$ and assume that $\ch_\X(u)<+\infty$, otherwise the inequality is automatically true.
    By approximation, we can also assume that $u\in\L^\infty(\X)$.
    By \Cref{theorem: gamma limsup}, there exists a sequence $n\mapsto \tilde u_n\in \L^0(\X_n)$ such that $\tilde u_n\to u$ in $\L^0$ and $\limsup_n \ch_{\X_n}(\tilde u_n)\leq \ch_\X(u)$.
    For $n\in \N$, define $u_n=(-\|u\|_{\L^\infty})\vee \tilde u_n\wedge \|u\|_{\L^\infty}$ and notice that the following two properties hold:
    \begin{itemize}
        \item $u_n\to u$ in $\L^p$ for all $p\in (1, +\infty)$;
        \item $\limsup_n \ch_{\X_n}(u_n)\leq \limsup_n \ch_{\X_n}(\tilde u_n)\leq \ch_\X(u)$.
    \end{itemize}

    By assumption, for all $n\in \N$ it holds
    \begin{equation}
        \int_{\X_n} u_n^2\log(u_n^2)\,\d\mm_{\X_n}\leq 4C^2 \ch_{\X_n}(u_n).
    \end{equation}
    By the convergence in $\L^p$ for all $p\geq 1$, 
    \begin{equation}
        \int_{\X_n} u_n^2\log(u_n^2)\,\d\mm_{\X_n}\to \int_\X u^2\log(u^2)\,\d\mm_\X.
    \end{equation}
    In particular, by the $\limsup$-inequality, 
    \begin{align}
        \int_\X u^2\log(u^2)\,\d\mm_\X&=\lim_n \int_{\X_n} u_n^2\log(u_n^2)\,\d\mm_{\X_n}\\
        &\leq 4C^2 \limsup_n \ch_{\X_n}(u_n)\\
        &\leq 4C^2 \ch_\X(u).
    \end{align}  
\end{proof}
\begin{remark}
    \Cref{c:SLP} is a stronger statement than~\Cref{corollary: stability constants} in the sense that it proves the stability of the log-Sobolev and the Poincar\'e inequalities under a weaker assumption than the weak convergence of pyramids since we only assume the limit infimum inclusion in \Cref{c:SLP}. 
\end{remark}
We conclude this section by proving the lower semicontinuity for the optimal transport cost.
To this aim, we recall that, given a Polish space $\X$, a lower semicontinuous (pseudo-)metric $\sfd$ on $\X$, $p\geq 1$ and two Borel probability measures $\mu, \nu$ on $\X$, we define 
\begin{equation}
    W^\sfd_p(\mu, \nu)\coloneqq\inf_{\sppi\in\adm(\mu, \nu)} \int\sfd^p\,\d\ppi=\min_{\sppi\in\adm(\mu, \nu)} \int\sfd^p\,\d\ppi.
\end{equation}
In the case that $\mu$ and $\nu$ are defined on an emm-space which is clear from context, we will drop the explicit mention of $\sfd$.

\begin{proposition}\label{proposition: semicontinuity cost} Let $1 \le p<\infty$.
     Let $n\mapsto \X_n$ be a sequence of emm-spaces fibrating over an extended topological-metric-measure space $\X$ with projections $n\mapsto p_n$ inducing a domination.
     Let $n\mapsto \mu_n, \nu_n\in \pr(\X_n)$ be two sequences of measures satisfying $\mu_n\to \mu$, $\nu_n\to \nu$ weakly in $\L^1$ for some $\mu, \nu\in \pr(\X)$ $($recall \Cref{d:L1W}$)$.
     Then $W_p(\mu, \nu)\leq\liminf_n W_p(\mu_n, \nu_n)$.
\end{proposition}
\begin{proof}
    Let $(\varphi_j)_{j\in \N}$ be a sequence of 1-Lipschitz, bounded and continuous functions on $\X$ which generate the distance and topology of $\X$ and let $\sfd_k(x, y)=\sup_{j\leq k}|\varphi_j(x)-\varphi_j(y)|$.
    By definition, $\sfd_\X=\sup_k \sfd_k$.
    
    By \Cref{lemma: regularity projections}, for all $k\in \N$ there exists a vanishing sequence $(\varepsilon_n)_{n\in \N}$ and a sequence of sets $n\mapsto E_n\subset \X_n$ with $\mm_{\X_n}(E_n)\geq 1-\varepsilon_n$ for all $n\in \N$ satisfying
    \begin{equation}
        \sfd_k(p_n(x), p_n(y))\leq \sfd_{\X_n}(x, y)+\varepsilon_n
    \end{equation}
    for all $x, y\in E_n$.
    
    Let $\tilde \ppi_n$ be an optimal coupling between $\mu_n$ and $\nu_n$ and let $\ppi_n=(p_n\times p_n)_\#\tilde\ppi_n$.
    Then
    \begin{align}
        W_p^{\sfd_k}(\mu, \nu)\leq \int_\X \sfd_k\,\d\ppi_n &=\int_{\X_n} \sfd_k^p(p_n(x), p_n(y))\,\d\tilde\ppi_n\\
        &\leq \int_{E_n\times E_n} (\varepsilon_n+\sfd_{\X_n}(x, y))^p\,\d\tilde\ppi_n\\
        &\quad+(\max \sfd_k)^p\Big(\mu_n(\X_n\setminus E_n)+ \nu_n(\X_n\setminus E_n)\Big).
    \end{align}
    Since both $n\mapsto \mu_n$ and $n\mapsto \nu_n$ are equi-integrable sequences, the second term in the right-hand side is vanishing when $n$ goes to $\infty$.
    In particular, sending $n$ to $\infty$, we get $W_p^{\sfd_k}(\mu, \nu)\leq\liminf W_p(\mu_n, \nu_n)$.
    
    Since $\sfd_k\nearrow \sfd_\X$, then $W_p^{\sfd_k}(\mu, \nu)\nearrow W_p(\mu, \nu)$ as $k\to\infty$.
    The thesis then follows by sending $k$ to~$\infty$.
\end{proof}

\begin{remark} Let $p=1$. 
    If a sequence of concentrated emm-spaces $\X_n$ has an equi-bounded diameter and converges to a extended topological-metric-measure space $\X$ in $\sfd_{\sf conc}$,  one can improve the conclusion of~\Cref{proposition: semicontinuity cost} from the limit infimum inequality to the limit equality.
    The proof is a simple argument for the convergence of the Kantorovich potentials: thanks to the uniform bound on the diameter, the Kantorovich potentials are  equi-Lipschitz, thus pre-compact; the definition of convergence in concentration guarantees that the limits of Kantorovich potentials are $c$-concave.

    The general case~$p>1$ is not as simple as $p=1$: while the compactness argument still works for $p>1$, it is not clear if the inequality $f_n\oplus g_n\leq \sfd_{\X_n}^p$ of admissible functions in the $W_p$-Kantorovich duality is  stable under the convergence. The following example tells a delicate point for this argument: let $\X_n=\mathbb S^n(1)$ be the sequence of $n$-dimensional unit spheres; $\X_n\to \star$ in concentration (see \Cref{c:HDS}) and $\X_n\times \X_n\to \star\times\star \cong \star$. In particular, the limit distance $\sfd_\infty$ on the one-point mm-space~$\star$ is identically zero while the sequence of the distances $\sfd_{\X_n}$ converges to the non-zero constant $\frac{\pi}{2}$ in $\L^0$. 
    
    This is closely related the stability of the $\CD$ and $\RCD$ conditions, again in the assumption of uniformly bounded diameter. If one could prove that for each $\mu, \nu$ there exist $\mu_n, \nu_n$ converging weakly in $\L^1$ to $\mu, \nu$ and $W_2(\mu_n, \nu_n)\to W_2(\mu, \nu)$, this would imply the stability of the $\CD$ condition. That would also be enough to prove the stability of the $\RCD$ condition and of the heat flow: see the arguments in \cite{gigliStabilityHeatFlow2024}.
\end{remark}

\section{Applications} \label{s:AP}
In this section, we discuss various examples of concentrated emm-spaces and their convergence in the concentration topology.
In the following sections, one of our focal points is to determine whether a given emm-space is emm-isomorphic to an mm-space. For this purpose, we introduce the following definition: 
\begin{definition}[Truly emm-space] \label{d:TE}An emm-space~$\X$ is {\it truly extended} if there is no mm-space that is emm-isomorphic to $\X$, namely, 
$$[\X] \in \mathfrak X\setminus \mathcal X,$$
where $[\X]$ is the equivalence class of $\X$ under emm-isomorphism.
\end{definition}

\subsection{Product spaces}
Let $({\rm M}_n,\mathsf{d}_n,\mathfrak{m}_n)$ be mm-spaces. Set
\[{\rm X}_{(N)} \coloneqq \prod_{n=1}^N {\rm M_n}, \quad {\rm X}\coloneqq\prod_{n=1}^\infty {\rm M}_n,\qquad \mfm_{(N)}\coloneqq \bigotimes_{n=1}^N \mathfrak{m}_n,\quad \mfm\coloneqq\bigotimes_{n=1}^\infty \mathfrak{m}_n,
\]
with product topology and the product distances
\[
\mathsf{d}_{(N)}(x,y)\coloneqq\Big(\sum_{n=1}^N \mathsf{d}_n(x_n,y_n)^2\Big)^{1/2}, \quad 
\mathsf{d}(x,y)\coloneqq\Big(\sum_{n=1}^\infty \mathsf{d}_n(x_n,y_n)^2\Big)^{1/2}\in[0,\infty].
\]
Define 
\[
{\rm X}^{(N)}\coloneqq\prod_{n>N} {\rm M}_n\quad\text{with metric }\mathsf{d}^{(N)}(\xi,\eta)=\Big(\sum_{n>N} \mathsf{d}_n(\xi_n,\eta_n)^2\Big)^{1/2},
\]
and 
\begin{equation}\label{eq1:tail_levy}
\mathfrak{m}^{(N)}\coloneqq\bigotimes_{n>N}\mathfrak{m}_n.
\end{equation}
\begin{proposition} \label{p:CPC}
Assume that for every $\varepsilon>0$ there exists $N$ such that for every $1$-Lipschitz
$h$ on the tail block $(\X^{(N)}, \mssd^{(N)})$, 
we have
\begin{equation}\label{eq:tail_levy}
\mathfrak{m}^{(N)}\big(|h-\mssm_h|>\varepsilon\big)\le \varepsilon,
\end{equation}
where $\mssm_h$ is the L\'evy mean of~$h$ $($see~\eqref{def: lévy mean}$)$.
Then, $\X$ is concentrated. 
In particular, 
$$\overline{\bigcup_{n \in \N}\mathcal P_{X_n}}^\square = \mathcal P_X. $$
\end{proposition}
\begin{proof}
First notice that as the space $\X$ is an infinite-product space, it is in particular  an inverse limit, thus an extended topological-metric-measure space.
By \Cref{cor: density of continuous functions}, it suffices to check the mm-factor condition~\eqref{eq:Aeps}  for all functions $f\in \Lip_1(\X, \tau, \sfd)$. Here recall that $\Lip_1(\X, \tau, \sfd)$ denotes the space of $\tau$-continuous and $1$-Lipschitz functions on~$\X$.  

Fix $\varepsilon>0$ and choose $N$ as above.
Let $\pi_\varepsilon\coloneqq\mathrm{pr}_{\le N}:{\rm X}\to K_\e:=\X_{(N)}=\prod_{n\le N}{\rm M}_n$.
Then $\pi_\varepsilon$ is $1$-Lipschitz because
$\mathsf{d}_{(N)}(\pi_\varepsilon(x),\pi_\varepsilon(y))\le \mathsf{d}(x,y)$.
Now let $f\in\Lip_1(\X, \tau, \sfd)$. For each fixed $u\in \X_{(N)}$ define a tail function
\[
h_u(\xi)\coloneqq f(u \times \xi)\qquad (\xi\in {\rm X}^{(N)}).
\]
Then, $h_u \in \Lip_1(\X^{(N)})$. 
We note that the L\'evy mean~
$$g(u):=\mssm_{h_u}$$ is $1$-Lipschitz on $(\X_{(N)},\mathsf{d}_{(N)})$:
indeed, if $u,v\in \X_N$, then $|h_u(\xi)-h_v(\xi)|\le \mathsf{d}_{(N)}(u,v)$ for every $\xi$,
so the distribution functions of $h_u$ and $h_v$ are shifted by at most $\mathsf{d}_{(N)}(u,v)$, hence the medians (thus the L\'evy mean also) differ
by at most $\mathsf{d}_{(N)}(u,v)$.
By \eqref{eq:tail_levy} applied to each $h_u$ and integrating in $u$,
\begin{align} \label{e:UEC}
\mfm\big(|f-(g\circ\pi_\varepsilon)|>\varepsilon\big)
=\int_{{\rm K}_\varepsilon} \mathfrak{m}^{(N)}\big(|h_u-g(u)|>\varepsilon\big)\,\diff\mfm_{N}(u)
\le \varepsilon.
\end{align}
This is \eqref{eq:Aeps} with  $K_\e=\X_N$.
\end{proof}

One of the simplest examples is the product of spheres. 
\begin{example} \label{e:IPS}
    Let $\X_n=\prod_{j=1}^n \mathbb S^j(1)$ be the product mm-space, where $\mathbb S^j(1)$ is the $j$-dimensional unit sphere endowed with the standard geodesic distance and the standard normalised spherical volume measure. Let $$\mathcal P=\overline{\bigcup_{n \in \N} \mathcal P_{\X_n}}^\square$$ and $\X=\prod_{n=1}^\infty \mathbb S^n(1) $ be the infinite-product emm-space.
    By~\cite[Example~7.36]{Sh16} (or \Cref{p:CPC} combined with the standard measure-concentration~of high-dimensional spheres),  $\X$ is concentrated  and 
    $$\mathcal P = \mathcal P_X. $$
    Furthermore, by a simple application of Borel--Cantelli to the set 
    $$A_n=\Big\{\big((x_n)_{n \in \N}, (y_n)_{n \in \N}\big) \in \X \times \X: \sfd_{\mathbb S^n}(x_n, y_n) \ge \frac{\pi}{2}\Big\}$$ and noting $\mm_\X^{\otimes 2}(A_n)=1/2$,  we obtain 
    $\sfd_{\X}=+\infty$ $\mm_{\X}^{\otimes 2}$-a.e.. In particular, it is not emm-isomorphic to any mm-space, thus is is truly extended in the sense of \Cref{d:TE}. Combined with the fact that $\X$ is concentrated as seen above, we conclude $$\X \in \partial \mathfrak X_{\sf conc} = \mathfrak X_{\sf conc} \setminus \mathcal X.$$ 
\end{example}

\subsection{Spheres}
Let $\mathbb S^n(r)$ denote the $n$-sphere with radius $r$ endowed with the standard geodesic metric and the standard normalised spherical volume measure. We denote by $\star$ the one-point mm-space (the mm-space consisting of one point), whose pyramid $\mathcal P_\star$ is the smallest pyramid.  The largest pyramid is the entire~$\mathcal X$, which can be represented by the emm-space~$\mathbb I_\infty$ as we have seen in~\Cref{p:ME}.  
Define an emm-space
$$\Gamma^\infty_{\lambda^2}=(\R^\infty, \tau^\infty, \ell_2, \gamma_{\lambda^2}^{\otimes \infty}),$$
where 
$\gamma_{\lambda^2}$ is the centred Gaussian measure with variance $\lambda^2>0$. 

\begin{corollary} \label{c:HDS}
   \begin{equation}
   \mathbb S^n(r_n) \xrightarrow{n \to \infty}
       \begin{cases}
           \star \qquad &\text{in $\sfd_{\sf conc}$} \quad\ \text{if} \quad \frac{r_n}{\sqrt{n}} \to 0
           \\
           \Gamma^\infty_{\lambda^2} \qquad &\text{in $\tau_{\Pi}$} \qquad \text{if} \quad \frac{r_n}{\sqrt{n}} \to \lambda
           \\
           \mathbb I_\infty\qquad &\text{in $\tau_{\Pi}$} \qquad \text{if} \quad \frac{r_n}{\sqrt{n}} \to \infty
       \end{cases}
   \end{equation}
   where by convergence in $\tau_{\Pi}$ we mean the convergence of the associated pyramid. In particular, if $\frac{r_n}{\sqrt{n}} \to 1$, then 
   $$\mathbb S^n(r_n) \xrightarrow{n \to \infty} \mathbb W \quad \text{in $\tau_{\Pi}$},$$
   where $\mathbb W$ is the abstract Wiener space. 
\end{corollary}
\begin{proof}
This follows by combination of \cite[Theorem 1.1]{Shi17} with the representations~given by Propositions~\ref{p:ME}, \ref{c:PD} and \Cref{c:PWI}. 
\end{proof}

\subsection{Configuration space with Poisson measure}

\begin{proposition} \label{p:CNC}
The configuration space~$\U=(\U(\R), \tau_{\sf vague}, \mssd_\U, \pi)$ with the Poisson measure $\pi$ of unit intensity $($recall \Cref{subsec:CF}$)$ is not concentrated. 
\end{proposition}
\begin{proof}
Define a $\tau_{\sf vague}$-Borel function as follows: for each $n\in\mathbb Z$,
\[
u_n(\eta)\coloneqq\Bigl(\inf_{x\in\eta}|x-n|\Bigr)\wedge \frac14,
\qquad \eta\in \U.
\]
We show that $u_n\in\Lip_1(\U, \mssd_\U)$ for every $n$, and also show that the $\pi$-classes $[u_n]\in \mathcal L_1$ do not admit any subsequence converging in probability modulo constants. In particular, $\mathcal L_1$ is not compact.

 We first prove that each $u_n$ is $1$-Lipschitz.
Fix $n\in\mathbb Z$ and set
 \[
 r_n(\eta)\coloneqq\inf_{x\in\eta}|x-n|\in[0,\infty]\]
Let $\eta,\zeta\in \U$ with $\mathsf d_\U(\eta,\zeta)<\infty$. Choose representatives
\(
\eta=\sum_{i= 1}^\infty\delta_{x_i}\), \(\zeta=\sum_{i= 1}^\infty\delta_{y_i},
\)
and a bijection $\sigma$ such that
\[
\sum_{i = 1}^\infty |x_i-y_{\sigma(i)}|^2<\infty.
\]
Fix $\delta>0$ and choose $i$ with $|x_i-n|\le r_n(\eta)+\delta$. Then
\[
r_n(\zeta)\le |y_{\sigma(i)}-n|
\le |x_i-n|+|x_i-y_{\sigma(i)}|
\le r_n(\eta)+\delta+|x_i-y_{\sigma(i)}|.
\]
Since
\[
\Bigl(\sum_{j = 1}^\infty |x_j-y_{\sigma(j)}|^2\Bigr)^{1/2}\ge |x_i-y_{\sigma(i)}|,
\]
we obtain
\[
r_n(\zeta)\le r_n(\eta)+\delta+\Bigl(\sum_{j =1}^\infty |x_j-y_{\sigma(j)}|^2\Bigr)^{1/2}.
\]
Letting $\delta\downarrow 0$ and then taking the infimum over $\sigma$ yields
\[
r_n(\zeta)-r_n(\eta)\le \mathsf d_\U(\eta,\zeta).
\]
By symmetry also $r_n(\eta)-r_n(\zeta)\le \mathsf d_\U(\eta,\zeta)$, hence
\[
|r_n(\eta)-r_n(\zeta)|\le \mathsf d_\U(\eta,\zeta).
\]
The truncation $t\mapsto t\wedge \frac14$ is $1$-Lipschitz on $\mathbb R$, so
\[
|u_n(\eta)-u_n(\zeta)|
\le |r_n(\eta)-r_n(\zeta)|
\le \mathsf d_\U(\eta,\zeta),
\]
proving $u_n\in\Lip_1(\U, \mssd_\U)$.

We now prove that $(u_n)_{n \in \N}$ is independent and identically distributed as random variables on the probability space~$(\U, \pi)$. 
Let $I_n\coloneqq(n-\frac14,n+\frac14)$. Then $u_n(\eta)$ depends only on the restriction of $\eta$ to $I_n$.
Since $(I_n)_{n\in\mathbb Z}$ are disjoint and the Poisson process has independent increments in disjoint intervals, the family $(u_n(\eta))_{n\in\mathbb Z}$ is independent, and by construction, identically distributed. 
Let
\[
R\coloneqq\inf_{x\in\eta}|x|
\]
(distance from the origin to the nearest point in $\eta$). As $\pi$ is the Poisson distribution with intensity~$1$,  the distance between the nearest points to the left and right of $0$ is the exponential distributed (with intensity $1$)~${\rm Exp}(1)$, hence $R\sim{\rm Exp}(2)$ and
\[
\pi(R>t)=e^{-2t}\qquad (t\ge 0).
\]
Thus $u_0=R\wedge \frac14$ has an atom at $\frac14$ of mass
\[
q\coloneqq\pi (R\ge \tfrac14)=e^{-1/2},
\]
and on $[0,\frac14)$ it has density $f(t)=2e^{-2t}\le 2$. 

By \Cref{lemma: no convergence iid} below,  we conclude the desired claim. \qedhere

\end{proof}

\begin{lemma}\label{lemma: no convergence iid}
    Let $\X$ be a Polish space endowed with a Borel probability measure~$\mm_\X$.
    Let $(U_n)_{n\geq}$ be a sequence of independent and  identically distributed random variables on $\X$ which are not deterministic.
    Then the sequence $n\mapsto U_n$ is not precompact either in~$\L^0(\X)$ or $\L^0(\X)/\R$.
\end{lemma}
\begin{proof}
    We prove by contradiction. Assume that there exists a non-relabled $\L^0$ converging subsequence to a limit $U$.
    In particular, $U$ is independent of all of the random variables $U_n$ for $n\in \N$ and thus $U$ is independent from $U$.
    This implies that $U$ is deterministic, which is a contradiction. 
    Indeed, convergence in $\L^0$ implies weak convergence of the laws $n\mapsto (U_n)_\#\mm_\X$ to $U_\#\mm_\X$, which contradicts the assumptions on the $U_n$.

    Similarly, we prove the statement in~$\L^0(\X)/\R$ by contradiction. 
    Assume that there is a non-relabelled converging subsequence $U_n\to U$ in $\L^0(\X)/\R$ and let $n\mapsto c_n$ be such that $U_n-c_n$ converges to $U$ in $\L^0$. 
    In particular, the sequence $(U_n-c_n)_\#\mm_\X$ is tight and, since $U_n$ is identically distributed, $c_n$ is bounded. 
    Taking a converging (non-relabelled) subsequence $c_n \to c$, it follows $U_n \to U+c$ in probability, which contradicts the statement above in~$\L^0(\X)$.
\end{proof}

\subsection{General framework of Gaussian fields} \label{ss:GFF}
Let $H$ be a separable real Hilbert space with orthonormal basis
$(e_k)_{k\in \N}$.  Fix $(\lambda_k)_{k \in \N}$ with 
\[
1\le \lambda_1\le\lambda_2\le\cdots\lambda_k\longrightarrow\infty.
\]
For $t\in\mathbb R$, define
\[
H^t:=H^t_{(\lambda_k)}
:=
\left\{
x=\sum_{k=1}^\infty x_ke_k:
\sum_{k=1}^\infty\lambda_k^t x_k^2<\infty
\right\},
\]
with Hilbert norm
\[
\|x\|_{H^t}^2
:=
\sum_{k=1}^\infty\lambda_k^t x_k^2.
\]
We simply write $H^t$ instead of~$H^t_{(\lambda_k)}$ unless any confusion could occur.
For negative~$t<0$, $H^t$ is the completion of $H$ for this norm.  If $t_1>t_2$, then
$H^{t_1}\subset H^{t_2}$.

\paragraph{Gaussian measures.}
Let $(q_k)_{k\in \N}$ be positive numbers.  Fix $r\in\mathbb R$ such that
\begin{equation}\label{eq:support-assumption}
\sum_{k=1}^{\infty}q_k\lambda_k^{-r}<\infty.
\end{equation}
Define the series
\begin{align} \label{e:GRV}
\Phi:=\sum_{k=1}^{\infty}\sqrt{q_k}\,\xi_k e_k
\end{align}
 where  $\xi_k: \Omega \to \R$ are independent real-valued centred Gaussian random variables with variance $1$ defined in a probability space~$(\Omega, \mathbb P)$. The series converges
 in $L^2(\Omega;H^{-r})$ and almost surely in $H^{-r}$. 
 Indeed,  let
\[
\Phi_N:=\sum_{k=1}^N\sqrt{q_k}\,\xi_ke_k.
\]
Since \(\|e_k\|_{H^{-r}}^2=\lambda_k^{-r}\) and  \(e_k\) are
orthogonal in \(H^{-r}\),
\[
\|\Phi_m-\Phi_n\|_{H^{-r}}^2
=
\sum_{k=n+1}^m q_k\lambda_k^{-r}\xi_k^2.
\]
Moreover, by monotone convergence, the expectation can be evaluated as 
\[
\mathbb E\left[\sum_{k=1}^\infty
q_k\lambda_k^{-r}\xi_k^2\right]
:=\int_{\Omega} \biggl(\sum_{k=1}^\infty
q_k\lambda_k^{-r}\xi_k^2\biggr)\diff \mathbb P
=
\sum_{k=1}^\infty q_k\lambda_k^{-r}<\infty.
\]
Hence the series inside the expectation on the left is finite almost surely, so \((\Phi_N)\) is
almost surely Cauchy in \(H^{-r}\) and therefore converges almost surely to an element $\Phi \in H^{-r}$. Finally,
\[
\mathbb E\|\Phi-\Phi_N\|_{H^{-r}}^2
=
\sum_{k>N}q_k\lambda_k^{-r}\longrightarrow0,
\]
which proves convergence in \(L^2(\Omega;H^{-r})\).

Let $\mathfrak m$ be its law (i.e., the push-forward $\Phi_{\#} \mathbb P$) on the Hilbert space
\[
\X:=H^{-r}
\]
equipped with its norm topology $\tau_{H^{-r}}$.
Equivalently, $\mathfrak m$ is the centred Gaussian measure characterised by 
$$\mathbb E[\Phi_k\Phi_\ell]=q_k\delta_{kl}.$$
With respect to the $H^{-r}$ orthonormal basis~$f_k=\lambda_k^{r/2}e_k$, its covariance
operator~$Q$ on $H^{-r}$ satisfies
\[
Qf_k=q_k\lambda_k^{-r}f_k.
\]

\paragraph{The emm-property for weighted Sobolev extended distances.}
We start proving the emm-property for general weighted Sobolev extended distances. 
Let $(w_k)_{k\in \N}$ be positive numbers. Assume
 \begin{align} \label{e:FLW}
 C:=\sup_{k \in \N} \frac{\lambda_k^{-r}}{w_k}<\infty.
 \end{align}
 Define
\[
H_w
:=
\left\{
x=\sum_{k=1}^\infty x_ke_k:
\sum_{k=1}^\infty w_kx_k^2<\infty
\right\},
\qquad
\|x\|_w^2:=\sum_{k=1}^\infty w_kx_k^2.
\]
Noting $\|x\|_{H^{-r}}^2 = \sum_{k=1}^\infty\lambda_k^{-r}x_k^2 \le C \sum_{k=1}^\infty w_k^{-r}x_k^2 = C\|x\|_{w}^2$, we have 
$$H_w \hookrightarrow H^{-r}.$$ 
On $\X=H^{-r}$, define the extended distance~
\begin{equation}\label{eq:weighted-distance}
\mathsf d_w(x,y)
:=
\begin{cases}
\|x-y\|_w,&x-y\in H_w,\\
+\infty,&x-y\notin H_w.
\end{cases}
\end{equation}
Split $\X=\X_{(N)} \times \X^{(N)}$ such that \begin{align} \label{e:SX}
\X_{(N)}:={\rm span}\{e_1, e_2, \ldots, e_N\}, \qquad \X^{(N)}:=\{x \in H^{-r}: x_1=x_2=\cdots = x_N=0\}
\end{align}
endowed with 
$$\sfd_{w, (N)}(x,y):=\sqrt{\sum_{k=1}^N w_k(x_k-y_k)^2}$$
and 
\begin{equation} \sfd_{w}^{(N)}(x,y):=
\begin{cases}\displaystyle
\sqrt{\sum_{k > N}^\infty w_k(x_k-y_k)^2} \qquad &\text{if} \quad \sum_{k > N} w_k(x_k-y_k)^2<+\infty
\\
+\infty \qquad &\text{otherwise}.
\end{cases}
\end{equation}
The Gaussian measure~$\mm$ similarly decomposes by push-forward through projections on $\X_{(N)}$ and $\X^{(N)}$:
\[
\mathfrak m=\mathfrak m_{(N)}\otimes\mathfrak m^{(N)}.
\]
Define the tail scale
\begin{equation}\label{eq:tail-scale}
A_N:=\sup_{k>N}\sqrt{q_kw_k}\in[0,\infty].
\end{equation}

\begin{lemma}[Gaussian tail concentration]\label{lem:GT}
Assume $\sup_{k \in \N} \frac{\lambda_k^{-r}}{w_k}<\infty$ and $A_N<\infty$.  If
\[
h:\X^{(N)}\to\mathbb R
\]
is Borel and $1$-Lipschitz,
then for the L\'evy mean~$\mssm_h$ of $h$ and every $t>0$,
\[
\mathfrak m^{(N)}(|h-\mssm_h|>t)
\le
2\exp\left(-\frac{t^2}{2A_N^2}\right).
\]
\end{lemma}

\begin{proof}
The Cameron--Martin space of the tail Gaussian measure \(\mathfrak m^{(N)}\) is the following Hilbert space
\[
K^{(N)}
=
\left\{
u=\sum_{k>N}u_ke_k:
\sum_{k>N}\frac{u_k^2}{q_k}<\infty
\right\},
\qquad
\|u\|_{K^{(N)}}^2
=
\sum_{k>N}\frac{u_k^2}{q_k}.
\]
For \(u\in K^{(N)}\), the assumption \(A_N<\infty\) gives
\[
\|u\|_{w,(N)}^2
=
\sum_{k>N}w_ku_k^2
\le
A_N^2
\sum_{k>N}\frac{u_k^2}{q_k}
=
A_N^2\|u\|_{K^{(N)}}^2.
\]
Hence, since \(h\) is \(1\)-Lipschitz with respect to \(\mathsf d_w^{(N)}\),
\[
|h(x+u)-h(x)|
\le
A_N\|u\|_{K^{(N)}}
\]
for Cameron--Martin shifts \(u\in K^{(N)}\). Thus \(h\) is \(A_N\)-Lipschitz along the Cameron--Martin directions of the tail abstract Wiener space. The Gaussian isoperimetric inequality for the abstract Wiener space~(see, e.g., \cite[Theorem 4.5.6]{Bog98}) therefore yields
\[
\mathfrak m^{(N)}(|h-\mssm_h|>t)
\le
2\exp\left(-\frac{t^2}{2A_N^2}\right).
\]
\end{proof}

\begin{theorem} \label{t:GT}
Let $(\lambda_k)_{k \in \N}$,  $(q_k)_{k \in \N}$ and $(w_k)_{k \in \N}$ be positive numbers as above, and $H^{-r}$, $\mm$ and $\sfd_w$ be the associated Sobolev space with order $-r$, the Gaussian measure and the Sobolev-type extended distance as given above. 
Fix $r>0$ such that 
$$\sum_{k = 1}^\infty q_k\lambda_k^{-r}<+\infty, \qquad \sup_{k \in \N} \frac{\lambda_k^{-r}}{w_k}<\infty$$ and let $a_k:=\sqrt{q_k w_k}$. 
Then the following hold:
\begin{enumerate}
    \item \label{GT0} $(\X,\tau_{H^{-r}},\mathsf d_w,\mathfrak m)$ is an extended topological-metric-measure space.
    \item \label{GT1} $\X$ is concentrated $($i.e., $\mathcal L_1(\X)$ is compact$)$ if and only if 
    $$a_k \xrightarrow{k \to \infty} 0;$$
    \item \label{GT2} $\sfd_w(x,y)<+\infty$ for $\mm^{\otimes 2}$-a.e.~$(x,y)$ if and only if 
    $$\sum_{k=1}^\infty a_k^2<+\infty;$$
    \item \label{GT3} $\sfd_w(x,y)=+\infty$ for $\mm^{\otimes 2}$-a.e.~$(x,y)$ if and only if 
    $$\sum_{k=1}^\infty a_k^2=+\infty.$$
\end{enumerate}
\end{theorem}
\begin{proof}
{\it Proof of~\ref{GT0}}. 
The space $\X=H^{-r}$ is Polish and $(H_w, \sfd_{w})$ is complete. 
Symmetry and separation are immediate.
If $\mathsf d_w(x,y),\mathsf d_w(y,z)<\infty$, then
$x-y,y-z\in H_w$, hence $x-z\in H_w$, and the triangle inequality follows
from the Hilbert norm $\|\cdot\|_w$.  If one of the two distances on the
right is infinite, the triangle inequality is automatic.
For every $N$ define
\[
S_N(x,y)
:=
\sum_{k=1}^Nw_k
\left|
 x_k-y_k
\right|^2.
\]
The coordinate functional $x\mapsto x_k$ is continuous on
$H^{-r}$, so $(S_N)^{1/2}$ is continuous on $\X\times \X$ and, obviously, 1-Lipschitz.
Moreover,
\[
\mathsf d_w(x,y)^2=\sup_{N\in \N}S_N(x,y).
\]
 
In particular, $\sfd_w$ is $\tau_{H^{-r}}$-recovered (recall \Cref{d:PST}).

Let now $C=\sup_{k \in \N} \frac{\lambda_k^{-r}}{w_k}<\infty$.
By an easy estimate, it holds 
$$\|\cdot\|_{H^{-r}}\leq \sqrt{C} \|\cdot \|_w.$$
In particular, $\Lip_1(\X, \|\cdot\|_{H^{-r}})\subset \Lip(\X, \tau_{H^{-r}}, \sfd_w)$.
This implies that $\tau_{H^{-r}}$ is the initial topology of $\Lip(\X, \tau_{H^{-r}}, \sfd_w)$.
This proves that $(\X, \tau_{H^{-r}}, \sfd_w, \mm)$ is an extended topological-metric-measure space in the sense of \Cref{d:ETMS}.

{\it Proof of \ref{GT1}.} Taking $A_N:=\sup_{k>N}a_k$, the sufficiency is a consequence of   \Cref{p:CPC} combined with \Cref{lem:GT}. 

We now prove the necessity. Assume $a_k\not\to0$.  Then there exist $c>0$ and a sequence of distinct indices $j\mapsto k_j$ such
that
\[
a_{k_j}\ge c\qquad\forall j.
\]
Define $U_j=\frac{c}{a_{k_j}}\sqrt{w_{k_j}} x_{k_j}$. Recalling~\eqref{e:GRV}, $j \mapsto U_j$ is independent and identically distributed as the centred Gaussian random variable with variance~$c^2$ under $\mm$. Furthermore, $\Lip_{\sfd_w}(U_j) \le \frac{c}{a_{k_j}} \le 1$, thus $U_j$ is $1$-Lipschitz with respect to~$\sfd_w$. 
By~\Cref{lemma: no convergence iid}, therefore, $j\mapsto U_j$ admits no converging subsequence in $L^0/\R$, thus, $\mathcal L_1(\X)$ is not compact, i.e., $\X$ is not concentrated. 

{\it Proof of \ref{GT2} and \ref{GT3}.}  Recalling the expression of our Gaussian random variables in \eqref{e:GRV}, it is equivalent to show that, for  any two independent Gaussian random variables $\Phi$ and $\tilde\Phi$ having the law $\mm$, we have $\sfd_w(\Phi, \tilde{\Phi})<+\infty$ almost surely in $(\Omega, \mathbb P)$. By a simple computation, 
\[
\Phi-\widetilde\Phi
=
\sum_{k=1}^\infty\sqrt{q_k}\,(\xi_k-\widetilde\xi_k)e_k.
\]
Thus
\begin{align}\label{e:ISD}
\mathsf d_w(\Phi,\widetilde\Phi)^2
=
\sum_{k=1}^{\infty}q_kw_k(\xi_k-\widetilde\xi_k)^2.
\end{align}
The expectation of this nonnegative random
variable is
\[
2\sum_{k=1}^\infty a_k^2,
\]
which is finite if $\sum_{k=1}^\infty a_k^2<\infty$.  In this case,  the sum~\eqref{e:ISD} is finite almost surely. 

Conversely, assume $\sum_{k=1}^\infty a_k^2=\infty$.  Choose $c_0>0$ such that
\[
p:=\mathbb P((\xi_1-\widetilde\xi_1)^2\ge c_0)>0.
\]
Let
\[
B_k:=\mathbf 1_{\{(\xi_k-\widetilde\xi_k)^2\ge c_0\}}.
\]
The random variables~$B_k$ are independent Bernoulli random variable with success probability $p$, and
\[
\mathsf d_w(\Phi,\widetilde\Phi)^2
\ge
c_0\sum_{k=1}^\infty a_k^2B_k.
\]
This implies that the right-hand side is infinite almost
surely. Indeed, to see this, letting $S=\sum_{k=1}^\infty a_k^2B_k$, taking the moment expectation yields 
$$\mathbb E[e^{-S}]=\prod_{k=1}^\infty (1-p+pe^{-a_k^2}) \le \exp\Bigl(-p\sum_{k=1}^\infty (1-e^{-a_k^2})\Bigr).$$
Using $1-e^{-x} \ge c \min\{x, 1\}$ $(x \ge 0)$ with an universal constant $c>0$ and the hypothesis $\sum_{k=1}^\infty a_k^2 =+\infty$, we conclude that the LHS equals $0$. Thus $S=+\infty$ almost surely, which concludes $\mathsf d_w(\Phi,\widetilde\Phi)=+\infty$ almost surely.  
\end{proof}

\begin{corollary}
    Suppose that  the assumptions of \Cref{t:GT} hold. Then
    \begin{itemize}
        \item if $\sum_{k=1}^\infty a_k^2<+\infty$, then $(\X,\tau_{H^{-r}},\mathsf d_w,\mathfrak m)$ is emm-isomorphic to the mm-space~$(H_w, \mathsf d_w,\mathfrak m|_{H_w})$.
        \item if $\sum_{k=1}^\infty a_k^2=+\infty$, then $(\X,\tau_{H^{-r}},\mathsf d_w,\mathfrak m)$ is not emm-isomorphic to any mm-space, i.e., truly extended. 
        In particular, if also $a_n\to 0$, then 
        $$(\X,\tau_{H^{-r}},\mathsf d_w,\mathfrak m) \in \partial\mathfrak X_{\sf conc}.$$
    \end{itemize}
\end{corollary}
\begin{proof}
    When $\sum_{k=1}^\infty a_k^2<+\infty$, we have $\mathbb E\|\Phi\|_w^2 = \sum_{k=1}^\infty a_k^2<+\infty$, thus $\mm(H_w)=1$. Hence $\mm$ can be regarded as a probability measure on~$H_w$ by the restriction~$\mm|_{H_w}$. Therefore,  the~embedding $H_w \hookrightarrow H^{-r}$ is an emm-isomorphism between  the mm-space~$(H_w, \mathsf d_w,\mathfrak m|_{H_w})$ and  the emm-space~$(\X, \tau_{H^{-r}}, \mathsf d_w,\mathfrak m)$, which proves the first item.

    The second item is immediate, since the condition $\sfd=+\infty$ $\mm\times \mm$-a.e.~is invariant up to emm-isomorphism and is false for mm-spaces.
\end{proof}

\begin{corollary}\label{corollary: ideality gaussian}
Let
$\X=\mathbb{R}^{\infty}$ with the product topology $\tau^\infty$ and 
\[\ell_2(x-y)\coloneqq\Big(\sum_{n=1}^\infty (x_n-y_n)^2\Big)^{1/2}\in[0,\infty],
\]
and let 
$$\gamma^\infty := \bigotimes_{n=1}^\infty\gamma_{\sigma_n^2},$$
where $\gamma_{\sigma_n^2}$ is  the one-dimensional centred Gaussian measure with variance $\sigma_n^2$. 
Then 
\begin{enumerate}
    \item $\X=(\R^\infty, \tau^\infty,\ell_2, \gamma^\infty)$ is concentrated if and only if  $\sigma_n \to 0$ as $n \to \infty$
    \item  $\ell_2=+\infty$ $\gamma^{\infty} \otimes \gamma^{\infty}$-a.e.,~ if and only if 
    $$\sum_{n=1}^\infty \sigma_n^2=+\infty.$$
\end{enumerate}
\end{corollary}
\begin{proof}
Take  $w_n=1$, $r=1$, $q_n=\sigma_n^2$,  and $\lambda_n=2^n(1+\max_{1 \le j \le n}\sigma_j^2)$. 
Then, 
$$\sum_{n=1}^\infty q_n \lambda_n^{-1} \le \sum_{n=1}^\infty 2^{-n}<\infty, \qquad \sup_{n \in \N} \lambda_n^{-1}<\infty.$$
Thus, by~\Cref{t:GT}, we obtain $\X=H^{-1},$ $\tau=\tau_{H^{-1}}$, $\sfd_w=\ell_2$, $\mm=\gamma^\infty$ and $a_n=\sqrt{q_nw_n}=\sigma_n$.
The map $J: H^{-1} \to \R^\infty$ defined as 
$$x=\sum_{n=1}^\infty x_ne_n \mapsto (x_n)_{n \in \N}$$
is a measure-preserving isometry, hence $(\R^\infty, \tau^\infty, \ell_2, \gamma^\infty)$ is emm-isomorphic to $(H^{-1}, \tau_{H^{-1}}, \sfd_w, \mm)$, which  concludes the claim. 
\end{proof}

\begin{corollary}
Let $\mathbb W=(W,H,\mu)$ be an abstract Wiener space with infinite-dimensional
Cameron--Martin space $H$ $($recall \Cref{d:WS}$)$. 
Then the emm-space~$\mathbb W=(W, \tau_W, \sfd_H, \mu)$ is not concentrated  and $\sfd_H=+\infty$ $\mu^{\otimes 2}$-a.e.. In particular, $\mathbb W$ is truly extended. 
\end{corollary}
\begin{proof}
Since every abstract Wiener space with inifinite-dimensional Cameron--Martin space~$H$ is emm-ismorphic to $\X=(\R^\infty, \tau^\infty,\ell_2, \gamma^{\otimes \infty}_{1})$ by \cite[Proposition 8.3]{SuYo25+}, the claim follows by \Cref{corollary: ideality gaussian}.
\end{proof}

\paragraph{Convergence in concentration.}
\begin{theorem}[Continuity under spectral convergence]\label{t:CT} Suppose the same assumptions in \Cref{t:GT}.
For each $n$, let $w^{(n)}:=(w_k^{(n)})_{k\in \N}$ be positive weights and let
$\mathsf d_n:=\sfd_{w^{(n)}}$ be the corresponding extended metric. 
Assume:

\begin{enumerate}[label=(\roman*)]
\item for every fixed $k$,
\[
w_k^{(n)}\longrightarrow w_k;
\]
\item (Uniform tail condition) there exists $n_0$ such that
\[
\lim_{N\to\infty}
\sup_{n\ge n_0}\sup_{k>N}\sqrt{q_kw_k^{(n)}}=0.
\]
\end{enumerate}
Then
\[
(\X,\tau_{H^{-r}},\mathsf d_n,\mathfrak m)
\xrightarrow{\sfd_{\sf conc}}
(\X,\tau_{H^{-r}},\mathsf d,\mathfrak m).
\]

\end{theorem}

\begin{proof} 
Fix $\varepsilon>0$.
By the uniform tail assumption, using \eqref{e:UEC} and \Cref{lem:GT}, we can choose $N$ so large that every
$1$-Lipschitz observable for any $\mathsf d_n$ with $n\ge n_0$, and every
$1$-Lipschitz observable for $\mathsf d$, is $\varepsilon$-close in probability to a
$1$-Lipschitz function depending only on the first $N$ coordinates. 
In other words, for every~$\varepsilon$, there exists $N_0=N_0(\varepsilon)$ independent of $n$ such that for every $N \ge N_0$
$$\sfd_{\sf conc}\big((\X, \tau_{H^{-r}}, \sfd_n, \mm), (\X_{(N)}, \sfd_{w^{(n)}, (N)}, \mm_{(N)})\big)\leq \varepsilon.$$
The uniform bound in the assumptions implies also that $$\lim_{N\to\infty}\sup_{k>N}\sqrt{q_kw_k}=0.$$ 
Thus, similarly, 
$$\sfd_{\sf conc}\big((\X, \tau_{H^{-r}}, \sfd, \mm), (\X_{(N)}, \sfd_{w, (N)}, \mm_{(N)})\big)\leq \varepsilon.$$

It remains to compare the finite-dimensional metrics~$\mathsf d_{w,(N)}$ and $\mathsf d_{w^{(n)},(N)}$.  On
$\X_{(N)}$ define
\[
T_n\left(\sum_{k=1}^Nx_ke_k\right)
:=
\sum_{k=1}^N
\sqrt{\frac{w_k^{(n)}}{w_k}}\,x_ke_k.
\]
Then
\[
\mathsf d_{w,(N)}(T_nu,T_nv)=\mathsf d_{w^{(n)},(N)}(u,v).
\]
Thus, $T_n$ is an isometric embedding of $(\X_{(N)}, \sfd_{w^{(n)}, N})$ into~$(\X_{(N)}, \sfd_{w, N})$.
Moreover, by coordinatewise convergence, it holds $T_n\to id$ pointwise, thus $$(T_n)_\#\mm_{(N)}\rightharpoonup \mm_{(N)}.$$
In particular, $n\mapsto (\X_{(N)}, \sfd_{w^{(n)}, N},\mm_{(N)})$ converges in the measure Gromov-Hausdorff sense to $(\X_{(N)}, \sfd_{w, N},\mm_{(N)})$, which implies 
\begin{equation}
    \square\big((\X_{(N)}, \sfd_{w^{(n)}, (N)}, \mm_{(N)}), (\X_{(N)}, \sfd_{w, (N)}, \mm_{(N)})\big)\xrightarrow{n \to \infty} 0 \qquad \forall N \in \N.
\end{equation}
Finally, 
\begin{align}
    &\sfd_{\sf conc}\big( (\X, \tau_{H^{-r}}, \sfd_n, \mm), (\X, \tau_{H^{-r}}, \sfd, \mm)\big)
    \\
    &\leq \sfd_{\sf conc}\big((\X, \tau_{H^{-r}}, \sfd_n, \mm), (\X_{(N)}, \sfd_{w^{(n)}, (N)}, \mm_{(N)})\big)
    \\
    &\quad+\sfd_{\sf conc}\big((\X_{(N)}, \sfd_{w^{(n)}, (N)},\mm_{(N)}), (\X_{(N)}, \sfd_{w,  (N)},\mm_{(N)})\big)
    \\
    &\quad +\sfd_{\sf conc}\big((\X, \tau_{H^{-r}}, \sfd, \mm), (\X_{(N)}, \sfd_{w, (N)}, \mm_{(N)})\big)\\
    &\leq 2\varepsilon+ \square\big((\X_{(N)}, \sfd_{w^{(n)}, (N)}, \mm_{(N)}), (\X_{(N)}, \sfd_{w,  (N)}, \mm_{(N)})\big),
\end{align}
where the inequality $\sfd_{\sf conc} \le \square$ in \eqref{e:BXD} is used in the last line. 
Sending $n\to \infty$ and then $\varepsilon \to 0$ proves the convergence.
\end{proof}

\begin{remark}
    One can easily check  that the sequence $n\mapsto (\X,\tau_{H^{-r}},\mathsf d_n,\mathfrak m)$ fibrates over $(\X,\tau_{H^{-r}},\mathsf d,\mathfrak m)$ with the constant sequence of projections $n\mapsto p_n$ where $p_n\equiv id$, which induces concentration.
    The same holds for the sequence of projections $n\mapsto q_n$, where $q_n=T_n\times id_{\X^{(N)}}$ and $T_n$ are as in the proof above.
\end{remark}
\paragraph{Inverse limit construction.}
Recall $\X_{(N+1)}$ be the space of the first $N$-coordinates given in~\eqref{e:SX} and 
let $P_{N+1, N}: \X_{(N+1)} \to \X_{(N)}$ be the canonical projection $x=\sum_{k=1}^{N+1} x_k e_k \mapsto \sum_{k=1}^{N} x_k e_k$. 
\begin{corollary}
 Under the same setting as \Cref{t:GT}, the following hold:
 \begin{enumerate}
     \item $(\X_{(N)}, \sfd_{w, (N)}, \mm_{(N)})$ and $P_{N+1, N}$ form an inverse system;
     \item $\X \cong \varprojlim \X_{(N)}$ $($emm-isomorphic$)$. 
     In particular, 
     $$\overline{\mathcal P_\X}^{\square}= \overline{\bigcup_{n\in \N} \mathcal P_{\X_n}}^\square = \pi_{\Pi}\text{-}\lim_{n \to \infty} \mathcal P_{\X_n}.$$
     \item If $\X$ is concentrated (i.e., $a_k \to 0$ as $k \to \infty$), then 
     $$\X_{(N)} \xrightarrow{\sfd_{\sf conc}} \X.$$
     In particular, 
     $$\mathcal P_\X= \overline{\bigcup_{n\in \N} \mathcal P_{\X_n}}^\square = \pi_{\Pi}\text{-}\lim_{n \to \infty} \mathcal P_{\X_n}.$$
 \end{enumerate}
\end{corollary}
\begin{proof}
    The first assertion is straightforward by definition of $\X_{(N)}$.
    Let $J: H^{-r} \to \varprojlim \X_{(N)}$ by 
    $$x=\sum_{k =1}^\infty x_ke_k \mapsto \Bigl(x_1e_1, x_1e_1+x_2e_2, \ldots, \sum_{k=1}^N x_ke_k, \ldots\Bigr).$$
    It is an isometry by definition, and 
    by the hypothesis of~\Cref{t:GT} the image $J(H^{-r})$ is a set of full measure with repect to the inverse limit measure~$\mm_\infty$. As $(P_{N+1, N})_\# \mm = \mm_{(N)}$ for every $N$, we have $J_\#\mm=\mm_\infty$. Therefore $J$ is an emm-isomorphism. 
    The third assertion is a direct application of \Cref{t:2}. The pyramid expressions follow by \Cref{c:IVI} and \Cref{thm: concentrated iff concentrated}.
\end{proof}
\paragraph{Fractional Gaussian fields.}
Assume that the eigenvalues satisfy the Weyl-type asymptotics
\begin{align} \label{e:WTA}
\lambda_k\asymp k^{2/d}
\end{align}
for some spectral dimension $d>0$. 
Assume now that
\[
q_k=\lambda_k^{-\alpha},
\qquad
w_k=\lambda_k^s.
\]
Then
\[
a_k^2=q_kw_k=\lambda_k^{s-\alpha}.
\]
Let $\X_s=(H^{-r}, \tau_{H^{-r}}, \sfd_{s}, \mm)$ be the corresponding emm-space for $s \in \R$ and assume $s+r \ge 0$ to have $\tau_{H^{-r}} \subset \tau_{\mssd_s}$ and $\alpha+r >d/2$ to have~$\mm(H^{-r})=1$. 
\begin{corollary}\label{cor:frac} Under the above assumption, the following hold:
\begin{enumerate}
 \item $\X_s$ is concentrated if and only if $s<\alpha$;
    \item $\sfd_s=+\infty$ $\mm^{\otimes 2}$-a.e.,~if and only if $s \ge \alpha - \frac{d}{2}$;
    \item if $s_n \to s \in (-\infty, \alpha)$ with $s_n + r \ge 0$ for sufficiently large $n$, then 
    $$\X_{s_n} \xrightarrow{\sfd_{\sf conc}} \X_s.$$
\end{enumerate}
\end{corollary}

\begin{proof}
The compactness condition is
$
\lambda_k^{(s-\alpha)/2}\to0$, 
equivalent to $s<\alpha$.
By~\Cref{t:GT}, $\sfd_{s}=+\infty$ $\mm^{\otimes 2}$-a.e.~exactly when
\[
\sum_{k=1}^\infty\lambda_k^{s-\alpha}=+\infty.
\]
By the Weyl-type growth~\eqref{e:WTA},  this is equivalent to
\[
\sum_{k=1}^\infty k^{2(s-\alpha)/d}=+\infty,
\]
which holds exactly when
\[
\frac{2(s-\alpha)}d \ge -1,
\]
that is, $s \ge \alpha-\frac d2$.
If $s_n\to s<\alpha$, choose $\overline s$ with
$
s,s_n\le\overline s<\alpha
$
for all sufficiently large $n$.  Then
\[
\sup_{k>N}\lambda_k^{(s_n-\alpha)/2}
\le
\lambda_{N+1}^{(\overline s-\alpha)/2}
\longrightarrow0
\]
uniformly in $n$, and \Cref{t:CT} applies.
\end{proof}

\subsection{Examples for Gaussian fields} \label{ss:EGF} 
In this section, we discuss several examples of fractional Gaussian fields originated in probability theory. 
Let \(M\) be a compact \(d\)-dimensional Riemannian manifold, possibly with
smooth boundary, equipped with its volume measure~\(\mathrm{vol}_M\). Let \(A\) be a positive self-adjoint operator on
\(L^2(M,\mathrm{vol}_M)\) with compact resolvent.  We choose an
\(L^2(M,\mathrm{vol}_M)\)-orthonormal eigenbasis \((e_k)_{k\in \N}\) of \(A\),
\[
Ae_k=\lambda_k e_k,
\qquad
\langle e_k,e_\ell\rangle_{L^2(M,\mathrm{vol}_M)}=\delta_{k\ell},
\]
and assume the Weyl-type growth
\[
\lambda_k\asymp k^{2/d}.
\]
If \(M\) has boundary, the boundary condition is included in the choice of the
domain of \(A\).  

The extended Sobolev distance with $w_k=\lambda_k^s$ is
\[
\mathsf d_s(\phi,\psi)
=
\begin{cases}
\|\phi-\psi\|_{H^s}=\left(
\displaystyle\sum_{k=1}^{\infty}
\lambda_k^s
|\langle \phi-\psi,e_k\rangle|^2
\right)^{1/2},
&
\displaystyle
\sum_{k=1}^{\infty}
\lambda_k^s
|\langle \phi-\psi,e_k\rangle|^2<\infty,
\\[3mm]
+\infty,
&
\text{otherwise.}
\end{cases}
\]
In this setting,
\[
\mathcal L_1\text{ is compact}\iff s<\alpha,
\qquad
\mathsf d_s=+\infty\quad\mathfrak m^{\otimes2}\text{-a.e.}
\iff s\ge\alpha-\frac d2.
\]
Hence $\X$ is concentrated and truly extended (i.e., $\X \in \partial \mathfrak X_{\sf conc}$) if and only if
\[
\alpha-\frac d2\le s<\alpha.
\]
\begin{remark}[Associated diffusion/stochastic quantisation]
Here we briefly discuss a canonical Dirichlet form associated with~$(\X, \sfd_{w}, \mm)$.  Let
\[
F(x)=\varphi(x_1,\ldots,x_N),
\]
where $\varphi\in C_b^2(\mathbb R^N)$.  The dual norm to
\(
\|h\|_w^2=\sum_{k=1}^\infty w_kh_k^2
\)
is
\[
\|\alpha\|_{w,*}^2=\sum_{k=1}^\infty w_k^{-1}\alpha_k^2.
\]
Consequently the squared slope of $F$ is
\[
|DF|_w^2
=
\sum_{k=1}^Nw_k^{-1}|\partial_k\varphi|^2.
\]
The corresponding Sobolev form is
\[
\mathcal E_w(F)
=
\frac12
\int
\sum_{k=1}^Nw_k^{-1}|\partial_k\varphi|^2\,\d\mathfrak m.
\]
It is closable in $\L^2(\X, \mm)$ (see \cite[Corollary~1.7]{AlbeverioRockner1989}). By the standard integration by parts with the Gaussian law~$\mm$, the self-adjoint generator on smooth cylinder functions is explicitly 
\[
\mathcal L_wF
=
\sum_{k=1}^N
w_k^{-1}
\left(
\partial_{kk}\varphi
-\frac{x_k}{q_k}\partial_k\varphi
\right).
\]
In the fractional case $q_k=\lambda_k^{-\alpha}$ and
$w_k=\lambda_k^s$,
\[
\mathcal L_sF
=
\sum_{k=1}^N
\lambda_k^{-s}
\left(
\partial_{kk}\varphi
-\lambda_k^\alpha x_k\partial_k\varphi
\right),
\]
and the associated diffusion semigroup satisfies
\[
P_t x_k=e^{-t\lambda_k^{\alpha-s}}x_k.
\]
Thus the compact $\mathcal L_1$ regime $s<\alpha$ is exactly the regime in which
the damping rates of high modes tend to infinity.
The associated diffusion  on $\X$ to this Dirichlet form is the $\mm$-reversible
Ornstein--Uhlenbeck evolution
\[
d\Phi_t=-A^{\alpha-s}\Phi_t\,dt
+\sqrt2\,A^{-s/2}\,dW_t,
\]
whose invariant law~$\mm$ is the law of \(A^{-\alpha/2}\xi\). Equivalently, letting $\Phi_t=\sum_{k=1}^\infty X_k(t) e_k$, the coordinate-wise SDE expression is 
\[
dX_k(t)=-\lambda_k^{\alpha-s}X_k(t)\,dt
+\sqrt{2\lambda_k^{-s}}\,dW^{(k)}_t, \quad k \in \N,
\]
where $W^{(1)}_t, W^{(2)}_t, \ldots$ are   independent standard Brownian motions. 
Fractional Gaussian fields and their white-noise representations can be found
in \cite{LodhiaSheffieldSunWatson2016}.  The connection between Gaussian
random fields, infinite-dimensional Ornstein--Uhlenbeck semigroups, and
symmetric Markov processes is considered in \cite{Kolsrud1988}.  The general
Dirichlet-form framework on rigged Hilbert and topological vector spaces,
including closability and construction of the associated diffusion, is
developed in
\cite{AlbeverioHoeghKrohn1977,AlbeverioRockner1989,AlbeverioRockner1990}.
The identification of such diffusions with infinite-dimensional stochastic
differential equations is studied in \cite{AlbeverioRockner1991}. For the
semigroup and mild-solution theory of linear stochastic evolution equations,
see \cite{DaPratoZabczyk2014}. 

\end{remark}
\begin{example}[Standard Brownian motion]
Let \(d=1\), \(\alpha=1\), $M=[0,1]$ and
\[
A=-\frac{d^2}{dr^2},
\qquad u(0)=0,\quad u'(1)=0.
\]
Then \(A^{-1/2}\xi\) has the law of standard Brownian motion on
\(C_0([0,1]) \subset H^{-r}([0,1])\).  The space \(\mathcal L_1\) is compact for \(s<1\), and $\sfd_w=+\infty$ $\mm^{\otimes 2}$-a.e. for 
\[
\frac12\le s.
\]
The corresponding path-space diffusion is
\[
d\omega_t=-A^{1-s}\omega_t\,dt
+\sqrt2\,A^{-s/2}\,dW_t.
\]
At \(s=1\) this becomes the classical Ornstein--Uhlenbeck process on Wiener
space. In this case, the emm-space is truly extended and the  compactness of \(\mathcal L_1\) fails.
\end{example}

\begin{example}[Brownian bridge]
Let \(d=1\), \(\alpha=1\), $M=[0,1]$ and
\[
A=-\frac{d^2}{dr^2},
\qquad u(0)=u(1)=0.
\]
Then \(A^{-1/2}\xi\) is a Brownian bridge.  Again,
\[
\mathcal L_1\text{ is compact for }s<1,
\qquad
\frac12\le s<1
\]
is the truly extended compact range.  The diffusion is
\[
d\omega_t=-A^{1-s}\omega_t\,dt
+\sqrt2\,A^{-s/2}\,dW_t,
\]
now with Dirichlet boundary conditions at both endpoints.
\end{example}

\begin{example}[Massive Gaussian free field]
Let \(M\) be a compact \(d\)-dimensional Riemannian manifold and
\[
A=I-\Delta_M,\qquad \Phi=A^{-1/2}\xi ,
\]
where $\Delta_M$ is the Laplace-Beltrami operator. 
The Gaussian field is realized on \(H^{-r}(M)\), with \(1+r>d/2\).  One has
\[
\mathcal L_1\text{ compact}\iff s<1,
\qquad
\mathsf d_s=+\infty\ \mathfrak m^{\otimes2}\text{-a.e.}
\iff s\ge1-\frac d2.
\]
Thus the truly extended and compact~$\mathcal L_1$ range is
\[
1-\frac d2\le s<1.
\]
The diffusion is
\[
d\Phi_t=-A^{1-s}\Phi_t\,dt
+\sqrt2\,A^{-s/2}\,dW_t.
\]
For \(s=0\), this is the stationary stochastic heat equation with invariant
law $\mm$ equal to the massive GFF.
\end{example}

\begin{example}[Spatial white noise]
Let \(M\) be a compact \(d\)-dimensional Riemannian manifold, and let
\(A\) be a positive self-adjoint second-order elliptic operator on
\(L^2(M)\) with compact resolvent. 
Take
\[
\Phi=\xi,\qquad \alpha=0.
\]
White noise is realized on \(H^{-r}(M)\) for \(r>d/2\).  Then
\[
\mathcal L_1\text{ compact}\iff s<0,
\qquad
-\frac d2\le s<0
\]
is the truly extended compact $\mathcal L_1$ range.  The diffusion is
\[
d\Phi_t=-A^{-s}\Phi_t\,dt
+\sqrt2\,A^{-s/2}\,dW_t.
\]
In dimension \(2\), this gives the range \(-1\le s<0\), related to
white-noise vorticity and enstrophy-type equilibrium measures.
\end{example}

\begin{example}[Massive membrane or massive bi-Laplacian field]
Let \(M\) be a compact \(d\)-dimensional Riemannian manifold,
\[
A=I-\Delta_M,\qquad \Phi=A^{-1}\xi ,
\]
so that \(\alpha=2\) and the covariance is \(A^{-2}\).  The observable space~$\mathcal L_1(\X)$
is compact for \(s<2\), while the truly extended range is
\[
2-\frac d2\le s.
\]
The diffusion is
\[
d\Phi_t=-A^{2-s}\Phi_t\,dt
+\sqrt2\,A^{-s/2}\,dW_t.
\]
This is the natural Gaussian dynamics preserving the membrane-field law.
\end{example}

\newpage
\begin{landscape}
\begin{table}[p]
\centering
\caption{Gaussian-field.
\(\mathcal L_1\) is compact for \(s<\alpha\), and the truly extended (i.e., $\sfd_{w}=+\infty$ $\mm^{\otimes 2}$-a.e.) if \(\alpha-\frac d2\le s\).}
\label{tab:gaussian-emm-examples}

\footnotesize
\setlength{\tabcolsep}{3.5pt}
\renewcommand{\arraystretch}{1.25}

\begin{tabularx}{\linewidth}{
@{}
>{\centering\arraybackslash}p{0.55cm}
>{\centering\arraybackslash}p{0.75cm}
>{\raggedright\arraybackslash}p{3.6cm}
>{\raggedright\arraybackslash}p{2.8cm}
>{\raggedright\arraybackslash}p{3.25cm}
>{\centering\arraybackslash}p{1.45cm}
>{\centering\arraybackslash}p{2.35cm}
>{\raggedright\arraybackslash}X
@{}}

\toprule
\multicolumn{2}{c}{Parameters}
&
\multicolumn{3}{c}{Gaussian model}
&
\multicolumn{2}{c}{Sobolev regimes}
&
\multicolumn{1}{c}{Diffusion}
\\

\cmidrule(lr){1-2}
\cmidrule(lr){3-5}
\cmidrule(lr){6-7}
\cmidrule(l){8-8}

\(d\)
&
\(\alpha\)
&
Operator \(A\) with eigenvalues~$(\lambda_k)$
&
$\mm= {\rm Law}(\Phi)$ invariant measure
&
Model name
&
\(\mathcal L_1\) compact
&
Truly extended ($\sfd_{w}=+\infty$ $\mm^{\otimes 2}$-a.e.)
&
heat flow/diffusion (stochastic quantisation)
\\
\midrule

\(1\)
&
\(1\)
&
\(\displaystyle
 A=-\frac{d^2}{dr^2}\),
\quad
\(u(0)=0,\ u'(1)=0\)
&
\(\Phi=A^{-1/2}\xi\)
&
Standard Brownian motion on \([0,1]\)
&
\(s<1\)
&
\(\displaystyle \frac12\le s\)
&
\(\displaystyle
 d\Phi_t
 =-A^{1-s}\Phi_t\,dt
 +\sqrt2\,A^{-s/2}\,dW_t
\)
\\
\addlinespace[3pt]

\(1\)
&
\(1\)
&
\(\displaystyle
 A=-\frac{d^2}{dr^2}\),
\quad
\(u(0)=u(1)=0\)
&
\(\Phi=A^{-1/2}\xi\)
&
Brownian bridge
&
\(s<1\)
&
\(\displaystyle \frac12\le s\)
&
\(\displaystyle
 d\Phi_t
 =-A^{1-s}\Phi_t\,dt
 +\sqrt2\,A^{-s/2}\,dW_t
\)
\\
\addlinespace[3pt]

\(d\)
&
\(1\)
&
\(A=I-\Delta_M\) on a compact
\(d\)-dimensional manifold \(M\)
&
\(\Phi=A^{-1/2}\xi\)
&
Massive Gaussian free field
&
\(s<1\)
&
\(\displaystyle 1-\frac d2\le s\)
&
\(\displaystyle
 d\Phi_t
 =-A^{1-s}\Phi_t\,dt
 +\sqrt2\,A^{-s/2}\,dW_t
\)
\\
\addlinespace[3pt]

\(d\)
&
\(0\)
&
A positive elliptic operator on $d$-dimensional compact manifold~$M$
&
\(\Phi=\xi\)
&
Spatial white noise
&
\(s<0\)
&
\(\displaystyle -\frac d2\le s\)
&
\(\displaystyle
 d\Phi_t
 =-A^{-s}\Phi_t\,dt
 +\sqrt2\,A^{-s/2}\,dW_t
\)
\\
\addlinespace[3pt]

\(d\)
&
\(2\)
&
\(A=I-\Delta_M\) on a compact
\(d\)-dimensional manifold \(M\)
&
\(\Phi=A^{-1}\xi\)
&
Membrane / bi-Laplacian field
&
\(s<2\)
&
\(\displaystyle 2-\frac d2\le s\)
&
\(\displaystyle
 d\Phi_t
 =-A^{2-s}\Phi_t\,dt
 +\sqrt2\,A^{-s/2}\,dW_t
\)
\\

\bottomrule
\end{tabularx}
\end{table}
\end{landscape}

\bibliographystyle{alpha} 
\bibliography{references} 
\end{document}